\documentclass{amsart}[12pt]

\usepackage{amsmath,amssymb,amscd,amsthm,indentfirst,enumerate}
\usepackage{amsfonts}
\usepackage[mathscr]{eucal}
\usepackage{color}
\usepackage{stmaryrd}
\usepackage{hyperref}

\usepackage[normalem]{ulem}

\swapnumbers
\newtheorem{prop}[equation]{Proposition}
\newtheorem{thm}[equation]{Theorem}

\newtheorem{lem}[equation]{Lemma}

\theoremstyle{definition}
\newtheorem{defn}[equation]{Definition}

\newtheorem{remark}[equation]{Remark}

\newtheorem{example}[equation]{Example}

\def\Mor{\mathrm{Mor}}

\def\Aut{\mathrm{Aut}}
\def\Obj{\mathrm{Obj}}

\def\typ{{\operatorname{typ}}}

\def\vp{\varphi}
\def\al{\alpha}
\def\be{\beta}
\def\ga{\gamma}

\def\La{\Lambda}

\def\de{\delta}

\def\si{\sigma}
\def\Ga{\Gamma}

\def\ze{\zeta}

\def\B{\ensuremath{\mathcal{B}}}
\def\C{\ensuremath{\mathcal{C}}}
\def\D{\ensuremath{\mathcal{D}}}
\def\E{\ensuremath{\mathcal{E}}}
\def\F{\ensuremath{\mathcal{F}}}
\def\H{\ensuremath{\mathcal{H}}}

\def\G{\ensuremath{\mathcal{G}}}
\def\L{\ensuremath{\mathcal{L}}}
\def\M{\ensuremath{\mathcal{M}}}

\def\O{\ensuremath{\mathcal{O}}}
\def\P{\ensuremath{\mathcal{P}}}
\def\Q{\ensuremath{\mathcal{Q}}}
\def\R{\ensuremath{\mathcal{R}}}

\def\Z{\ensuremath{\mathcal{Z}}}

\def\CC{\ensuremath{\mathbb{C}}}

\def\QQ{\ensuremath{\mathbb{Q}}}
\def\ZZ{\ensuremath{\mathbb{Z}}}

\def\Out{\mathrm{Out}}
\def\Aut{\mathrm{Aut}}
\def\Hom{\mathrm{Hom}}
\def\im{\mathrm{im}}
\def\Ker{\mathrm{Ker}}
\def\Coker{\mathrm{Coker}}
\def\incl{\text{incl}}
\def\id{\textrm{id}}
\def\Id{\textrm{Id}}

\def\op{\ensuremath{\mathrm{op}}}

\def\sylp{\operatorname{Syl}_p}
\def\geom{{\operatorname{geom}}}

\def\Ab{\mathbf{Ab}}

\def\Gps{\mathbf{Gps}}

\def\Inn{\operatorname{Inn}}

\def\Sym{\text{Sym}}
\def\Rep{\operatorname{Rep}}
\def\GL{\operatorname{GL}}
\def\SU{\operatorname{SU}}
\def\PSU{\operatorname{PSU}}

\def\pruffer{\ZZ/p^\infty}

\def\defeq{\overset{\operatorname{def}}{=}}

\def\Ad{\operatorname{Ad}}
\def\Lred{\L_{/0}}
\def\SFL{(S,\F,\L)}
\def\0{_{/0}}

\def\fus{{\operatorname{fus}}}

\def\res{\operatorname{Res}}
\def\Tr{{\operatorname{Tr}}}

\DeclareMathAlphabet\EuR{U}{eur}{m}{n}
\SetMathAlphabet\EuR{bold}{U}{eur}{b}{n}

\def\SpAd{{\operatorname{SpAd}}}
\def\rk{\operatorname{rk}}
\def\diag{{\operatorname{diag}}}

\newcommand{\widebar}[1]{\overline{#1}}

\newcommand{\modl}[1]{\ensuremath{#1}\text{-}\mathbf{mod}}

\def\Ext{\operatorname{Ext}}
\def\Image{\operatorname{Im}}

\def\nsg{\trianglelefteq}
\def\OutAd{\operatorname{OutAd}}
\def\proj{\operatorname{proj}}

\def\zpfree{$\ZZ/p^\infty$-free}

\def\gps{\mathbf{Gps}}

\newcommand{\wh}[1]{\widehat{#1}}

\newcommand{\xto}[1]{\xrightarrow{#1}}

\newcommand{\pcomp}[1]{{#1}^\wedge_p}

\newcommand{\higherlim}[2]{\displaystyle\setbox1=\hbox{\rm lim}
	\setbox2=\hbox to \wd1{\leftarrowfill} \ht2=0pt \dp2=-1pt
	\setbox3=\hbox{$\scriptstyle{#1}$}
	\def\test{#1}\ifx\test\empty
	\mathop{\mathop{\vtop{\baselineskip=5pt\box1\box2}}}\nolimits^{#2}
	\else
	\ifdim\wd1<\wd3
	\mathop{\hphantom{^{#2}}\vtop{\baselineskip=5pt\box1\box2}^{#2}}_{#1}
	\else
	\mathop{\mathop{\vtop{\baselineskip=5pt\box1\box2}}_{#1}}%
	\nolimits^{#2}
	\fi\fi}

\newcommand{\lrb}[1]{\llbracket #1\rrbracket}
\newcounter{saveequation}

\input xy
\xyoption{all}

\date{\today}
\title{Groups of unstable Adams operations\\ on $p$-local compact groups}

\author{Ran Levi}
\address{Institute of Mathematics, University of Aberdeen,
Fraser Noble Building 138, Aberdeen AB24 3UE, U.K.}
\email{r.levi@abdn.ac.uk}
\author{Assaf Libman}
\address{Institute of Mathematics, University of Aberdeen,
Fraser Noble Building 136, Aberdeen AB24 3UE, U.K.}
\email{a.libman@abdn.ac.uk}


\begin{document}
\setcounter{section}{0}
\pagestyle{plain}

%\numberwithin{equation}{prop}
\numberwithin{equation}{section}
\renewcommand{\theequation}{\thesection.\arabic{equation}}
%\alph{thmain}
\renewcommand{\thethmain}{\Alph{thmain}}

\renewcommand{\theenumi}{(\arabic{enumi})}
\renewcommand{\labelenumi}{(\arabic{enumi})}

\begin{abstract}
A $p$-local compact group is an algebraic object modelled on the homotopy theory associated with $p$-completed classifying spaces of compact Lie groups and p-compact groups. In particular $p$-local compact groups give a unified framework in which one may study $p$-completed classifying spaces from an algebraic and homotopy theoretic point of view. Like connected compact Lie groups and p-compact groups, $p$-local compact groups admit unstable Adams operations - self equivalences that are characterised by their cohomological effect. Unstable Adams operations on $p$-local compact groups were constructed in a previous paper by F. Junod and the authors. In the current paper we study groups of unstable operations from a geometric and algebraic point of view. We give a precise description of the relationship between algebraic and geometric operations, and show that under some conditions unstable Adams operations are determined by their degree. 
We also examine a  particularly well behaved subgroup of unstable Adams operations. 
\end{abstract}

\maketitle

%%%%%%%%%%%%%%%%%%%%%%%%%%%%%%%%%%%
%%%%%%%%%%%%%%%%%%%%%%%%%%%%%%%%%%%
%%%%%%%%%%%%%%%%%%%%%%%%%%%%%%%%%%%
%%%%%%%%%%%%%%%%%%%%%%%%%%%%%%%%%%%
%%%%%%%%%%%%%%%%%%%%%%%%%%%%%%%%%%%
%%%%%%%%%%%%%%%%%%%%%%%%%%%%%%%%%%%
%%%%%%%%%%%%%%%%%%%%%%%%%%%%%%%%%%%
%%%%%%%%%%%%%%%%%%%%%%%%%%%%%%%%%%%
%%%%%%%%%%%%%%%%%%%%%%%%%%%%%%%%%%%

%%%%%%%%
%%%%%%%%
%%%%%%%%
%%%%%%%%
%%%%%%%%
%%%%%%%%
%%%%%%%%
%%%%%%%%

\section{Introduction}

Let $G$ be a  compact connected Lie group.
An unstable Adams operation of degree $k \geq 1$ on $BG$ is a self-map $f \colon BG \to BG$ which induces multiplication by $k^i$ on $H^{2i}(BG;\QQ)$ for every $i>0$.
In \cite{JMO1} unstable Adams operations for compact connected simple Lie groups $G$ were classified.
The analysis is centred around studying suitable self-maps of the $p$-completion $\pcomp{BG}$.
Later on, after $p$-compact groups were introduced by Dwyer and Wilkerson, their classification by Andersen-Grodal \cite{AG09} and Andersen-Grodal-M\o ller-Viruel \cite{AGMV} relied on studying a suitable notion of unstable Adams operation of $p$-compact groups.

Unstable Adams operations on $p$-local compact groups were defined in \cite{JLL}, where two related definitions are given. The first relies only on the algebraic structure of a $p$-local compact group $\G$. We shall refer to such operations as ``algebraic". The set of all such operations forms a group under composition, that we denote here by $\Ad(\G)$.  The second definition is more closely related to the way unstable Adams operations are defined for compact Lie and $p$-compact groups, and we shall refer to them as ``geometric". These form a subgroup of the group of all self homotopy equivalences of $B\G$, that we will denote by $\Ad^\geom(\G)$. 
The aim of this paper is  to study these groups. Detailed definitions will be given in Sections \ref{sec_geom_Ad} and \ref{Alg_Ad}.

To state our results some concepts are needed. More details will be given in later sections, including a brief review of $p$-local compact groups. Let $G$ be a discrete group.
We say that $G$ is \emph{\zpfree} if it contains no subgroup isomorphic to $\ZZ/p^\infty$. We say that $G$ has a \emph{Sylow $p$-subgroup} if it contains a subgroup $P$ isomorphic to a discrete $p$-toral group, such that any discrete $p$-toral subgroup of $G$ is conjugate to a subgroup of $P$. 
A \emph{normal Adams automorphism of degree $\zeta\in\ZZ_p$} of a discrete $p$-toral group $S$ is an automorphism which restricts to the $\zeta$-power map on the maximal torus $T\nsg S$, and induces the identity on $S/T$. 
The groups of these automorphisms of $S$ is denoted $\Ad(S)$.
If $S$ is not a split extension of its maximal torus, then the degree of such an automorphism must be a $p$-adic unit by \cite[Lemma 2.2(i)]{JLL}. 

If $\F$ is a fusion system over $S$, then the set of normal Adams automorphisms of $S$ that are $\F$-fusion preserving form the group $\Ad_{\fus}(S)$.
It contains $\Ad(S) \cap \Aut_\F(S)$ as a normal subgroup with quotient denoted $\OutAd_{\fus}(S)$.

A $p$-local compact group $\G = \SFL$ (or its underlying fusion system) is said to be \emph{weakly connected} if the maximal torus $T\le S$ is self-centralising in $S$ (and hence $\F$-centric).

%
% Proposition {prop_adgeom_outadfus}
%
\begin{prop}\label{prop_adgeom_outadfus}
Let $\G=\SFL$ be a $p$-local compact group.
Then $\Ad^{\geom}(\G)$ is \zpfree~ and contains a finite normal Sylow $p$-subgroup.
Furthermore, there is a short exact sequence
\[
1\to \higherlim{\mathcal{O}^c(\F)}{1}\mathcal{Z} \to \Ad^\geom(\G)\to \OutAd_{\fus}(S)\to 1.
\]
In particular, if $p$ is odd then $\Ad^{\geom}(\G) \cong \OutAd_{\fus}(S)$.
Finally, $\OutAd_{\fus}(S)$ is solvable of class $\leq 2$, and so $\Ad^{\geom}(\G)$ is solvable of class at most 3.
\end{prop}

Next we consider the group of algebraic operations and its relationship with the group of geometric operations. 
Geometric realisation gives rise to a homomorphism from the group of algebraic Adams operations to the group of geometric Adams operations.
\[
\ga \colon \Ad(\G) \xto{ \ (\Psi,\psi) \mapsto \pcomp{|\Psi|} \ } \Ad^{\geom}(\G).
\]
Similar to the situation with groups, given $t \in T$ there is an automorphism of $\L$ obtained by ``conjugating'' by $\widehat{t} \in \Aut_\L(S)$.
These automorphisms are called \emph{inner automorphisms} and they form the group $\Inn_T(\G) \leq \Ad(\G)$ that  is contained in the kernel of $\gamma$.
Let $D(\F)$ be the subgroup of $\ZZ_p^\times$ consisting of the degrees of all $\vp \in \Aut_\F(S) \cap \Ad(S)$.
In \cite[Proposition 3.5]{JLL} we showed that $\gamma$ is an epimorphism. Our next theorem gives more refined information.

%
% Theorem: geometric vs. algebraic uAos
%
\begin{thm}\label{thm_geom_vs_alg_uao}
Let $\G=\SFL$ be a weakly connected $p$-local compact group.
Let $T$ denote the maximal torus of $S$.
Then
\begin{enumerate}[\rm (i)]
\item
$\Ad(\G)$ has a normal maximal discrete $p$-torus denoted $\Ad(\G)_0$.
It contains a Sylow $p$-subgroup which is normal if $p=2$.

\item
$\Ad(\G)_0$ is contained in the kernel of $\gamma$ (see \eqref{def_gamma}) and there is a short exact sequence
\[
1 \to D(\F) \to \Ad(\G)/\Ad(\G)_0 \xto{\ \bar{\gamma} \ } \Ad^{\geom}(\G) \to 1
\]

\item
$\Ad(\G)_0 =\Inn_T(\G) \cong T/Z(\F)$ where $Z(\F)$ is the centre of $\F$.
\end{enumerate}
\end{thm}

We now turn our attention to the question to what extent is an unstable operation determined by its degree.  
If $G$ is a compact connected Lie group, then by \cite[Theorem 1]{JMO1} for any $k \geq 1$ there is up to homotopy at most one unstable Adams operation of degree $k$.
The situation for  $p$-local compact groups is not quite so simple. 
It is not hard to see that a fusion preserving Adams automorphism induced by a geometric operation is  only determined modulo $D(\F)$, and so its  degree is, at best, only determined modulo $D(\F)$.
Also, all elements of $\Ad(\G)_0 = \Inn_T(\G)$ are (distinct) algebraic unstable Adams operations of degree $1$.
Thus, the best we can hope for is to address the question of the injectivity of the homomorphisms
\[
\Ad^\geom(B\G) \xto{\ \deg \ } \ZZ_p^\times/D(\F) \qquad
\text{ and } \qquad
\Ad(\G)/\Ad(\G)_0 \xto{ \ \deg \ } \ZZ_p^\times.
\]
The next proposition gives criteria for injectivity of these maps.

\begin{prop}\label{prop_degree_determines_operation}
Suppose that $\G=(S,\F,\L)$ is weakly connected and let $W=\Aut_\F(T)$ be its Weyl group.
If $H^1(W,T)=0$ then the degree map $\deg\colon\OutAd_\fus(S) \xto{} \ZZ_p^\times/D(\F)$ is injective, and there are exact sequences
\begin{eqnarray*}
(1) && 1 \to \higherlim{\O(\F^c)}{1} \Z \to \Ad^{\geom}(\G) \xto{\quad \deg \quad} \ZZ_p^\times/D(\F) \\
(2) && 1 \to \higherlim{\O(\F^c)}{1} \Z \to \Ad(\G)/\Ad(\G)_0 \xto{\quad \deg \quad} \ZZ_p^\times 
\end{eqnarray*}
If in addition $p\neq 2$, then the degree maps in (1) and (2) are injective. 
\end{prop}

Proposition \ref{prop_degree_determines_operation} motivates the search for conditions under which $H^1(W,T) =0$. 

\begin{prop}\label{prop_H1W_vanish_intro}
Let $\F$ be a weakly connected saturated fusion system and let $W=\Aut_\F(T) \leq \GL_r(\ZZ_p)$ be its Weyl group.
Assume that either one of the following conditions holds:
\begin{enumerate}[\rm (i)]
\item
$p$ is odd and $D(\F) \neq 1$, or

\item $p=2$, $D(\F)\neq 1$ and $H^1(W/D(\F), T^{D(\F)})=0$, or

\item
$p$ is odd and the Weyl group $W$ is a pseudo-reflection group.
\end{enumerate}
Then $H^1(W,T)=0$.
\end{prop}
Condition (iii) in the proposition is in fact an easy reinterpretation of work of Andersen, \cite[Theorem 3.3]{An99}.

Next, we study the possible degrees of unstable Adams operation on a given $p$-local compact group $\G$.
We will define a particular type of operations, which we call \emph{special  relative to a collection  of subgroups $\R$} (see Definition \ref{def_special_uao}) for which we will obtain a complete answer. 
These form a subgroup of $\Ad(\G)$ that we denote by $\SpAd(\G;\R)$. 
We will show in Appendix \ref{app_jll_special} that all unstable Adams operations constructed in \cite{JLL} are special.

% \DEL{\NEWr{I removed a paragaph from here. See below}}

To study special Adams operations relative to a collection $\R$, we consider the full subcategory $\L^\R$ of the linking system $\L$ whose objects are the elements of $\R$, as an extension, in the appropriate sense, of a category $\Lred^\R$ by a naturally defined functor $\Phi \colon \Lred^\R \to \modl{\ZZ_p}$. Utilising the theory of extension of categories (a variation of which was done by Hoff in \cite{Hoff}),  we show that the category $\L^\R$ corresponds, up to isomorphism of extensions, to a well defined  element $[\L^\R]\in H^2(\Lred^\R;\Phi)$. This also allows us to study self equivalences of $\L^R$ that are compatible with  its structure as an extension of $\Lred^\R$ via their effect on the extension class $[\L^R]$.  Our analysis corresponds nicely to the way in which  a group extension $G$ of a quotient group $K$ by an abelian group $A$ corresponds, up to isomorphism of extensions, to a unique class in $H^2(K;A)$, and maps of extensions are controlled by their effect on the corresponding extension classes. 

For a $p$-local compact group $\G = \SFL$, let $\H^\bullet(\F^c)$ denote the distinguished family of subgroups defined in  \cite[Sec. 3]{BLO3}. 
For every $k \geq 1$ let $\Gamma_k(p)\le \ZZ_p^\times$ denote the subgroup of all elements $\zeta$ such that $\zeta\equiv 1\mod p^k$, and set $\Gamma_0(p) = \ZZ_p^\times$.
Our next theorem that gives existence and uniqueness criteria for special unstable Adams operations, as well as information about the group of special operations and its relationship with the group of all (algebraic) operations.

\begin{thm}\label{thm_spad_in_ad}
Let $\G=\SFL$ be a weakly connected $p$-local compact group.
Suppose that a collection $\R \subseteq \F^c$ has finitely many $\F$-conjugacy classes.
Then the following statements hold.
\begin{enumerate}[\rm (i)]
\item If $(\Psi,\psi) \in \SpAd(\G;\R)$ is of degree $\zeta$ then $\zeta \cdot [\L^\R]=[\L^\R]$ in $H^2(\Lred^\R,\Phi)$. If in addition $\R \supseteq \H^\bullet(\F^c)$, see \cite[Sec. 3]{BLO3}, then there is an exact sequence 
\[
H^1(\Lred^\R,\Phi) \to \SpAd(\G;\R)/\Inn_T(\G) \xto{\deg} \Gamma_m(p) \to 1
\]
where $p^m$ is the order of $[\L^\R]$ in $H^2(\Lred^\R,\Phi)$. \label{thm_spad_in_ad_ii}
In particular, $\SpAd(\G;\R)$ has a normal Sylow $p$-subgroup with $\Inn_T(\G)$ as its maximal discrete $p$-torus.

\item $\SpAd(\G;\R)$ is a normal subgroup of $\Ad(\G)$ of finite index.
The quotient group is solvable of class at most $3$. \label{thm_spad_in_ad_i}
\end{enumerate}
\end{thm}

Let $\deg(\Ad(\G))\le \ZZ_p^\times$ denote the image of the degree map. Let $p^n$ be the order of the extension class of $S$ in $H^2(S/T,T)$. By \cite[Proposition 2.8(i)]{JLL}, $\deg(\Ad(\G)) \leq \Ga_n(p)$. On the other hand,  Theorem \ref{thm_spad_in_ad}  shows that  the restriction of the degree map to $\SpAd(\G;\R))$ is onto $\Ga_m(p)$, where $p^m$ is the order of the extension class of $\L^\R$ in $H^2(\Lred^\R,\Phi)$ and $\R=\H^\bullet(\F^c)$. Thus $\Ga_m(p) \leq \deg(\Ad(\G))$. Putting these observation together we have 
\[
\Ga_m(p) \leq \deg(\Ad(\G)) \leq \Ga_n(p).
\]
This gives a more precise result than the one obtained in \cite[Theorem A]{JLL}.

With Theorem \ref{thm_spad_in_ad} in mind, one could hope na\"ivly that every (geometric) unstable Adams operation is homotopic to one induced by a special operation. Our next result shows that this is not the case. 
In fact, 
the image of $\SpAd(\G;\R)$ in $\Ad^{\geom}(\G)$ is in general a proper subgroup.
Examples are given by connected compact Lie groups as is shown by the next result.

\begin{prop}\label{prop_in_psu2p_not_all_special_intro}
Let $p \geq 3$ be a prime and set $G=\PSU(2p)$.
Let $S \leq G$ be a maximal discrete $p$-toral group and let $\G=(S,\F_S(G),\L_S^c(G))$ be the associated $p$-local compact group.
Let $\R$ denote the collection of all centric radical subgroups of $S$.
Then $\SpAd(\G;\R) \lneq \Ad(\G)$.
In fact, the following composition is not surjective
\[
\SpAd(\G;\R) \xto{ \incl } \Ad(\G) \xto{(\Psi,\psi) \mapsto \pcomp{|\Psi|}} \Ad^\geom(\G).
\]
\end{prop}

We only have a crude bound on  the index of $\SpAd(\G;\R)$ in $\Ad(\G)$ (see Remark \ref{rem:finite index of spad}), and we expect that only very rarely all geometric unstable Adams operations are induced by special ones.

Finally, we show that unstable Adams operations on weakly connected $p$-local compact groups are ``stably" determined by their degree in the sense that there exists some $n>0$, which only depends on $\G$, such that the $n$-th powers of any two geometric unstable Adams operations of the same degree are homotopic.

\begin{prop}\label{stable}
Let $\G=\SFL$ be a weakly connected $p$-local compact group.
Let $f_1, f_2 \in \Ad^{\geom}(\G)$ be such that $\deg(f_1)=\deg(f_2)$ in $\ZZ_p^\times/\D(\F)$.
Then $(f_1)^{mk}=(f_2)^{mk}$ where $1 \leq k \leq |\higherlim{\O(\F^c)}{1}\mathcal{Z}|$ and where $1 \leq m \leq |H^1(S/T,T)|$.
In fact, $m$ is a divisor of $(p-1)p^r$ for some $r \geq 0$.
\end{prop}

We end with a glossary of notation that will be used throughout the paper.
\begin{itemize}
\item
$\Ad(S) \leq \Aut(S)$ is the group of all the \emph{normal} Adams automorphisms of $S$ (Definition \ref{def_adams_aut_of_S}).

\item
If $\F$ is a fusion system over $S$, let $\Ad_{\fus}(S)$ denote the group of fusion preserving Adams automorphisms of $S$, namely 
$\Ad_{\fus}(S) = \Aut_{\fus}(S) \cap \Ad(S)$.  

\item
$\OutAd_{\fus}(S)$ is the image of $\Ad_{\fus}(S)$ in $\Out_{\fus}(S)=\Aut_{\fus}(S)/\Aut_\F(S)$ 
(Definition \ref{def_adfuss}).

\item
$\Ad^{\geom}(\G) \leq \Out(B\G)$ is the group of the \emph{homotopy classes} of  \emph{geometric}  unstable Adams operations on $B\G$ (Definition \ref{def_geometric_Adams_operation}). 

\item
$\Ad(\G)$ is the group of \emph{algebraic} unstable Adams operation of $\G$ (Definition \ref{def_algebraic_unstable_adams_operation}).

\item
$\Inn_T(\G) \leq \Ad(\G)$ for the subgroup of the Adams operations induced by conjugation in $\L$ by some $\hat{t} \in \Aut_\L(S)$ where $t \in T$ (Definition \ref{def_inner}).
\end{itemize}

%@@@@@@
The paper is organised as follows. In Section \ref{sec_p-local_compact_groups} we recall some definitions and useful facts on $p$-local compact groups. In Section \ref{sec_geom_Ad} we discuss geometric unstable operations and prove  Proposition \ref{prop_adgeom_outadfus} (Proposition \ref{prop_adgeom_outadfus_in_sec}).
Section \ref{Alg_Ad} is dedicated to the proof of Theorem \ref{thm_geom_vs_alg_uao} (Theorem \ref{thm_geom_vs_alg_uao_in_sec}), and in Section \ref{sec_degree_uniqueness} we  prove Propositions \ref{prop_degree_determines_operation}, \ref{prop_H1W_vanish_intro} and \ref{stable} (Propositions \ref{prop_degree_determines_operation_in_sec}, \ref{prop_H1W_vanish} and \ref{sec_stable}, respectively). In Section \ref{sec_extensions} we recall some basic theory of extensions of categories and introduce the category $\Lred$. Sections \ref{specials} is dedicated to  special unstable Adams operations, and the proof of Theorem \ref{thm_spad_in_ad} (Theorem \ref{thm_spad_in_ad_in_sec}). Finally in Section \ref{sec_not_all_special} we prove Proposition \ref{prop_in_psu2p_not_all_special_intro} (Proposition \ref{prop_in_psu2p_not_all_special}). We end with two appendices. Appendix \ref{appendixA} contains proofs of statements from Section \ref{sec_extensions}, while Appendix \ref{app_jll_special} is dedicated to  the observation that the operations constructed in \cite{JLL} are all special.

The authors are grateful to the referee for a thorough reading and useful comments.

%%%%%%%%%%%%%%%%%%%%%%%%%%%%%%
%%%%%%%%%%%%%%%%%%%%%%%%%%%%%%
%%%%%%%%%%%%%%%%%%%%%%%%%%%%%%
%%%%%%%%%%%%%%%%%%%%%%%%%%%%%%
%%%%%%%%%%%%%%%%%%%%%%%%%%%%%%
%%%%%%%%%%%%%%%%%%%%%%%%%%%%%%
%%%%%%%%%%%%%%%%%%%%%%%%%%%%%%
%%%%%%%%%%%%%%%%%%%%%%%%%%%%%%
%%%%%%%%%%%%%%%%%%%%%%%%%%%%%%
%%%%%%%%%%%%%%%%%%%%%%%%%%%%%%
%%%%%%%%%%%%%%%%%%%%%%%%%%%%%%
%%%%%%%%%%%%%%%%%%%%%%%%%%%%%%
\section{Recollections of $p$-local compact groups}
\label{sec_p-local_compact_groups}

This section briefly introduces $p$-local compact groups and collect some results about them that we will use.
The reference is \cite{BLO3}.

Fix a prime $p$.
Recall that $\ZZ/p^\infty \defeq \cup_{n \geq 1} \ZZ/p^n$. %$\ZZ/p^\infty \defeq \ZZ\left[\frac{1}{p}\right]/\ZZ$.
This is a divisible group.
A \emph{discrete $p$-torus} of rank $n$ is a group $T$ isomorphic to $\bigoplus^n \ZZ/p^\infty$.
A group $P$ is called \emph{discrete $p$-toral} if it contains a discrete $p$-torus $T$ as a normal subgroup and $P/T$ is a finite $p$-group.
In this case $T$ is characteristic in $P$ and we write $T=P_0$ and call it the \emph{identity component} of $P$.
Every sub-quotient of $P$ is a discrete $p$-toral group by \cite[Lemma 1.3]{BLO3}.
Also, an extension of discrete $p$-toral groups is a discrete $p$-toral group.
The \emph{order} of a discrete $p$-toral group $P$ is the pair $(\rk(P_0), |P/P_0|)$.
These pairs are ordered lexicographically.

A fusion system $\F$ over a discrete $p$-toral group $S$ is a category whose objects are all the subgroups of $S$.
Morphisms between $P,Q \leq S$ are group monomorphisms and $\Hom_\F(P,Q)$ always contains $\Hom_S(P,Q)$, namely the homomorphisms $P \to Q$ induced by conjugation by elements of $S$.
In addition every morphism in $\F$ can be factored as an isomorphism in $\F$ followed by an inclusion homomorphism.
See \cite[Definition 2.1]{BLO3}

We say that $P,Q \leq S$ are $\F$-conjugate if they  are isomorphic as objects of $\F$.
A subgroup $P \leq S$ is called fully centralised (resp. fully normalised) if for every $P' \leq S$ which is $\F$-conjugate to $P$, the order of $C_S(P')$ (resp. $N_S(P')$) is at most the order of $C_S(P)$ (resp. $N_S(P)$).

A fusion system $\F$ over $S$ is called \emph{saturated} if 
\begin{enumerate}[(I)]
\item Every fully normalised $P \leq S$ is fully centralised, $\Out_\F(P)\defeq\Aut_\F(P)/\Inn(P)$ is finite and contains $\Out_S(P)$ as a Sylow $p$-subgroup.

\item
If $\vp \in \Hom_\F(P,S)$ and if $\vp(P)$ is fully centralised, then $\vp$ extends to $\psi \in \Hom_\F(N_\vp,S)$ where
\[
N_\vp = \{g \in N_S(P) : \vp \circ c_g \circ \vp^{-1} \in \Aut_S(\vp(P))\}.
\]

\item
If $P_1 \leq P_2 \leq \dots$ is an increasing sequence of subgroups of $S$ and $\vp \in \Hom(\bigcup_n P_n,S)$ is a homomorphism such that $\vp|_{P_n} \in \Hom_\F(P_n,S)$ then $\vp \in \Hom_\F(\bigcup_n P_n,S)$.
\end{enumerate}

A subgroup $P$ of $S$ is called \emph{$\F$-centric} if $C_S(P')=Z(P')$ for every $P'$ which is $\F$-conjugate to $P$.
It is called \emph{$\F$-radical} if $O_p(\Out_\F(P))=1$ where $O_p(K)$ denotes the  largest normal $p$-subgroup of a finite group $K$.

If $P \leq S$ is $\F$-centric then it must contain $Z(S)$.
The centre of $\F$ is defined by:
\[
Z(\F) = \{ x \in Z(S) : \vp(x)=x \text{ for every } P \in \F^c \text{ and every } \vp \in \Hom_\F(P,S)\}.
\]
In light of Alperin's fusion theorem \cite[Theorem 3.6]{BLO3} this is the same as the subgroup of $x \in Z(S)$ such that $\vp(x)=x$ for any $P \leq S$ which contains $x$ and every $\vp \in \Hom_\F(P,S)$.

A \emph{linking system} $\L$ associated to a saturated fusion system $\F$ over $S$ is a small category whose objects are the $\F$-centric subgroups of $S$.
It is equipped with a surjective functor $\pi \colon \L \to \F^c$ which is the identity on objects and with monomorphisms of groups $\de_P\colon P \to \Aut_\L(P)$, one for each $P \in \F^c$ such that the following hold.
\begin{itemize}
\item[(A)]
For each $P,Q \in \L$ the group $Z(P)$ acts freely on $\L(P,Q)$ via $\de_P\colon P \to \Aut_\L(P)$ and pre-composition of morphisms, and $\pi \colon \L(P,Q) \to \Hom_\F(P,Q)$ is the quotient map by this action.

\item[(B)]
For any $P \in \L$ and $g \in P$, $\pi(\de_P(g))=c_g \in \Aut_\F(P)$.

\item[(C)]
For any $\vp \in \L(P,Q)$ and $g \in P$, the following diagram commutes in $\L$
\[
\xymatrix{
P \ar[d]_{\de_P(g)} \ar[r]^{\vp} &
Q \ar[d]^{\de_Q(\pi(\vp)(g))}
\\
P \ar[r]_{\vp} & Q.
}
\]
\end{itemize}
The morphisms $\de_P(g)$ will be denoted $\widehat{g}$.

A \emph{$p$-local compact group} is a triple $\G=\SFL$ where $\F$ is a saturated fusion system over $S$ and $\L$ is an associated centric linking system.
The \emph{classifying space} of $\G$ denoted $B\G$ is by definition $\pcomp{|\L|}$.

The orbit category of $\F$ denoted $\O(\F)$ has the same objects as $\F$ and morphism sets  $\Hom_{\O(\F)}(P,Q)=\Hom_\F(P,Q)/\Inn(Q)$.
The full subcategory on the $\F$-centric subgroups is denoted $\O(\F^c)$.
In \cite[proof of Theorem 7.1]{BLO3} the functor $\Z \colon \O(\F^c)^\op \to \modl{\ZZ_{(p)}}$ was  defined:
\[
\Z \colon P \mapsto Z(P)=C_S(P).
\]
This functor is a key tool in the study of $B\G$ and its self equivalences.

We end this section with a few remarks and observations. 

\begin{remark}\label{R:extend delta}
 It is shown in \cite[Proposition 1.5]{JLL} that the monomorphisms $\de_P$ can be extended to monomorphisms $N_S(P) \xto{g \mapsto \widehat{g}} \Aut_\L(P)$ which satisfy (B), and more generally to injective functions $\de_{P,Q} \colon N_S(P,Q) \xto{g \mapsto \widehat{g}} \L(P,Q)$ which satisfy (B).
Moreover  $\widehat{g} \circ \widehat{h}=\widehat{gh}$ whenever $h \in N_S(P,Q)$ and $g \in N_S(Q,R)$.
Thus, the identity element $e \in S$ give rise to morphisms $\iota_P^Q \in \L(P,Q)$ for any $P \leq Q$ where $\iota_P^Q=\widehat{e}$.
\end{remark}

\begin{remark}\label{R:mono-epi}
By \cite[Corollary 1.8]{JLL} the category $\L$ has the property that every morphism in $\L$ is both a monomorphism and an epimorphism (but in general not an isomorphism).
This allows us, by \cite[Lemma 1.7(i)]{JLL}, to define ``restrictions'': if $\vp \in \L(P,Q)$ and $P' \leq P$ and $Q' \leq Q$ are $\F$-centric subgroups such that $\pi(\vp)(P') \leq Q'$ then there exists a unique morphism $\psi \in \L(P',Q')$ such that $\vp  \circ \iota_{P'}^P = \iota_{Q'}^Q \circ \psi$.
We write $\vp|_{P'}^{Q'}$ for this unique morphism and call it the restriction of $\vp$.
\end{remark}

\begin{remark}\label{R:bullet collection}
In \cite[Section 3]{BLO3} a collection of subgroups of $S$ denoted $\H^\bullet(\F)$ was constructed.
We will not recall the precise details of its construction here.
The full subcategory of $\F$ on this set of objects is denoted $\F^\bullet$.
There is a functor $\F \to \F^\bullet$ where $P \leq P^\bullet$ for any $P \leq S$.
This functor is left adjoint to the inclusion $\F^\bullet \subseteq \F$.
In addition $\H^\bullet(\F)$ contains the collection $\F^{cr}$ of all $P \leq S$ which are both $\F$-centric and $\F$-radical. The intersection $\H^\bullet(\F)\cap \F^c$ will be denoted $\H^\bullet(\F^c)$.
The functor $P \mapsto P^\bullet$ ``lifts'' to a functor $\L \to \L^\bullet$ such that $\widehat{g}^\bullet = \widehat{g}$ for any $g \in N_S(P,Q)$.
See \cite[Proposition 1.12]{JLL}. 
\end{remark}

\begin{prop}\label{P:finite limits Z}
For every $k \geq 1$ the groups $\higherlim{\O(\F^c)}{k}\mathcal{Z} $ are finite $p$-groups.
\end{prop}

\begin{proof}
Let $\Z_0$ be the sub-functor of $\Z$ such that $\Z_0(P)=Z(P)_0$ (the identity component of $Z(P)$).
By \cite[Proposition 5.8]{BLO3} and a long exact sequence argument, there is an isomorphism $\omega_1 \colon \higherlim{\O(\F^c)}{k}\mathcal{Z} \to \higherlim{\O(\F^c)}{k}\mathcal{Z}/Z_0$.
By \cite[Proposition 5.2]{BLO3} $\higherlim{\O(\F^c)}{*}\mathcal{Z}/Z_0 \cong \higherlim{\O(\F^{c\bullet})}{*}\mathcal{Z}/Z_0$ and $\O(\F^{c\bullet})$ is equivalent to a finite category by \cite[Lemmas 2.5 and 3.2(a)]{BLO3}.
Since $\Z/\Z_0$ has values in finite abelian $p$-groups, it follows that $\higherlim{\O(\F^{c\bullet})}{*}\mathcal{Z}/Z_0$ are finite abelian $p$-groups, hence so are $\higherlim{\O(\F^c)}{k}\mathcal{Z}$ for $k \geq 1$.
\end{proof}

In fact, all the groups in Proposition \ref{P:finite limits Z} vanish except when $p=2$ and $k=1$ by \cite[Theorem A]{LL}.

%%%%%%%%%%%%%%%%%%%%%%%%%%%%%%
%%%%%%%%%%%%%%%%%%%%%%%%%%%%%%
%%%%%%%%%%%%%%%%%%%%%%%%%%%%%%
%%%%%%%%%%%%%%%%%%%%%%%%%%%%%%
%%%%%%%%%%%%%%%%%%%%%%%%%%%%%%
%%%%%%%%%%%%%%%%%%%%%%%%%%%%%%
%%%%%%%%%%%%%%%%%%%%%%%%%%%%%%
%%%%%%%%%%%%%%%%%%%%%%%%%%%%%%
%%%%%%%%%%%%%%%%%%%%%%%%%%%%%%
%%%%%%%%%%%%%%%%%%%%%%%%%%%%%%
%%%%%%%%%%%%%%%%%%%%%%%%%%%%%%
%%%%%%%%%%%%%%%%%%%%%%%%%%%%%%
\section{Geometric unstable Adams operations}\label{sec_geom_Ad}
In this section we recall the  concept of geometric unstable Adams operations and make some basic observations, ending with the proof of Proposition \ref{prop_adgeom_outadfus}.

\subsection{Groups with maximal discrete $p$-tori}

\begin{defn}\label{def_zpfree}
A group $G$ is called \emph{$\ZZ/p^\infty$-free} if it contains no subgroup isomorphic to $\ZZ/p^\infty$, or equivalently if every homomorphism $\ZZ/p^\infty \to G$ is trivial.
We say that $T \leq G$ is a maximal discrete $p$-torus in $G$ if $T \cong \bigoplus^n \ZZ/p^\infty$ and any other subgroup of $G$ isomorphic to a discrete $p$-torus is conjugate to a subgroup of $T$.
We say that $S \leq G$ is a {\em Sylow $p$-subgroup} if it is discrete $p$-toral, and every discrete $p$-toral subgroup of $G$ is conjugate to a subgroup of $S$.
\end{defn}

\begin{lem}\label{lem_extn_zpfree}
Let $1 \to K \to G \to H \to 1$ be an exact sequence of group.
\begin{enumerate}[(i)]
\item\label{lem_extn_zpfree_i}
If $H$ and $K$ are \zpfree~ then $G$ is \zpfree.

\item\label{lem_extn_zpfree_ii}
If $G$ is \zpfree~ and $K$ is finite then $H$ is \zpfree.
\end{enumerate}
\end{lem}
\begin{proof}
(\ref{lem_extn_zpfree_i}).
Suppose $T \leq G$ is a discrete $p$-torus.
Its image in $H$ must be trivial by assumption.
Hence $T \leq K$ which implies $T=1$.

(\ref{lem_extn_zpfree_ii}).
Assume that $H$ is not \zpfree.
Then there exists a sequence $x_1, x_2, x_3, \dots$ of non-identity elements in $H$ such that $x_i=(x_{i+1})^p$ for all $i \geq 1$.
The preimages of these elements in $G$ are cosets $X_1, X_2, X_3,\dots$ of $K$ and hence finite subsets of $G$.
The function (not a homomorphism) $G \xto{x \mapsto x^p} G$ restricts to functions $X_{i+1} \xto{x \mapsto x^p} X_i$ for all $i$ and we obtain a tower
\[
\dots \to X_{i+1} \xto{x \mapsto x^p} X_i \xto{x \mapsto x^p} \dots \xto{x \mapsto x^p} X_2 \xto{x \mapsto x^p} X_1.
\]
Since the sets $X_i$ are finite, $\varprojlim X_i$ is not empty.
An element in $\varprojlim X_i$ is a sequence $g_1,g_2,g_3,\dots$ of non-identity elements of $G$ such that $g_i=g_{i+1}{}^p$.
These elements generate a copy of $\ZZ/p^\infty$ in $G$ which is a contradiction.
\end{proof}

\subsection{Adams automorphisms of discrete $p$-toral groups and fusion systems}

Let $S$ be a discrete $p$-toral group and let $T$ denote its maximal torus.
Recall that $\Aut(T) \cong \GL_r(\ZZ_p)$ where $r=\rk(T)$.
It contains a copy of $\ZZ_p^\times$ in its centre (the diagonal matrices).

\begin{defn}\label{def_adams_aut_of_S}
An \emph{Adams automorphism} of $S$ of degree $\zeta \in \ZZ_p^\times$ is an automorphism $\vp$ of $S$ such that $\vp|_T$ is multiplication by $\zeta$.
We say that $\vp$ is a \emph{normal} Adams automorphism if it induces the identity on the set of components $S/T$.
Set 
\[
\Ad(S) = \{ \text{all \emph{normal} Adams automorphisms of $S$} \}.
\]
This is a normal subgroup of $\Aut(S)$.
Restriction $\Aut(S) \to \Aut(T)$ induces a ``degree homomorphism''
\[
\deg \colon \Ad(S) \to \ZZ_p^\times.
\]
Let $\Ad^1(S)$ or $\Ad^{\{\deg=1\}}(S)$ be the subgroup of the normal Adams automorphisms of degree $1$.
\end{defn}

Let $\Aut_T(S)$ denote the automorphisms of $S$ induced by conjugation by elements of  $T$.
Clearly $\Aut_T(S) \nsg \Ad^1(S)$.

\begin{lem}\label{L:coholology G with T}
Let $G$ be a finite group which acts on a discrete $p$-torus $T$.
Then $H^{k}(G,T)$ is a finite $p$-group for every $k\geq 1$.
\end{lem}

\begin{proof}
Write $M=H^k(G,T)$ where $k \geq 1$.
A transfer argument with the trivial subgroup of $G$ shows that $M$ has finite exponent $\leq |G|$.
Let $P_*$ be a projective resolution of the trivial $\ZZ[G]$-module $\ZZ$ such that every $P_i$ is finitely generated. Then the cochain complex $\Hom_{\ZZ[G]}(P_*, T)$ is  complex of discrete $p$-tori, and hence $M$ is a discrete $p$-toral group.
Since it has finite exponent,  $M$ is a finite $p$-group.
\end{proof}

Recall that if $\F$ is a fusion system over $S$, then $\al \in \Aut(S)$ is called \emph{fusion preserving} if for every $\vp \in \Hom_\F( P, S)$, the composite $\al \circ \vp \circ \al^{-1}\in \Hom_\F(\al(P), S)$. 
The group of fusion preserving automorphisms is denoted $\Aut_{\fus}(S)$.
It clearly contains $\Aut_\F(S)$ as a normal subgroup.
We define (see \cite[Section 7]{BLO3})
%\begin{equation}\label{def outfuss}
\[
\Out_{\fus}(S) = \Aut_{\fus}(S)/\Aut_\F(S).
\]
%\end{equation}

\begin{defn}\label{def_adfuss}
Let $\F$ be a saturated fusion system over $S$.
Set
\[
\Ad_{\fus}(S) = \Ad(S) \cap \Aut_{\fus}(S).
\]
Let $\OutAd_{\fus}(S)$ be the image of $\Ad_{\fus}(S)$ in $\Out_{\fus}(S)$.
\end{defn}

The goal of this subsection is to prove

%
% Cor outadfuss
%
\begin{prop}\label{prop_outadfuss}
Let $\F$ be a saturated fusion system over $S$ and let $T$ be the identity component of $S$.
Then the following statements hold for the group $\OutAd_{\fus}(S)$.
\begin{enumerate}[(i)]
\item
It is \zpfree~ and solvable of class at most $2$.
\label{prop_outadfuss_i}

\item 
It contains a normal Sylow $p$-subgroup which is abelian if either $p$ is odd, or if $p=2$ and the order of the extension class of $S$ in $H^2(S/T;T)$ is at least $4$.
\label{prop_outadfuss_ii}
\end{enumerate}
\end{prop}

\begin{proof}
The torsion elements in $\ZZ_p^\times$ form the group $U_p$ of roots of unity in $\ZZ_p$, and by \cite[Sec. 6.7, Prop. 1,2]{Ro} 
\[
U_p \cong 
\left\{
\begin{array}{ll}
C_2=\{\pm 1 \} & \text{if $p=2$} \\
C_{p-1} & \text{if $p>2$}
\end{array}
\right. 
\]
In particular $U_p$ is finite and therefore $\ZZ_p^\times$ is \zpfree, hence so are all of $\Ga_m(p)$.
In fact, if $p>2$ then $\Ga_k(p)$ contains no finite $p$-subgroup, and if $p=2$ then $\Ga_k(2)$ contains no finite $2$-subgroup if $k \geq 2$.
By \cite[Proposition 2.8]{JLL} there is a short exact sequence
\[
1 \to H^1(S/T;T) \to \Ad(S)/\Aut_T(S) \xto{ \deg } \Ga_m(p) \to 1
\]
where $p^m$ is the order of the extension class of $S$ in $H^2(S/T,T)$.
Also $H^1(S/T;T)$ is a finite abelian $p$-group by Lemma \ref{L:coholology G with T}.
It follows that $\Ad(S)/\Aut_T(S)$ is solvable of class $\leq 2$, and from Lemma \ref{lem_extn_zpfree}(\ref{lem_extn_zpfree_i}) that it is \zpfree.
Let $P$ be the preimage of the normal Sylow $p$-subgroup of $\Ga_m(p)$.
Then $P$ is a normal Sylow $p$-subgroup of $\Ad(S)/\Aut_T(S)$ and it is abelian if either $p$ is odd or if $p=2$ and $m\geq 2$ since in this case $U_p \cap \Ga_m(p)=1$, hence $P \cong H^1(S/T,T)$.

Observe that $\Aut_\F(S) \leq \Aut_\fus(S)$ and that $\Aut_T(S) \leq \Ad(S)$.
Hence $\Ad_\fus(S) \cap \Aut_\F(S) = \Ad(S) \cap \Aut_\F(S)$ and by Definition \ref{def_adfuss} we obtain the short exact sequence
\[
1 \to \frac{\Aut_\F(S) \cap \Ad(S)}{\Aut_T(S)} \to \frac{\Ad_{\fus}(S)}{\Aut_T(S)} \to \OutAd_{\fus}(S) \to 1.
\]
From the results above about $\Ad(S)/\Aut_T(S)$ it follows that $\OutAd_\fus(S)$ is solvable of class at most $2$.
The first group in this exact sequence is finite since it is a subgroup of $\Aut_\F(S)/\Aut_T(S)$ which is finite since $S/T$ is finite and $\Out_\F(S)=\Aut_\F(S)/\Inn(S)$ is finite by \cite[Lemma 2.5]{BLO3}.
Lemma \ref{lem_extn_zpfree}(\ref{lem_extn_zpfree_ii}) now implies that $\OutAd_\fus(S)$ is \zpfree~ and this completes the proof of (\ref{prop_outadfuss_i}).
Recall that $\Ad(S)/\Aut_T(S)$ has a normal Sylow $p$-subgroup.
Let $P$ be the normal Sylow $p$-subgroup of $\Ad_\fus(S)/\Aut_T(S)$ and let $\bar{P}$ be its image in $\OutAd_\fus(S)$.
Then $\bar{P}$ is normal and it is abelian if either $p$ is odd or if $p=2$ and $m\geq 2$ namely if the extension class of $S$ in $H^2(S/T,T)$ has order at least $4$.
It remains to show $\bar{P}$ is a Sylow $p$-subgroup of $\OutAd_\fus(S)$.
If $Q \leq \OutAd_\fus(S)$ is a discrete $p$-toral group, it is finite since $\OutAd_\fus(S)$ is \zpfree~ and its preimage $H$ in $\Ad_\fus(S)/\Aut_T(S)$ is therefore finite.
If $H_p \in \sylp(H)$ then $H_p \leq P$ since $P$ is normal and its image in $\OutAd_\fus(S)$ is $Q$.
Hence $Q \leq \bar{P}$ as needed.
\end{proof}

%
% Subsection
%
\subsection{Geometric Adams operations}
\label{subsec_geomtric_adams_operations}

The goal of this subsection is to prove Proposition \ref{prop_adgeom_outadfus}.
Let $\G=(S,\F,\L)$ be a $p$-local compact group.

\begin{defn}[Compare {\cite[Definition 3.4]{JLL}}]
\label{def_geometric_Adams_operation}
A \emph{geometric unstable Adams operation} of degree $\zeta \in \ZZ_p^\times$ is a self equivalence $f \colon B\G \to B\G$ such that there exists $\vp \in \Ad(S)$ of degree $\zeta$ which renders the following diagram commutative
\[
\xymatrix{
BS \ar[d] \ar[r]^{B\vp} & BS \ar[d] 
\\
B\G \ar[r]_f & B\G.
}
\] 
We will write
\[
\Ad^{\geom}(\G)
\]
for the group of \emph{homotopy classes} of geometric unstable Adams operations.
\end{defn}

This definition differs slightly from \cite{JLL} where we wrote $\Ad^g(\G)$ for the topological monoid of unstable Adams operations and used $\pi_0\Ad^g(\G)$ for what we now denote by $\Ad^{\geom}(\G)$.

If one is willing to assumes that the maximal torus $T\le S$ is self-centralising in $S$, then by \cite[Proposition 3.2]{BLO7} one may equivalently define a geometric unstable Adams operation of degree $\zeta \in \ZZ_p^\times$ on $\G$ to be a self-map of $B\G$ that induces multiplication by $\zeta^k$  on $H^{2k}_{\QQ_p}(B\G)$  for every $k \geq 1$, where $H^*_{\QQ_p}(X)=H^*(X;\ZZ_p)\otimes  \QQ$.

At this point the reader may wonder about the restriction to normal Adams automorphisms of $S$ (In other words, why do we not  define $\Ad(S)$ to be the group of all  Adams automorphisms, rather than only the normal ones).
Indeed, in general there are fusion preserving Adams automorphisms which are not normal.
However, recall that one has the following.
\begin{lem}[{\cite[Lemma 2.5]{JLL}}]\label{lem_connected_implies_weakly_connected}
If $T$ is self-centralising in $S$ then every Adams automorphism of $S$ is normal.
\end{lem}

As we  shall see below, the requirement that $T$ is self-centralising in $S$ appears frequently in our analysis. This motivates the following definition.

\begin{defn}\label{def_weakly_connected}
A saturated fusion system $\F$ over $S$ is \emph{weakly connected} if its maximal torus $T$ is self-centralising in $S$, namely $C_S(T)=T$.
\end{defn}

A compact Lie group $G$ is connected if and only if every element of $G$ is conjugate to an element of its maximal torus, and in this case the maximal torus is self-centralising in $G$, so the corresponding fusion system is weakly connected. It is  tempting to define  a fusion system to be connected if every element of $T$ is $\F$-conjugate to $T$. However, it is not at all clear that this is indeed the correct definition of connectivity for general fusion system. Furthermore, for the purposes of this paper the only condition we need is weak connectivity.

We recall the following from \cite{BLO3}.
For any discrete $p$-toral group $Q$ define an equivalence relation on the set $\Hom(Q,S)$ where $\rho \sim \rho'$ if there exists some $\F$-isomorphism $\chi \colon \rho(Q) \to \rho'(Q)$ such that $\rho'= \chi \circ \rho$.
Set $\Rep(Q,\L) = \Hom(Q,S)/\sim$.
Notice that $\Rep(S,\L)=\Aut(S)/\Aut_\F(S)$, (left cosets of $\Aut_\F(S)$ in $\Aut(S)$).
Recall from \cite[Theorem 6.3]{BLO3} that there is a natural bijection
$[BQ,B\G] \cong \Rep(Q,\L)$.
In particular there is a natural bijection of sets
\begin{equation}
[BS,B\G] \cong \Aut(S)/\Aut_\F(S).  
\end{equation}
The map $\res \colon [B\G,B\G] \to [BS,B\G]$ therefore gives a homomorphism
\begin{equation}\label{def_blo3_mu}
\Out(B\G) \xto{\mu} \Out_{\fus}(S)
\end{equation}
and by \cite[Proposition 7.2, Proposition 5.8]{BLO3} there is an exact sequence
\begin{equation}\label{eq_blo3_outbg_outfuss}
0 \to \higherlim{\mathcal{O}^c(\F)}{1}\mathcal{Z} \to \Out(B\G) \xto{\mu} \Out_{\fus}(S) \to \higherlim{\mathcal{O}^c(\F)}{2}\mathcal{Z}
\end{equation}
where the groups at the ends are finite abelian $p$-groups by  \cite[Proposition 5.8]{BLO3} and since $\O^c(\F)$ is equivalent to a finite subcategory.

We are now ready to restate and prove Proposition \ref{prop_adgeom_outadfus}.
\begin{prop}\label{prop_adgeom_outadfus_in_sec}
Let $\G=\SFL$ be a $p$-local compact group.
Then $\Ad^{\geom}(\G)$ is \zpfree~ and contains a normal Sylow $p$-subgroup.
In addition there is a short exact sequence
\[
1\to \higherlim{\mathcal{O}^c(\F)}{1}\mathcal{Z} \to \Ad^\geom(\G)\to \OutAd_{\fus}(S)\to 1.
\]
If $p$ is odd then $\Ad^{\geom}(\G) \cong \OutAd_{\fus}(S)$.
In addition $\Ad^{\geom}(\G)$ is solvable, in fact, $\OutAd_{\fus}(S)$ is solvable of class $\leq 2$.
\end{prop}
\begin{proof}
As we have seen above, by \cite{BLO3} there is a commutative square
\[
\xymatrix{
\Out(B\G) \ar[rr]^{\mu} \ar[d]_{\res} &&
\Out_{\fus}(S) \ar@{^(->}[d] 
\\
[BS,B\G] \ar[rr]_{\cong} & &
\Aut(S)/\Aut_\F(S)
}
\]
Hence, if $f \colon B\G \to B\G$ is a geometric unstable Adams operation then by definition there exists $\vp \in \Ad(S)$ such that the diagram in Definition \ref{def_geometric_Adams_operation} commutes up to homotopy and since the right arrow in the diagram above is injective, $\mu([f])=[\vp]$.
In particular $\vp$ is fusion preserving and hence $\mu([f]) \in \OutAd_{\fus}(S)$.
Conversely, suppose that $f \colon B\G \to B\G$ is a self equivalence such that $\mu([f]) \in \OutAd_{\fus}(S)$.
Then there exists some $\vp \in \Ad_{\fus}(S)$ such that $\mu([f])=[\vp]$.
Since the bottom arrow in the diagram above is bijective it follows that the diagram in Definition \ref{def_geometric_Adams_operation} commutes up to homotopy and therefore $f$ is a geometric unstable Adams operation.
We have thus shown that
\[
\Ad^{\geom}(\G) = \mu^{-1}(\OutAd_{\fus}(S)).
\]
By \cite[Theorem B]{LL} for any $p$-local compact group $\higherlim{\O^c(\F)}{2}\Z=0$ and, if $p$ is odd then $\higherlim{\O^c(\F)}{1}\Z=0$.
In light of \eqref{eq_blo3_outbg_outfuss}, the vanishing of $\varprojlim^2 \Z$ implies the exactness of the sequence in the statement of this proposition.
The vanishing of $\varprojlim^1 \Z$, provided $p$ is odd, establishes the isomorphism $\Ad^{\geom}(\G) \cong \OutAd_{\fus}(S)$. 
It follows from Proposition \ref{prop_outadfuss} that $\OutAd_{\fus}(S)$ is solvable of class at most $2$ and since $\higherlim{\mathcal{O}^c(\F)}{1}\mathcal{Z}$ is a finite abelian $p$-group, $\Ad^{\geom}(\G)$ has a normal Sylow $p$-subgroup.
Lemma \ref{lem_extn_zpfree}\eqref{lem_extn_zpfree_i} shows that it is also \zpfree.
\end{proof}

%%%%%%%%%%%%%%%%%%%%%%%%%%%%%%
%%%%%%%%%%%%%%%%%%%%%%%%%%%%%%
%%%%%%%%%%%%%%%%%%%%%%%%%%%%%%
%%%%%%%%%%%%%%%%%%%%%%%%%%%%%%
%%%%%%%%%%%%%%%%%%%%%%%%%%%%%%
%%%%%%%%%%%%%%%%%%%%%%%%%%%%%%
%%%%%%%%%%%%%%%%%%%%%%%%%%%%%%
%%%%%%%%%%%%%%%%%%%%%%%%%%%%%%
%%%%%%%%%%%%%%%%%%%%%%%%%%%%%%
%%%%%%%%%%%%%%%%%%%%%%%%%%%%%%
%%%%%%%%%%%%%%%%%%%%%%%%%%%%%%
%%%%%%%%%%%%%%%%%%%%%%%%%%%%%%
\section{Algebraic unstable Adams operations}\label{Alg_Ad}

In this section we will prove Theorem \ref{thm_geom_vs_alg_uao}.
Throughout we will fix a $p$-local compact group $\G=(S,\F,\L)$.

\begin{defn}[{Compare \cite[Definition 1.11]{JLL}}]\label{def_cover}
Let $\G=\SFL$ be a $p$-local compact group and $\phi \colon S \to S$ be a fusion preserving automorphism.
Let $\L', \L''$ be subcategories of $\L$ and $\Phi \colon \L' \to \L''$ be a functor.
We say that $\Phi$ \emph{covers} $\phi$ if the following square is commutative
\[
\xymatrix{
\L' \ar[r]^{\Phi} \ar[d]_{\pi} &
\L'' \ar[d]^{\pi} 
\\
\F \ar[r]_{\phi_*} & \F
}
\]
and if for every $P,Q \in \L'$ and every $g \in N_S(P,Q)$ the morphisms $\widehat{g} \in \L(P,Q)$ and $\widehat{\phi(g)} \in \L(\phi(P),\phi(Q))$ belong to $\L'$ and $\L''$ respectively, and moreover $\Phi(\widehat{g})=\widehat{\phi(g)}$.
\end{defn}

\begin{defn}[{\cite[Definition 3.3]{JLL}}]\label{def_algebraic_unstable_adams_operation}
An algebraic unstable Adams operation on $\G$, or simply an \emph{unstable Adams operation on $\G$}, is a pair $(\Psi,\psi)$, where $\psi \in \Ad_{\fus}(S)$ and $\Psi \colon \L \to \L$ is a functor which \emph{covers} $\psi$.
In other words,
\begin{itemize}
\item[(a)]
The following diagram is commutative where $\psi_*$ is the automorphism of $\F$ induced by $\psi$ 
\[
\xymatrix{
\L \ar[d]_\pi \ar[r]^\Psi & \L \ar[d]^\pi 
\\
\F \ar[r]_{\psi_*} & \F
}
\]

\item[(b)]
For every $\F$-centric $P,Q \leq S$ and every $g \in N_S(P,Q)$ we have $\Psi(\widehat{g})=\widehat{\psi(g)}$.
\end{itemize}
Define 
\[
\Ad(\G) = \{ \text{ all unstable Adams operations $(\Psi,\psi)$ on $\G$ }\}.
\]
If $H \leq \ZZ_p^\times$ we will write $\Ad^H(\G)$ for those $(\Psi,\psi)$ such that $\deg(\psi) \in H$.
\end{defn}

\begin{remark}\label{rem_geom_uao}
Recall that the inclusion $BS \xto{\incl} B\G$ is induced by the geometric realisation of the inclusion of categories $\B S \subseteq  \B\Aut_\L(S) \subseteq \G$ via $s \mapsto \hat{s}$.
By definition of an unstable Adams operation $(\Psi,\psi)$, upon taking geometric realisations the following diagram commutes strictly
\[
\xymatrix{
BS \ar[d] \ar[r]^{B\psi} & BS \ar[d] 
\\
B\G \ar[r]_{\pcomp{|\Psi|}} & B\G
}
\]
and hence $\pcomp{|\Psi|}$ is a geometric unstable Adams operation (Definition \ref{def_geometric_Adams_operation}).
We obtain a homomorphism
\begin{equation}
\label{def_gamma}
\gamma \colon \Ad(\G) \xto{ \ (\Psi,\psi) \mapsto \pcomp{|\Psi|} \ } \Ad^\geom(\G).
\end{equation}

Recall from \cite[Section 7]{BLO3} that an automorphism $\Phi \colon \L \to \L$ is called \emph{isotypical} if for every $\F$-centric subgroup $P \leq S$ the isomorphism $\Aut_\L(P) \to \Aut_\L(\Phi(P))$ induced by $\Phi$ carries $\de(P)$ to $\de(\Phi(P))$.
The collection of all isotypical self equivalences of $\L$ that preserve inclusions is a group \cite[Lemma 1.14]{AOV}, and we denote it here by  $\Aut_{\typ}(\L)$ (The notation in \cite{AOV} is $\Aut^I_{\typ}(\L)$).
It is clear from the definition that if $(\Psi,\psi)$ is an unstable Adams operation  then $\Psi \in \Aut_{\typ}(\L)$.
Moreover, $\psi$ is completely determined by $\Psi$ because for every $s \in S$ we have $\Psi(\hat{s})=\widehat{\psi(s)}$ and since $\de \colon S \to \Aut_\L(S)$ is injective.
Therefore
\[
\Ad(\G) \xto{ \ (\Psi,\psi) \mapsto \Psi \ } \Aut_{\typ}(\L).
\]
is injective and we can, and will, identify $\Ad(\G)$ with a subgroup of $\Aut_{\typ}(\L)$.
\end{remark}

Geometric realisation of isotypical equivalences gives rise to a homomorphism \cite[Section 7]{BLO3}
\[
\Omega \colon \Aut_{\typ}(\L) \xto{ \ \Phi \mapsto [\pcomp{|\Phi|}] \ } \Out(B\G).
\]
By Remark \ref{rem_geom_uao} the restriction of $\Omega$ to $\Ad(\G)$ is the map $\gamma$ in \eqref{def_gamma}.
Among all isotypical equivalences of $\L$ there are the ones that are induced by ``conjugation'' by automorphisms of $S$.
To be precise, there is a homomorphism
\begin{equation}\label{eq:autls_to_auttypl}
\tau \colon \Aut_\L(S) \to \Aut_{\typ}(\L)
\end{equation}
defined for every $\vp \in \Aut_\L(S)$  as follows.
On object, $\tau(\vp)(P) = \pi(\vp)(P)$ where $\pi \colon \L \to \F$ is the projection.
If $\al \in \L(P,Q)$ is a morphism, set $\tau(\vp)(\al) = (\vp|_Q^{\vp(Q)}) \circ \al \circ (\vp|_P^{\vp(P)})^{-1}$ where we write $\vp(P)$ instead of $\pi(\vp)(P)$ etc.
It is easy to check, using the axioms of linking systems, that $\tau(\vp)$ is isotypical.
In fact, $\Aut_\L(S)\nsg\Aut_\typ(\L)$  and the quotient group is denoted $\Out_\typ(\L)$.
It is shown in \cite[Theorem 7.1]{BLO3} that there is an isomorphism
\begin{equation}\label{eq:outtyp_outbg}
\Out_\typ(\L) \cong \Out(\pcomp{|\L|})
\end{equation}
via the assignment $[\Phi] \mapsto [\pcomp{|\Phi|}]$.
The definition of $\Out_\typ(\L)$ in \cite{BLO3} and our definition coincide by \cite[Lemma 1.14]{AOV}.

%%%%%%
%%%%%%
%%%%%%
\begin{lem}\label{lem_deg1_conj_T}
Let $\F$ be a  weakly connected fusion system over $S$.
If $\psi \in \Aut_\F(S)$ is an Adams automorphism of degree $1$, then $\psi$ is conjugation by $t$ for some $t \in T$.
\end{lem}

\begin{proof}
Since $\F$ is weakly connected, $T$ is $\F$-centric. % and by Lemma \ref{lem_connected_implies_weakly_connected} $\psi$ is normal.
Now, $\psi|_T = \id_S|_T$  and therefore  \cite[Prop. 2.8]{BLO3} implies that $\psi$ is conjugation by $t$ for some $t\in Z(T)=T$.
\end{proof}

\begin{lem}\label{lem_deg1_operation_in_F}
Let $\G=\SFL$ be a  weakly connected $p$-local compact group, and suppose that 
\[
(\Psi,\Id_S) \ \in \ \Ker\left(\gamma\colon\Ad(\G) \to \Ad^\geom(\G)\right).
\]
Then $\Psi=\tau(\hat{t})$ for some $t \in T$.
\end{lem}

\begin{proof}
By \cite[Theorem 7.1]{BLO3} the map $\gamma$ is the restriction to $\Ad(\G)$ of the composition of homomorphisms 
\[
\Aut_{\typ}(\L) \xto{ \proj } \Out_{\typ}(\L) \xto{\cong} \Out(B\G).
\]
% Consider some $(\Psi,\psi)$ in \mdf{$\Ker(\gamma) \cong\Aut_\L(S)$}.
Note that $\Psi(P)=\Id_S(P)=P$ for every $\F$-centric $P \leq S$ and that given any $\vp \in \Mor_\L(P,Q)$ the images of $\vp$ and $\Psi(\vp)$ in $\Hom_\F(P,Q)$ are equal because $\Psi$ covers $(\Id_S)_*=\Id_\F$.
Also, since $(\Psi, \Id_S)\in\Ker(\ga)$, $\Psi$ is conjugation by some $\rho\in\Aut_\L(S)\cong\Aut_\L(S)$. In particular for every $\vp \in \Aut_\L(S)$, $\Psi(\vp)\circ\rho = \rho\circ\vp$.

Since $(\Psi,\Id_S)$ is an Adams operation,  $\Psi(\hat{s})=\hat{s}$ for every $s \in S$, and so $\rho \in C_{\Aut_\L(S)}(S)=Z(S)$, where $S$ is considered as a subgroup of $\Aut_\L(S)$ via the canonical inclusion.
Since $\F$ is weakly connected $T = C_S(T) \ge Z(S)$, and therefore $\rho=\hat{t}$ for some $t \in Z(S) \le T$. Thus, by the definition of $\tau$ in \eqref{eq:autls_to_auttypl}, $\Psi=\tau(\hat{t})$ as claimed.
\end{proof}

We will now consider the following subgroup of $\Aut_\L(S)$
\[
\Aut_\L^{\Ad}(S) = \{ \vp \in \Aut_\L(S) : \pi(\vp) \in \Aut_\F(S) \cap \Ad(S)\}.
\]

Recall that $Z(\F)$ denotes the centre of $\F$.

%%%
%%%
%%%
\begin{lem}\label{lem_autlads_to_adg_exact_seq}
For any $p$-local compact group $\G=\SFL$ there is an exact sequence
\[
1 \to Z(\F) \xto{z \mapsto \hat{z} } \Aut_\L^{\Ad}(S) \xto{ \ \tau \ } \Ad(\G) \xto{ \ \ga \ } \Ad^{\geom}(\G) \to 1.
\]
\end{lem}

\begin{proof}
The homomorphism $\ga$ defined in \eqref{def_gamma} is surjective by \cite[Prop. 3.5]{JLL}.
We have seen above that if $\vp \in \Aut_\L^{\Ad}(S)$ then $\tau(\vp)$ is an unstable Adams operation whose geometric realisation is homotopic to the identity, so $\im(\tau) \leq \ker(\ga)$.
Consider some $(\Psi,\psi)$ in the kernel of $\gamma$.
We will show that $\Psi$ is in the image of $\tau$.
By the bijection \eqref{eq:outtyp_outbg} and by the definition of $\Out_{\typ}(\L)$ it follows that there exists a natural isomorphism $\rho \colon \Id_\L \to \Psi$.
In particular, at the object $S \in \L$ we obtain a morphism $\rho_S \colon S \to S$ in $\L$, i.e. $\rho_S \in \Aut_\L(S)$.
In fact, $\rho_S$ determines $\rho$ completely because for every $P \in \L$ we have the following commutative square
\[
\xymatrix{
P \ar[r]^{\iota_P^S} \ar[d]_{\rho_P} & 
S \ar[d]^{\rho_S} 
\\
\psi(P) \ar[r]_{\Psi(\iota_P^S)} & S.
}
\]
Since $\iota_P^S=\widehat{e}$ where $e \in N_S(P,S)$ and since $\Psi(\widehat{e})=\widehat{e}$, we see that $\Psi(\iota_P^S)=\iota_{\psi(P)}^S$.
Since $\iota_P^S$ is a monomorphism we deduce that  $\rho_P$ is determined by $\rho_S$, in fact $\rho_P = \rho_S|_P^{\psi(P)}$ for any $P \in \L$.
For every $s \in S$ we obtain the following commutative square
\[
\xymatrix{
S \ar[r]^{\hat{s}} \ar[d]_{\rho_S} & S \ar[d]^{\rho_S} \ar[d]^{\rho_S}
\\
S \ar[r]_{\widehat{\psi(s)}} & S.
}
\]
It follows from Axiom (C) of linking systems \cite[Def. 4.1]{BLO3} that $\pi(\rho_S)=\psi$, in particular $\rho_S \in \Aut_\L^{\Ad}(S)$.
Now we claim that $\Psi=\tau(\rho_S)$.
First, for every $\F$-centric $P \leq S$,
\[
\Psi(P) = \psi(P) = \pi(\rho_S)(P) = \tau(\rho_S)(P).
\]
So $\Psi$ and $\tau(\rho_S)$ agree on objects of $\L$.
Next, suppose that $\al \in \Mor_\L(P,Q)$.
Since $\rho$ is a natural transformation there is a commutative diagram
\[
\xymatrix{
P \ar[d]_{\rho_P} \ar[r]^{\al} &
Q \ar[d]^{\rho_Q} 
\\
\psi(P) \ar[r]_{\Psi(\al)} & \psi(Q).
}
\]
Since $\rho_P$ and $\rho_Q$ are restrictions of $\rho_S$ we get
\[
\Psi(\al) = \rho_Q\circ \al \circ \rho_P^{-1} = (\rho_S|_Q^{\psi(Q)}) \circ \al \circ (\rho_S|_{P}^{\psi(P)})^{-1} = \tau(\rho_S)(\al).
\]
This completes the proof that the sequence is exact at $\Ad(\G)$.

We now show exactness at $\Aut_\L^{\Ad}(S)$.
If $z \in Z(\F)$ then for any $\vp \in \Mor_\L(P,Q)$ we have $z \in Z(S) \leq P,Q$ and
\[
\tau(\hat{z})(\vp) = \hat{z}|_Q^Q \circ \vp \circ (\hat{z}|_P^P)^{-1} =
\hat{z}|_Q^Q \circ (\widehat{\pi(\vp)(z)}|_Q^Q)^{-1} \circ \vp = 
\hat{z}|_Q^Q \circ (\widehat{z}|_Q^Q)^{-1} \circ \vp = 
\vp.
\]
Therefore $Z(\F) \to \Aut_\L^{\Ad}(S) \xto{\tau} \Ad(\G)$ is trivial.
Now suppose $\al \in \Ker(\tau)$.
Then in particular $\al \in C_{\Aut_\L(S)}(S)=Z(S)$, namely $\al=\hat{z}$ for some $z \in Z(S)$.
In addition for any $\vp \in \Mor_\L(P,Q)$ we must have $\hat{z}|_Q^Q \circ \vp \circ (\hat{z}|_P^P)^{-1} =\vp$.
Axiom (C) of linking systems and the fact that $\vp$ is an epimorphism in $\L$ imply that $\pi(\vp)(z)=z$.
It follows that $z \in Z(\F)$ and this shows the exactness at $\Aut_\L^{\Ad}(S)$.
\end{proof}

Composition of the degree map with the projection $\pi \colon \L \to \F$ gives rise to a degree homomorphism
\[
\deg \colon \Aut_\L^{\Ad}(S) \xto{\pi} \Aut_\F(S) \cap \Ad(S) \xto{ \ \deg \ } \ZZ_p^\times.
\]
Recall that $\Out_\F(S)$ is finite, hence every element of $\Aut_\F(S)$ has finite order.
Therefore the subgroup 
\begin{equation}\label{def_DF}
D(\F) \defeq  \Image \  \left( \Ad(S) \cap \Aut_\F(S) \xto{\deg} \ZZ^\times_p \right).
\end{equation}
must be contained in the group $U_p$ of the roots of unity in $\ZZ_p$ which, by \cite[Sec. 6.7, Prop. 1,2]{Ro}, is isomorphic to the cyclic group $C_{p-1}$ if $p$ is odd, and to $C_2$ if $p=2$.

%Thus $D(\F)$ is the group of all possible degrees of the Adams automorphisms of $S$ which belong to $\F$.

\begin{lem}\label{lem_t_to_autlads_ato_df}
Let $\G=(S,\F,\L)$ be a weakly connected $p$-local compact group.
There is a short exact sequence
\[
1 \to T \to \Aut_\L^{\Ad}(S) \xto{\deg} D(\F) \to 1.
\]
\end{lem}
\begin{proof}
By definition, $\deg$ is onto $D(\F)$.
Also $T \leq \Ker(\deg)$ since $T$ is abelian so $\Aut_T(S) \leq \Ad^{\deg=1}(S)$.
Suppose that $\vp \in \Ker(\deg)$.
Lemma \ref{lem_deg1_conj_T} shows that $\pi(\vp)$ is conjugation by some $t' \in T$.
Since $Z(S) \leq C_S(T) =T$ and $T$ is $\F$-centric by the assumption of weak connectedness, $\vp=\hat{t}$ for some $t \in T$.
\end{proof}

\begin{defn}\label{def_inner}
Let $\Inn_T(\G) \leq \Ad(\G)$ for the subgroup of inner Adams operations of degree 1, i.e., operations of the form $(c_{\widehat{t}},c_t)$, where $c_t \in \Inn(S)$ is conjugation by $t\in T$ and $c_{\widehat{t}} \in \Aut(\L)$ is ``conjugation'' by $\widehat{t} \in \Aut_\L(S)$, (See \eqref{eq:autls_to_auttypl} where we used the notation $\tau(\widehat{t})$).
\end{defn}

We now restate and prove Theorem \ref{thm_geom_vs_alg_uao}.
\begin{thm}\label{thm_geom_vs_alg_uao_in_sec}
Let $\G=(S,\F,\L)$ be a weakly connected $p$-local compact group (Definition \ref{def_weakly_connected}).
Let $T$ denote the maximal torus of $S$.
Then
\begin{enumerate}[(i)]
\item
$\Ad(\G)$ has a normal maximal discrete $p$-torus denoted $\Ad(\G)_0$.
It contains a Sylow $p$-subgroup which is normal if $p=2$.

\item
$\Ad(\G)_0$ is contained in the kernel of $\gamma$ (see \eqref{def_gamma}) and there is a short exact sequence
\[
1 \to D(\F) \to \Ad(\G)/\Ad(\G)_0 \xto{\ \bar{\gamma} \ } \Ad^{\geom}(\G) \to 1
\]

\item
$\Ad(\G)_0=\Inn_T(\G) \cong T/Z(\F)$ where $Z(\F)$ is the centre of $\F$.
\end{enumerate}
\end{thm}
\begin{proof}
Clearly $T \leq \Aut_\L^{\Ad}(S)$, since $\Aut_T(S) \leq \Aut_\F(S) \cap \Ad(S)$.
Since $Z(\F) \leq Z(S) \leq T$, The exact sequences in Lemmas \ref{lem_t_to_autlads_ato_df} and \ref{lem_autlads_to_adg_exact_seq} yield the exact sequence
\[
1 \to D(\F) \xto{\tau} \Ad(\G)/\tau(T) \xto{\bar{\ga}} \Ad^{\geom}(\G) \to 1.
\]
Since $D(\F)$ is finite and $\Ad^{\geom}(\G)$ is \zpfree~ by Proposition \ref{prop_adgeom_outadfus}, Lemma \ref{lem_extn_zpfree}(\ref{lem_extn_zpfree_i}) implies that $\Ad(\G)/\tau(T)$ is \zpfree.
It follows that $\tau(T)$ is a normal maximal discrete $p$-torus in $\Ad(\G)$ denoted $\Ad(\G)_0$.

Let $P$ be the unique Sylow $p$-subgroup of $\Ad^{\geom}(\G)$ guaranteed in Proposition \ref{prop_adgeom_outadfus} and let $H$ be its preimage in $\Ad(\G)/\Ad(\G)_0$.
%Clearly $H \nsg \Ad(\G)/\Ad(\G)_0$.
If $p=2$ then $D(\F) \leq U_p \cong C_2$, so $H$ is a finite $2$-group and it is clearly the Sylow $2$-subgroup of $\Ad(\G)/\Ad(\G)_0$.
In particular $\Ad(\G)$ has a normal Sylow $2$-group.
If $p$ is odd, any Sylow $p$-subgroup of (the finite) group $H$ is a Sylow $p$-subgroup of $\Ad(\G)/\Ad(\G)_0$.
This completes the proof.
\end{proof}

%%%%%%%%%%%%%%%%%%%%%%%%%%%%%%
%%%%%%%%%%%%%%%%%%%%%%%%%%%%%%
%%%%%%%%%%%%%%%%%%%%%%%%%%%%%%
%%%%%%%%%%%%%%%%%%%%%%%%%%%%%%
%%%%%%%%%%%%%%%%%%%%%%%%%%%%%%
%%%%%%%%%%%%%%%%%%%%%%%%%%%%%%
%%%%%%%%%%%%%%%%%%%%%%%%%%%%%%
%%%%%%%%%%%%%%%%%%%%%%%%%%%%%%
%%%%%%%%%%%%%%%%%%%%%%%%%%%%%%
%%%%%%%%%%%%%%%%%%%%%%%%%%%%%%
%%%%%%%%%%%%%%%%%%%%%%%%%%%%%%
%%%%%%%%%%%%%%%%%%%%%%%%%%%%%%
\section{The degree of an unstable Adams operation (uniqueness)}\label{sec_degree_uniqueness}

Unstable Adams operations of connected compact Lie groups have the pleasant property that their degree determines them completely up to homotopy.
That is, given a connected compact Lie group $G$ and an integer $k \geq 1$, up to homotopy  there is at most one unstable Adams operation on $BG$ of degree $k$.
This was shown in \cite[Theorem 1]{JMO1}.
In the setup of $p$-local compact groups,  Proposition \ref{prop_adgeom_outadfus} makes it clear that this does not hold in general.
First, at least when $p=2$, any non-trivial element in $\varprojlim^1 \Z$ gives rise to a geometric unstable Adams operation whose underlying automorphism of $S$ is the identity.
Secondly, since $\OutAd_{\fus}(S) \cong \Ad_{\fus}(S)/(\Aut_\F(S) \cap \Ad(S))$, the degree can only be defined modulo $D(\F)$.
That is, the degree map that we must consider is
\[
\deg \colon \Ad^{\geom}(\G) \to \OutAd_{\fus}(S) \xto{\deg} \ZZ_p^\times/D(\F).
\]
From the algebraic point of view we can define a degree homomorphism
\[
\deg \colon \Ad(\G) \xto{(\Psi,\psi) \mapsto \psi} \Ad_{\fus}(S) \xto{\deg} \ZZ_p^\times.
\]
But observe that every $t \in T \setminus Z(\F)$ gives an unstable Adams operation $\tau(\widehat{t})$ of degree $1$, hence the kernel of this degree map is in general not trivial.
In this section we find conditions which guarantee that the degree does determine the unstable Adams operation in a suitable sense.

\begin{defn}\label{def_weyl_F}
Let $\F$ be a saturated fusion system over $S$.
The Weyl group of $\F$ is $W(\F)=\Aut_\F(T)$.
\end{defn}
If no confusion can arise we will usually denote $W(\F)$ simply by $W$.

Recall that $\Aut_\L^{\Ad}(S)$ is the subgroup of $\Aut_\L(S)$ of the morphisms which project to $\Aut_\F(S) \cap \Ad(S)$.
There is therefore a degree homomorphism $\deg \colon \Aut_\L^{\Ad}(S) \to \Aut_\F(S) \cap \Ad(S) \xto{\deg} \ZZ_p^\times$.
Its image is $D(\F)$ by \eqref{def_DF}.

\begin{lem}\label{lem_dft_ses}
Let $\G=\SFL$ be a weakly connected $p$-local compact group, and suppose that $H^1(W,T)=0$.
Then there is a short exact sequence
\[
1 \to T \xto{\de|_T} \Aut_\L^{\Ad}(S) \xto{\deg} D(\F) \to 1
\]
\end{lem}

\begin{proof}
Clearly $\de|_T$ is injective and the degree map is surjective.
Also $\deg \circ \de|_T$ is the trivial homomorphism since $T$ is abelian.
Suppose that $\vp \in \Ker(\deg)$.
Then $\psi(\vp)|_T=\id_T$ and since $T$ is $\F$-centric, it follows from \cite[Proposition 2.8]{BLO3} that $\pi(\vp)=c_z$ for some $z \in Z(T)=T$.
Since $Z(S) \leq T$ we now deduce that $\vp=\hat{t}$ for some $t \in T$.
\end{proof}

\begin{lem}\label{lem_surj_adg_to_adfuss}
The homomorphism 
\[
\Ad(\G) \xto{ \ (\Psi,\psi) \mapsto \psi \ } \Ad_{\fus}(S)
\]
is surjective.
\end{lem}

\begin{proof}
Consider some $\vp \in \Ad_{\fus}(S)$ and let $[\vp]$ denote its class in $\OutAd_{\fus}(S)$.
The exact sequences in Proposition \ref{prop_adgeom_outadfus} and Theorem \ref{thm_geom_vs_alg_uao} show that there exists $(\Psi,\psi) \in \Ad(\G)$ such that $[\vp]=[\psi]$.
Hence $\vp=\al \circ \psi$ for some $\al \in \Aut_\F(S) \cap \Ad(S)$.
Let $\tilde{\al} \in \Aut_\L^{\Ad}(S)$ be a lift for $\al$.
Then $\tau(\tilde{\al}) \circ \Psi$ is an unstable Adams operation with an underlying automorphism $\varphi$.
\end{proof}

\begin{lem}\label{L:deg 1 implied idW}
Let $\G=\SFL$ be a weakly connected $p$-local compact group and let $(\Psi,\psi) \in \Ad(\G)$. 
If $\deg(\psi)=1$, then the automorphism induced by $\Psi$ on $W=W(\F)$ is the identity.
\end{lem}
\begin{proof}
Notice first that $\Psi$ induces an automorphism of $\Aut_\L(T)$, and hence an automorphism of $W=\Aut_\F(T)$.
Set $G=\Aut_\L(T)$ and choose $g \in G$.
Since $\Psi$ covers $\psi_*$ and since $\psi$ is the identity on $T$, for any $t \in T$
\[
(\pi\circ \Psi(g))(t)=\psi_*(\pi(g))(t) = \psi(\pi(g)(\psi^{-1}(t))) =
\psi(\pi(g)(t)) =\pi(g)(t).
\]
Thus, $\Psi(g)$ and $g$ have the same image in $\Aut_\F(T)=W$.
\end{proof}

\begin{lem}\label{lem_h1wt=>deg_injective}
Let  $\G=\SFL$ be weakly connected and assume that $H^1(W,T)=0$.
Then \[\deg\colon\OutAd_{\fus}(S) \to\ZZ_p^\times/D(\F)\] is injective.
\end{lem}

\begin{proof}
Consider $\psi \in \Ad_{\fus}(S)$ such that $\deg(\psi) \in\D(\F)$.
Then there exists $\vp \in \Aut_\F(S) \cap \Ad(S)$ such that $\deg(\psi)=\deg(\vp)$.
Set $\theta = \vp^{-1} \circ \psi$.
Clearly $\theta \in \Ad^1(S)$ and it is fusion preserving.
By Lemma \ref{lem_surj_adg_to_adfuss} there exists $(\Theta,\theta) \in \Ad(\G)$.
Let $\al$ be the automorphism that $\Theta$ induces on $\Aut_\L(T)$.
Clearly, $\al$ is the identity on $T \leq \Aut_\L(T)$ since $\deg(\theta)=1$.
By Lemma \ref{L:deg 1 implied idW}, $\al$ induces the identity on $W=\Aut_\L(T)/T$.
Since we assume that $H^1(W,T)=0$, \cite[Lemma 2.2]{JLL} implies that $\al$ is conjugation by some $t \in T \leq G$.
Now, $\Theta$ is an unstable Adams operation, so for any $s \in S=N_S(T)$
\[
\widehat{\theta(s)}=\Theta(\widehat{s})=\al(\widehat{s})=\widehat{tst^{-1}}.
\]
This shows that $\theta=c_t$ and therefore $\psi=\vp \circ c_t \in \Aut_\F(S) \cap \Ad(S)$.
It follows that the kernel of $\deg$ is trivial.
\end{proof}

We now restate and prove Proposition \ref{prop_degree_determines_operation}.
\begin{prop}\label{prop_degree_determines_operation_in_sec}
Suppose that $\G=(S,\F,\L)$ is weakly connected and let $W=\Aut_\F(T)$ be its Weyl group.
If $H^1(W,T)=0$ then $\OutAd_\fus(S) \xto{\deg} \ZZ_p^\times/D(\F)$ is injective and there are exact sequences
\begin{eqnarray*}
(1) && 1 \to \higherlim{\O(\F^c)}{1} \Z \to \Ad^{\geom}(\G) \xto{\quad \deg \quad} \ZZ_p^\times/D(\F) \\
(2) && 1 \to \higherlim{\O(\F^c)}{1} \Z \to \Ad(\G)/\Ad(\G)_0 \xto{\quad \deg \quad} \ZZ_p^\times 
\end{eqnarray*}
If in addition $p\neq 2$, then the degree maps in (1) and (2) are injective. 
\end{prop}
\begin{proof}
The exact sequence (1) follows from the exact sequence in Proposition \ref{prop_adgeom_outadfus} and Lemma \ref{lem_h1wt=>deg_injective}.
Lemma \ref{lem_autlads_to_adg_exact_seq} together with Lemma \ref{lem_t_to_autlads_ato_df} and the fact that the image of $T \xto{\tau} \Ad(\G)$ is $\Ad(\G)_0$ by Theorem \ref{thm_geom_vs_alg_uao}, give rise to the following commutative diagram with short exact rows
\[
\xymatrix{
1 \ar[r] & 
\Aut_\L^{\Ad}(S)/T \ar[d]_{\cong}^{\deg} \ar[r]^{\tau} &
\Ad(\G)/\Ad(\G)_0 \ar[d]^{\deg} \ar[r]^{\bar{\ga}} &
\Ad^{\geom}(\G) \ar[d]^{\deg} \ar[r] &
1
\\
1 \ar[r] &
D(\F) \ar[r] &
\ZZ_p^\times \ar[r] &
\ZZ_p^\times/D(\F) \ar[r] &
1
}
\]
The snake lemma together with the exact sequence (1) yield the exact sequence (2).
If $p\neq 2$ then by \cite{LL} the group $\varprojlim^1 \Z$ vanishes and therefore the degree maps in (1) and (2) are injective.
\end{proof}

We are led to find conditions under which $H^1(W,T)=0$.
Recall that $U_p \leq \ZZ_p^\times$ is the torsion subgroup in the group of $p$-adic units.
It acts in the natural way on any discrete $p$-torus $T$ via central automorphisms.

\begin{lem}\label{lem_HUZp=0_on_natural_module}
Let $1\neq G\le U_p$ be a subgroup which acts on a discrete $p$-torus $T$ of rank $r \geq 1$ in the natural way. 
If $p \neq 2$ then $H^{n}(G,T)=0$ for all $n\geq 0$.
If $p=2$ then $H^{2n}(G,T)=(\ZZ/2)^r$ and $H^{2n+1}(G,T)=0$ for all $n \geq 0$.
\end{lem}

\begin{proof}
If $p \neq 2$ then $U_p\cong C_{p-1}$ and $T$ is $p$-torsion, hence $H^{i}(G,T)=0$ for $i\geq 1$.
Choose some $1 \neq \zeta \in G$.
Then $\zeta  \neq 1\mod p$, so $\zeta-1$ is an invertible element in $\ZZ_p$ and therefore it must act without fixed points on $T$.
This shows that $H^0(G,T)=0$.

Now suppose that $p=2$.
Then $G=U_p=C_2$.
If $t \in G$ is the non-trivial element, a projective resolution of $\ZZ$ is given by
\[
\dots \to \ZZ[G] \xto{1-t} \ZZ[G] \xto{1+t} \ZZ[G] \xto{1-t} \ZZ[G] \to \ZZ.
\]
Applying $\Hom_{\ZZ G}(-,T)$, one obtains the cochain complex
\[
\cdots\to T \xto{ \cdot 2} T \xto{0} T \xto{\cdot 2} T \xto{0} \cdots
\]
whose homology groups are $H^*(G,T)$.
The description of $H^*(G,T)$ now follows.
\end{proof}

We recall that $\QQ_p=\ZZ_p \otimes \ZZ[\frac{1}{p}]$.
We obtain a short exact sequence 
\[
1 \to \ZZ_p \to \QQ_p \to \ZZ/p^\infty \to 1.
\]
We also recall that $\Aut(\ZZ/p^\infty) \cong \ZZ_p^\times$ where every $\zeta \in \ZZ_p$ acts on $\ZZ/p^k \subseteq \ZZ/p^\infty$ via multiplication by $\zeta \mod p^k$.
Using the inclusion $\ZZ[\frac{1}{p}] \leq \QQ_p$ it is easy to check that for any $\zeta \in \ZZ_p^\times$, multiplication by $\zeta$ in $\QQ_p$ induces the automorphism $\zeta$ on $\ZZ_p$.
Thus the sequence is a short exact sequence of $\ZZ_p^\times$-modules.
More generally, if $T$ is a discrete $p$-torus of rank $r>0$, any decomposition $T \cong \oplus_r \ZZ/p^\infty$ gives rise to a short exact sequence of $\ZZ_p$-modules
\begin{equation}\label{eq_assoc_lvt} 
1 \to L \to V \to T \to 1
\end{equation}
where $V=\oplus_r \QQ_p$ and $L=\oplus_r \ZZ_p$.
Also, $\Aut(T) \cong \GL_r(\ZZ_p)$ which acts naturally on $L$ and $V$, and the sequence becomes a short exact sequence of $\GL_r(\ZZ_p)$-modules.
We notice that different decompositions of $T$ give rise to isomorphic $\GL_r(\ZZ_p)$-modules $V$ and $L$.

\begin{lem}\label{lem_HWL_dim_shift}
Let $T$ be a discrete $p$-torus, and let $G\le\Aut(T) \cong \GL_r(\ZZ_p)$ be a finite subgroup. 
Let $V$ and $L$  be the associated $\QQ_p$ representation and the corresponding integral lattice respectively as above.  
Then 
\[
H^i(G,L) \cong \begin{cases}
 H^{i-1}(G,T) & i>1,\\
 \Coker(H^0(G,V)\to H^0(G,T)) & i=1
 \end{cases}
\]
\end{lem}

\begin{proof}
Since $G$ is finite and $V$ is a rational vector space, $H^i(G,V)=0$ for all $i>0$. Thus the claim follows at once from the standard dimension shifting argument (see for instance \cite[III.7]{Br}).
\end{proof}

Recall that an element $w \in \GL_r(\QQ_p)$ is called a pseudo-reflection if $\rk(w-1)=1$, namely $w$ fixes a hyperplane of dimension $r-1$.
A subgroup $W \leq \GL_r(\QQ_p)$ is called a pseudo-reflection group if it is generated by pseudo-reflections.

We notice that $\Ad(S) \xto{  \ \deg \ } \ZZ_p^\times \hookrightarrow Z(\Aut(T))$ is a factorisation of the restriction map $\Ad(S) \xto{ \ \vp \mapsto \vp|_T \ } \Aut(T)$.
In light of \eqref{def_DF}, we see that $\D(\F)$ can be identified with a subgroup of $Z(\Aut_\F(T))$ which acts in the natural way on $T$ considered as a subgroup of $U_p \leq \ZZ_p^\times$.

\begin{prop}\label{prop_H1W_vanish}
Let $\F$ be a weakly connected saturated fusion system and let $W=\Aut_\F(T) \leq \GL_r(\ZZ_p)$ be its Weyl group.
Assume that either one of the following conditions holds:
\begin{enumerate}[(a)]
\item
$p$ is odd and $D(\F) \neq 1$, or

\item $p=2$, $D(\F)\neq 1$ and $H^1(W/D(\F), T^{D(\F)})=0$, or

\item
$p$ is odd and the Weyl group $W$ is a pseudo-reflection group.
\end{enumerate}
Then $H^1(W,T)=0$.
\end{prop}

\begin{proof} %[Proof of Proposition \ref{prop_H1W_vanish}]
By the remarks above there is a central extension
\[
1\to D(\F)\to W\to W/D(\F)\to 1
\]
and $D(\F) \leq U_p$ acts on $T$ in the natural way.
The associated Lyndon-Hochschild-Serre spectral sequence takes the form
\[
E_2^{i,j} \ = \ H^i\Big(\,W/D(\F)\, ,\, H^j(D(\F),T)\,\Big) \Rightarrow H^{i+j}(W,T).
\]
If $p\neq 2$ and $D(\F) \neq 1$ then Lemma \ref{lem_HUZp=0_on_natural_module} implies that $H^j(D(\F),T)=0$ for all $j\geq 0$, and hence that $E^{i,j}=0$ for all $i,j$.
Thus $H^k(W,T)=0$ for all $k\geq 0$. 

If $p=2$ and $D(\F) \neq 1$, then $E_2^{i,j} = 0 $ for $j$ odd since in this case $H^j(D(\F), T)=0$ by Lemma \ref{lem_HUZp=0_on_natural_module}.
Hence the only, potentially non-zero contribution to $H^1(W,T)$ comes from $H^1(W/D(\F), H^0(D(\F),T)) = H^1(W/D(\F), T^{D(\F)})$ which vanishes by hypothesis. Hence  $H^1(W,T)=0$.

Finally, suppose that $p \neq 2$ and $W \leq \Aut(T) \cong \GL_r(\ZZ_p)$ is a pseudo-reflection group.
Let $L$ and $V$ be the associated $p$-adic lattice and $\QQ_p$-vector space associated to $T$ as in \eqref{eq_assoc_lvt}.
This is a short exact sequence of $\ZZ_p[W]$-modules.
Now, $W \leq \GL_r(\ZZ_p)$ acts faithfully on $L$ so \cite[Theorem 3.3]{An99} applies and it follows that $H^2(W,L)=0$.
Hence, $H^1(W,T)=0$ by Lemma \ref{lem_HWL_dim_shift}.
\end{proof}

We end this section with a proof of Proposition \ref{stable} restated as \ref{sec_stable}.
This is a ``stabilisation'' result showing that given any two unstable Adams operations of the same degree, there is some $n$, such that their $n$-th powers (i.e., $n$-fold composition) give homotopic operations.

Let $S$ be a discrete $p$-toral group with maximal torus $T$.
The image of $x \in S$ under $\pi \colon S \to S/T$ will be denoted $\overline{x}$.
A section $\si \colon S/T \to S$ is a function such that $\pi \circ \si=\id_{S/T}$.

\begin{lem}\label{L:power Adams aut not 1}
Let $S$ be a discrete $p$-toral group and $1 \neq \zeta \in \Ga_1(p)$.
Suppose that $\psi_1,\psi_2 \in \Ad(S)$ have degree $\zeta$.
Then there exists some $r>0$ such that $\psi_1^{p^r} = \psi_2^{p^r}$.
\end{lem}

\begin{proof}
By \cite[Lemma 2.6]{JLL} there exist sections $\si_1,\si_2 \colon S/T \to S$ such that $\psi_i \circ \si_i=\si_i$ and such that $\psi_i(t\cdot \si_i(a))=t^\zeta\cdot \si_i(a)$ for any $a \in S/T$ and $t \in T$.
For any $x \in S$ set $\tau_i(x)=x\si_i(\overline{x})^{-1}$ and notice that $\tau_i(x) \in T$.
Then,
\[
\psi_i^n(x)=
\psi_i^n(x \si_i(\overline{x})^{-1}\si_i(\overline{x})) =
\tau_i(x)^{\zeta^n} \si_i(\overline{x}).
\]
Set $\de(\overline{x})=\si_1(\overline{x}) \si_2(\overline{x})^{-1}$ and notice that $\de(\overline{x}) \in T$.
In addition $\tau_2(x)^{-1} \tau_1(x) = \de(\overline{x})^{-1}$.
Hence, For any $n>0$
\begin{multline*}
\psi_1^n(x) \cdot \psi_2^n(x)^{-1} =
\tau_1(x)^{\zeta^n} \si_1(\overline{x}) \cdot \si_2(\overline{x})^{-1} \tau_2(x)^{-\zeta^n} =
\tau_1(x)^{\zeta^n} \cdot \delta(\overline{x}) \cdot \tau_2(x)^{-\zeta^n} =
\\
(\tau_2(x)^{-1} \tau_1(x))^{\zeta^n} \de(\overline{x}) =
\de(\overline{x})^{1-\zeta^n}.
\end{multline*}
Since $S/T$ is finite and $T$ is torsion, there exists some $r$ such that $p^r \de(\overline{x}) =1$ for every $x \in S$.
Also, $\zeta \in \Ga_1(p)$, namely $\zeta = 1 \mod p$, so by Fermat's little theorem $\zeta^{p^r}=1 \mod p^r$.
Therefore $\de(\overline{x})^{1-\zeta^{p^r}}=1$ for every $x \in S$, hence $\psi_1^{\zeta^{p^r}}(x)=\psi_2^{\zeta^{p^r}}(x)$.
\end{proof}

\begin{lem}\label{L:N finite normal in G}
Let $N$ be a finite normal subgroup of a group $G$.
Suppose that $g,h \in G$ are in the same coset of $N$.
Then there exists $1 \leq n \leq |N|$ such that $g^n=h^n$.
\end{lem}
\begin{proof}
By hypothesis $g=hx_1$ for some $x_1 \in N$.
Induction on $n$ shows that $(hx_1)^n=h^nx_n$ for some $x_n \in N$.
Since $N$ is finite there are $1 \leq k_1 < k_2 \leq |N|+1$ such that $x_{k_1}=x_{k_2}$, hence $g^{k_2} \cdot g^{-k_1} = h^{k_2}x_{k_2} x_{k_1}^{-1} h^{-k_1}=h^{k_2-k_1}$.
Set $n=k_2-k_1$ and the proof is complete.
\end{proof}

Recall the degree homomorphism $\deg \colon \Ad^{\geom}(\G) \to \ZZ_p^\times/\D(\F)$. (See Proposition \ref{prop_degree_determines_operation}).

\begin{prop}\label{sec_stable}
Let $\G$ be a weakly connected $p$-local compact group.
Let $f_1, f_2 \in \Ad^{\geom}(\G)$ be such that $\deg(f_1)=\deg(f_2)$ in $\ZZ_p^\times/\D(\F)$.
Then $(f_1)^{mk}=(f_2)^{mk}$ where $1 \leq k \leq |\higherlim{\O(\F^c)}{1}\mathcal{Z}|$ and $1 \leq m \leq |H^1(S/T,T)|$.
In fact, $m$ is a divisor of $(p-1)p^r$ for some $r \geq 0$.
\end{prop}

\begin{proof}
Let $[\psi_1]$ and $[\psi_2]$ be the images of $f_1$ and $f_2$ in $\OutAd_{\fus}(S)$ as in Proposition \ref{prop_adgeom_outadfus}.
By assumption $\deg(\psi_1)=\deg(\psi_2) \mod \D(\F)$.
Since $\OutAd_{\fus}(S)=\Aut_{\fus}(S)/(\Aut_\F(S) \cap \Ad(S))$ and $\D(\F)=\deg(\Aut_\F(S) \cap \Ad(S))$, we may choose $\psi_2$ such that $\deg(\psi_1)=\deg(\psi_2)$.
Thus, $[\psi_1] \cdot [\psi_2]^{-1}$ is in the image of $\Ad^{\deg=1}_{\fus}(S)/\Aut_T(S)$ in $\OutAd_{\fus}(S)$ which is a normal finite subgroup by \cite[Proposition 2.8(iii)]{JLL} and Lemma \ref{L:coholology G with T}.
It follows from Lemma \ref{L:N finite normal in G} that there exists $1 \leq m \leq |H^1(S/T,T)|$ such that $[\psi_1]^m=[\psi_2]^m$.
Hence, Proposition \ref{prop_adgeom_outadfus} implies that $f_1^m=f_2^m \mod \higherlim{\O(\F^c)}{1}\mathcal{Z}$.

The group $\higherlim{\O(\F^c)}{1}\mathcal{Z}$ is finite by Proposition \ref{P:finite limits Z}.
%\DEL{the isomorphism $\omega_1$ in Step 2 of the proof of \cite[Theorem 7.1]{BLO3} and by a standard $\Lambda$-functor argument in \cite[Proposition 5.4]{BLO3} and in the proof of \cite[Corollary 5.6]{BLO3} \NEW{ASSAF: CAN WE GIVE A NICER REFERENCE/PROOF?}.}
Lemma \ref{L:N finite normal in G} applied to $f_1^m$ and $f_2^m$ gives $1 \leq k \leq |\higherlim{\O(\F^c)}{1}\mathcal{Z}|$ such that $f_1^{mk}=f_2^{mk}$, as needed.

It remains to show that $m$ is a divisor of $p^r(p-1)$.
Let $\zeta$ denote $\deg(\psi_1)=\deg(\psi_2)$.
Replacing   $\psi_i$ by $\psi_i^{p-1}$, if necessary, we may assume that $\zeta = 1 \mod p$.
If $\zeta=1$, then $[\psi_1], [\psi_2]$ are in the image of $\Ad^{\deg=1}_{\fus}(S)/\Aut_T(S)$ in $\OutAd_{\fus}(S)$ which is finite and abelian by \cite[Prop. 2.8(iii)]{JLL} and Lemma \ref{L:coholology G with T}.
Therefore $m$ can be chosen to be the exponent of $H^1(S/T,T)$ which is a $p$-power.
If $\zeta \neq 1$ then Lemma \ref{L:power Adams aut not 1} shows that $\psi_1^{p^r}=\psi_2^{p^r}$ for some $r \geq 0$ and the proof is complete.
\end{proof}

%%%%%%%%%%%%%%%%%%%%%%%%%%%%%%
%%%%%%%%%%%%%%%%%%%%%%%%%%%%%%
%%%%%%%%%%%%%%%%%%%%%%%%%%%%%%
%%%%%%%%%%%%%%%%%%%%%%%%%%%%%%
%%%%%%%%%%%%%%%%%%%%%%%%%%%%%%
%%%%%%%%%%%%%%%%%%%%%%%%%%%%%%
%%%%%%%%%%%%%%%%%%%%%%%%%%%%%%
%%%%%%%%%%%%%%%%%%%%%%%%%%%%%%
%%%%%%%%%%%%%%%%%%%%%%%%%%%%%%
%%%%%%%%%%%%%%%%%%%%%%%%%%%%%%
%%%%%%%%%%%%%%%%%%%%%%%%%%%%%%
%%%%%%%%%%%%%%%%%%%%%%%%%%%%%%
%%%%%  EXTENSIONS OF CATEGORIES
%%%%%
\section{Extensions of categories}\label{sec_extensions}

In Section \ref{specials} we will study a particularly nice subgroup of unstable Adams operations. To do so, some background on extensions of categories and automorphism of such extensions is required. This is the aim of this section.  We included the definitions and results we need here. The cohomological theory of classification of extensions is carried out in Appendix \ref{appendixA}.   
A large portion of this material is contained in a different form in \cite{Hoff}.

\begin{defn}\label{def:extension-C-by-Phi}
Let $\C$ be a small category and let $\Phi \colon \C \to \Ab$ be a functor.
An \emph{extension} of $\C$ by $\Phi$ is a small category $\D$  with the same object set as that of $\C$, together with a functor $\D \xto{\pi} \C$, and for every $X\in\C$ a ``distinguished'' monomorphisms of groups $\de_X\colon\Phi(X)\to \Aut_\D(X)$, such that the following hold.

\begin{enumerate}
\item
The functor $\pi$ is the identity on the objects  and is surjective on morphism sets. 
Furthermore, for each $X,Y\in\C$, the action of $\Phi(Y)$ on $\Mor_\D(X,Y)$ via $\de_{Y}$ by left composition is free, and the projection 
\[
\Mor_\D(X,Y)\to \Mor_\C(X,Y).
\]
is the quotient map by this action.
\label{ext-axiom-1}

\item
For any $d \in \Mor_\D(X,Y)$ and any $g \in \Phi(X)$ the following square commutes in $\D$.
\[\xymatrix{
X \ar[rr]^d\ar[d]_{\de_{X}(g)}&& Y\ar[d]^{\de_Y(\Phi(\pi(d))(g))} \\
X\ar[rr]^d && Y
}\]
\label{ext-axiom-2}
\end{enumerate}
We will write $\E=(\D,\C,\Phi,\pi,\de)$ for the extension.
\end{defn}

To simplify notation throughout, if $d$ is a morphism in $\D$, then $\pi(d)\in\C$ will be denoted by $[d]$.
If $g\in \Phi(X)$,  we denote $\de_X(g)$ by $\lrb{g}$. 
In addition for any $c \in \C(X,Y)$ we will write $c_* \colon \Phi(X) \to \Phi(Y)$ for the homomorphism $\Phi(c)$.
Thus, the relation in Definition \ref{def:extension-C-by-Phi}\ref {ext-axiom-2} can be written  
\begin{equation}\label{compat}
d\circ\lrb{g} = \lrb{\Phi([d])(g)}\circ d \quad \text{ or simply } \qquad
d \circ \lrb{g} = \lrb{[d]_*(g)} \circ d.
\end{equation}

We remark that the restriction to functors $\Phi \colon \C \to \Ab$ is only made for the sake of simplification.
In fact, we could have considered functors into the category of groups in which case we would have recovered Hoff's results \cite{Hoff} in full generality.

\begin{example}\label{group_ext}
Let $N \xto{i} G \xto{\pi} H$ be an extension of groups with $N$ abelian.
Thus, $N$ becomes an $H$-module.
Every group $\Gamma$ gives rise to a category $\B\Gamma$ with one object whose set of automorphisms is $\Gamma$.
We then obtain an extension of categories $\E=(\B G, \B H, \Phi, \B\pi, \B i)$ where $\Phi \colon \B H \to \Ab$ is the functor representing the $H$-module $N$.
\end{example}

\begin{example}\label{L_ext}
Let $\SFL$ be a $p$-local compact group.
Let $\F^c$ be the full subcategory of $\F$ of the $\F$-centric subgroups.
There is a functor $\zeta \colon (\F^c)^\op  \to\Ab$ taking an object $P$ to its centre $Z(P) = C_S(P)$ (We use the notation $\zeta$ to distinguish it from the functor $\Z$ defined in Section \ref{sec_p-local_compact_groups}).
Also the distinguished homomorphisms $\de_P \colon P \to \Aut_\L(P)$ restrict to $\de_P \colon \Z(P) \to \Aut_{\L^\op}(P)$.
It follows directly from the definitions of linking systems that $\L^{\op}$ is an extension of $(\F^c)^\op$ by $\zeta$ with structure maps $\L^{\op} \xto{\pi^\op} (\F^c)^\op$ and $\de_P \colon \Z(P) \to \Aut_{\L^{op}}(P)$.
Here we used the fact that if $\Gamma$ is an abelian group then $B \Gamma \cong \B \Gamma^{\op}$ via the identity on objects and morphisms.
\end{example}

The next example is the one that the next sections will build on.
Due to its importance we give it as Definition \ref{def_Lred}.
Let $\SFL$ be a $p$-local compact group.
Let $P,Q$ be subgroups of $S$ and suppose that $f \colon P \to Q$ is a homomorphism.
Then $f(P_0) \leq Q_0$ because $f(P_0)$ is a discrete $p$-torus.

%%%
%%%
%%%
\begin{defn}\label{def_Lred}
Let $\SFL$ be a $p$-local compact group.
Define the category $\Lred$ as follows.
First, $\Obj(\Lred)=\Obj(\L)$.
For any $P,Q \in \L$ set
\[
\Mor_{\Lred}(P,Q)= \widehat{Q_0} \backslash \Mor_\L(P,Q).
\]
where $\widehat{Q_0}=\de_Q(Q_0)$ acts on $\Mor_\L(P,Q)$ by post-composition.
We will write $\bar{\vp}$ for the equivalence class (orbit) of $\vp \in \Mor_{\L}(P,Q)$.
We write
\[
\pi\0 \colon \L \to \Lred
\]
for the projection functor.
\end{defn}

This defines a category because for any $P \xto{\al} Q \xto{\be} R$ and any $t \in Q_0$ and $u \in R_0$ we have $\pi(\be)(t) \in Q_0$ so
\[
(\hat{u} \circ \be) \circ (\hat{t} \circ \al) = \hat{u} \cdot \widehat{\pi(\be)(t)} \circ \be \circ \al
= \be \circ \al \mod \widehat{R_0}.
\]

Any homomorphism $f \in \Hom_\F(P,Q)$ restricts to $\vp \colon P_0 \to Q_0$ of the maximal tori.
Also, if $t \in Q_0$  then $c_t$ induces the identity on $Q_0$ and therefore $\vp|_{P_0}^{Q_0} = (c_t \circ \vp)|_{P_0}^{Q_0}$.
This justifies the following definition.

From now on, given a category $\C$ we will use the  notation $\C(x,y)$ for $\Mor_\C(x,y)$.

%%%
%%%
%%%
\begin{defn}\label{def_Phi_L}
Let $\Phi \colon \Lred \to \Ab$ be the functor which on objects is defined by $\Phi \colon P \mapsto P_0$.
For a  morphisms $\bar{\vp} \in \Lred(P,Q)$ set $\Phi(\bar{\vp})=\pi(\vp)|_{P_0}^{Q_0}$.
\end{defn}

%%%
%%%
%%%
\begin{prop}\label{prop_L_extension_of_Lred}
Let $\SFL$ be a $p$-local finite group.
Then $\L$ is an extension  of $\Lred$ by the functor $\Phi \colon \Lred \to \Ab$ (in the sense of Definition \ref{def:extension-C-by-Phi}).
The structure maps are given by $\pi\0 \colon \L \to \Lred$ and $\de_{/0}\defeq\de_P|_{P_0} \colon P_0 \to \Aut_\L(P)$.
\end{prop}

\begin{proof}
The functor $\pi\0\colon\L \to\Lred$ is clearly the identity on objects and is surjective on morphism sets.
The group $\de_Q(Q_0) \leq \Aut_\L(Q)$ acts freely on $\L(P,Q)$ because all morphisms in $\L$  are epimorphisms by \cite[Corollary 1.8]{JLL}.
So condition \ref{ext-axiom-1} of Definition \ref{def:extension-C-by-Phi} holds.
Condition \ref{ext-axiom-2} follows from axiom (C)  of linking systems.
\end{proof}

Next we consider morphisms of extensions.

%%%
%%%
%%%
\begin{defn}\label{def_morphism_of_extensions}
Let $\E=(\D,\C,\Phi,\pi,\de)$ and $\E'=(\D',\C',\Phi',\pi',\de')$ be extensions.
A morphism $\E \to \E'$ is a functor $\Psi \colon \D \to \D'$ such that there exists a functor $\overline{\Psi} \colon \C \to \C'$ which satisfies $\pi' \circ \Psi = \overline{\Psi} \circ \pi$.
\end{defn}

An automorphism of extensions is therefore an isomorphism of categories $\al \colon \D \to \D$ such that both $\al$ and $\al^{-1}$ are morphisms of the extension $\E$.
Among these, there are the inner automorphisms of the extension $\E$.

\begin{defn}\label{def_InnE}
Let $\E=(\D,\C,\Phi,\pi,\de)$ be an extension.
Given a choice of elements $u(X) \in \Phi(X)$ for every $X \in \C$,  we obtain an  automorphism $\tau_u \colon \E \to \E$ where $\tau_u$ is the identity on objects and for any $d \in \D(X,Y)$ we define
\[
\tau_u(d) = \lrb{u(Y)} \circ d \circ \lrb{u(X)}^{-1}.
\]
An automorphism of $\E$ is called \emph{inner} if it is equal to some $\tau_u$.
The collection of all the inner automorphisms of $\E$ is denoted $\Inn(\E)$.
\end{defn}

We remark that the functor $\overline{\Psi} \colon \C \to \C'$ in Definition \ref{def_morphism_of_extensions}, if it exists then it is unique because $\pi$ and $\pi'$ are surjective.
In addition there is no condition on the functors $\Phi$ and $\Phi'$ in the definition of morphisms because of the following lemma.

\begin{lem}\label{lem:def-Psi-star}
Let $\E=(\D,\C,\Phi,\pi,\de)$ and $\E'=(\D',\C',\Phi',\pi',\de')$ be  extensions.
Then any morphism $\Psi\colon \E \to \E'$  gives rise to a unique natural transformation
\[
\eta(\Psi) \colon \Phi \to \Phi' \circ \overline{\Psi}
\]
that assigns to every  object $X\in\C$  the homomorphism $\eta_X\colon\Phi(X)\to\Phi'\circ\overline{\Psi}(X)$, determined by the formula
\begin{equation}\label{eta-def}
\lrb{\eta_X(g)}=\Psi(\lrb{g}),
\end{equation}
for any $g\in \Phi(X)$.
(Here $\eta$ denotes $\eta(\Psi))$.
\end{lem}

\begin{proof}
For every object $X \in \C$ the functors $\Psi$ and $\overline{\Psi}$ give rise to a morphism of exact sequences
\[
\xymatrix{
1 \ar[r] &
\Phi(X)  \ar[r]^{\delta_X} \ar@{.>}_{\eta_X}[d]&
\Aut_\D(X) \ar[r]^{\pi} \ar[d]^{\Psi} &
\Aut_\C(X)  \ar[d]^{\overline{\Psi}} 
%& 
%1
\\
1 \ar[r] &
\Phi'(\overline{\Psi}(X))  \ar[r]^{\delta'_{\overline{\Psi}(X)}} &
\Aut_{\D'}(\overline{\Psi}(X)) \ar[r]^{\pi'} &
\Aut_{\C'}(\overline{\Psi}(X)) 
%& 
%1
}
\]
This defines $\eta_X$, which satisfies \eqref{eta-def} by definition.
It remains to prove that the homomorphisms $\eta_X$ define a natural transformation $\eta \colon \Phi \to \Phi' \circ \overline{\Psi}$.
For any $c \in \Mor_{\C}(X,Y)$, we need to show that 
%\begin{equation}\label{nat}
\[
\xymatrix{
\Phi(X) \ar[rr]^{\Phi(c)} \ar[d]_{\eta_{X}} &&
\Phi(Y) \ar[d]^{\eta_{Y}} 
\\
\Phi'(\overline{\Psi}(X)) \ar[rr]_{\Phi'(\overline{\Psi}(c))} &&
\Phi'(\overline{\Psi}(Y)).
}
\]
%\end{equation}
Choose $d \in \D(X,Y)$ such that $c=[d]$ and let $g \in \Phi(X)$.
By \eqref{eta-def} and \eqref{compat}
\begin{multline*}
\Psi(d) \circ \lrb{\eta_X(g)} = 
\Psi(d) \circ \Psi(\lrb{g}) =
\Psi(d \circ \lrb{g}) =
\Psi(\lrb{\Phi(c)(g)} \circ d) = 
\\
\Psi(\lrb{\Phi(c)(g)}) \circ \Psi(d)=
\lrb{\eta_Y(\Phi(c)(g))} \circ \Psi(d).
\end{multline*}
On the other hand, by applying \eqref{compat} directly to the left hand side of this equality and noticing that $[\Psi(d)]=\overline{\Psi}(c)$ we get
\[
\Psi(d) \circ \lrb{\eta_X(g)} =
\lrb{\Phi'(\overline{\Psi}(c))(\eta_X(g))} \circ \Psi(d).
\]
By comparing the right hand sides of these equalities and using the free action of $\Phi'(Y)$ on $\D'(X,Y)$ where $\Psi(d)$ belongs, we see that $\eta_Y(\Phi(c)(g)) = (\Phi' \circ \overline{\Psi})(c)(\eta_X(g))$ as needed.
\end{proof}

\begin{lem}\label{lem_inner_are_idid}
Let $\E=(\D,\C,\Phi,\pi,\de)$ be an extension.
Then any $\Theta \in \Inn(\E)$  induces the identity on $\C$ and $\eta(\Theta)=\Id$.
\end{lem}

\begin{proof}
Using the notation of Definition \ref{def_InnE} we write $\Theta=\tau_u$.
Then $\tau_u$ induces the identity on $\C$ because $[\lrb{u(Y)} \circ \vp \circ \lrb{u(X)^{-1}}]=[\lrb{u(Y) \cdot \vp_*(u(X)^{-1})} \circ \vp] = [\vp]$.
By Lemma \ref{lem:def-Psi-star}, for any $X \in \C$ and any $x \in \Phi(X)$ we have $\lrb{\eta_X(\tau_u)(x)}=\tau_u(\lrb{x})=\lrb{u(X)} \circ \lrb{x} \circ \lrb{u(X)^{-1}} = \lrb{x}$ because $\Phi(X)$ is abelian.
This shows that $\eta(\tau_u)=\Id$ as needed.
\end{proof}

At this stage it is useful to remark about which functors $\Psi \colon \D \to \D'$ give rise to a morphism of extensions $\E \to \E'$ as in Definition \ref{def_morphism_of_extensions}.

%%%
%%%
%%%
\begin{defn}\label{def_Frigid}
A functor $F \colon \gps \to \gps$ is said to be \emph{inclusive} if for any group $G$, $F(G) \leq G$ and these inclusions is natural with respect to group homomorphism, namely they form a natural transformation of functors $\iota\colon F \to \Id$.
Let $F$ be an inclusive functor. 
An extension $\E=(\D,\C,\Phi,\pi,\de)$ is called \emph{$F$-rigid} if for every $X \in \Obj(\D)$, the injection $\de_X\colon\Phi(X)\to\Aut_\D(X)$ is an isomorphism onto $F(\Aut_\D(X))$.
\end{defn}

Here is an example of an inclusive  functor $F$, as in  Definition \ref{def_Frigid}, which will play a role in this paper.

%%%
%%%
%%%
\begin{defn}\label{def_Lambdagps}
Let $\La \colon \Gps \to \Gps$ be the functor which assigns to every group $G$ the subgroup $\La(G)$ generated by the images of all homomorphisms $\vp \colon \pruffer \to G$.
\end{defn}

Recall that a group $G$ is called virtually discrete $p$-toral if it is an extension of a finite group by a discrete $p$-torus.
In this case $\La(G)=G_0$ is the identity component of $G$ and is a discrete $p$-torus.
Hence, the restriction of $\Lambda$ to the full subcategory of virtually discrete $p$-toral groups factors through the category $\Ab$.

The reason we consider $F$-rigid extensions is that morphisms between them (Definition \ref{def_morphism_of_extensions}) are just functors between the categories.
This is the content of the next proposition.

%%%
%%%
%%%
\begin{prop}\label{prop_functors_of_rigid_extensions}
Suppose that $\E=(\D,\C,\Phi,\pi,\de)$ and $\E'=(\D',\C',\Phi',\pi',\de')$ are $F$-rigid extensions.
Then any functor $\Psi \colon \D \to \D'$ is a morphism of extensions $\E \to \E'$.
\end{prop}

\begin{proof}
For any $C \in \Obj(\D)$ the functor $\Psi$ induces a homomorphism $\Psi \colon \Aut_{\D}(C) \to \Aut_{\D'}(\Psi (C))$.
By applying $F$ and using the natural transformation $\iota\colon F\to \Id$ we get
\[
\Psi(\Phi(C))=\Psi(F(\Aut_\D(C)) \leq F(\Aut_{\D'}(\Psi(C))) = \Phi'(\Psi(C)).
\]
In view of this we define $\psi \colon \C \to \C'$ as follows.
On objects, $\psi \colon C \mapsto \Psi(C)$.
Fix $c \in \C(C_0,C_1)$ and choose a lift $d \in \D(C_0,C_1)$.
Define $\psi \colon c \mapsto [\Psi(d)]$.
Now, $\psi$ is well defined on morphisms since if $d'$ is another lift for $c$ then $d'=\lrb{x} \circ d$ for some $x \in \Phi(C_1)$ and therefore 
\[
[\Psi(\lrb{x} \circ d)]=[\Psi(\lrb{x})] \circ [\Psi(d)] =[\Psi(d)]
\]
because $\Psi(\lrb{x}) \in \Phi'(\Psi(C_1))$ as we have seen above.
The verification that $\psi$ respects identities and compositions is straightforward and the equality  $\pi' \circ \Psi = \psi \circ \pi$ holds by the way we defined $\psi$.
\end{proof}

%%%
%%%
%%%
\begin{prop}\label{prop_sfl_rigid_extension}
Let $\SFL$ be a $p$-local compact group.
The extension $\E=(\L,\Lred,\Phi,\pi\0,\de_{/0})$ in Proposition \ref{prop_L_extension_of_Lred} is $\Lambda$-rigid (Definitions \ref{def_Frigid}, \ref{def_Lambdagps}).
\end{prop}

\begin{proof}
It follows from \cite[Lemma 2.5]{BLO3} that for every $\F$-centric $P \leq S$ the group $\Aut_\L(P)$ is an extension of $P_0$ by a finite group and therefore $\Lambda(\Aut_\L(P))=P_0$.
\end{proof}

We saw (Example \ref{group_ext}) that group extensions are a particular example of extension of categories in the sense described here. Similarly to the case of group extensions there is a theory of  classification of extensions of categories by cohomology. In particular an extension $\E=(\D,\C,\Phi,\pi,\de)$ is classified up to the appropriate concept of equivalence by a class $[\D]\in\H^2(\C, \Phi)$. The theory is rather well known and appears in a different form in \cite{Hoff}. For the convenience of the reader we collect the necessary material in Appendix \ref{appendixA}. 
Here we record only the following two results.

\begin{prop}\label{prop_morphisms_of_extensions}
Let $\E=(\D,\C,\Phi,\pi,\de)$ and $\E'=(\D',\C',\Phi',\pi',\de')$ be extensions. 
Let $\psi \colon \C \to \C'$ be a functor, and let $\eta\colon\Phi\to\Phi'\circ\psi$ be a natural transformation. 
Then the following are equivalent.
\begin{enumerate}[(i)]
\item
There exists a morphism of extensions  $\Psi\colon\E \to \E'$ such that $\psi=\overline{\Psi}$ and $\eta = \eta(\Psi)$.
\label{c_i_morphisms_of_extensions}

\item 
The homomorphisms in cohomology induced by $\psi$ and $\eta$
\[
H^2(\C;\Phi) \xto{\eta_*} H^2(\C;\Phi' \circ \psi) \xleftarrow{\psi^*} H^2(\C',\Phi')
\]
satisfy $\eta_*([\D])=\psi^*([\D'])$.
\label{c_ii_morphisms_of_extensions}
\end{enumerate}
\end{prop}

\begin{prop}\label{prop_h1_and_automorphisms_of_E_copy}
Let $\E=(\D,\C,\Phi,\pi,\de)$ be an extension.
Then there is an isomorphism of groups
\[
\Gamma \colon H^1(\C,\Phi) \to \Aut(\E;1_\C,1_\Phi)/\Inn(\E).
\]
\end{prop}

Proposition \ref{prop_morphisms_of_extensions} will be restated and proved as Proposition \ref{prop_morphisms_of_extensions_apndx}.
Proposition \ref{prop_h1_and_automorphisms_of_E_copy} will be restated and proved as Proposition \ref{prop_h1_and_automorphisms_of_E}.

%%%%%%%%%%%%%%%%%%%%%%%%%%%%%%
%%%%%%%%%%%%%%%%%%%%%%%%%%%%%%
%%%%%%%%%%%%%%%%%%%%%%%%%%%%%%
%%%%%%%%%%%%%%%%%%%%%%%%%%%%%%
%%%%%%%%%%%%%%%%%%%%%%%%%%%%%%
%%%%%%%%%%%%%%%%%%%%%%%%%%%%%%
%%%%%%%%%%%%%%%%%%%%%%%%%%%%%%
%%%%%%%%%%%%%%%%%%%%%%%%%%%%%%
%%%%%%%%%%%%%%%%%%%%%%%%%%%%%%
%%%%%%%%%%%%%%%%%%%%%%%%%%%%%%
%%%%%%%%%%%%%%%%%%%%%%%%%%%%%%
%%%%%%%%%%%%%%%%%%%%%%%%%%%%%%
%%%%%  SPECIAL ADAMS OPERATIONS
%%%%%
%%%%%
\section{Special Adams operations}\label{specials}

This section is dedicated to a particularly nicely behaved family of unstable Adams operations. These operations can be analysed by considering the linking system $\L$ as an extension of the category $\Lred$ (Definition \ref{def_Lred}) by the functor $\Phi$ (Definition \ref{def_Phi_L}), with structure maps $\delta$  and $\delta\0$ (Proposition \ref{prop_L_extension_of_Lred}). In particular we prove Theorem \ref{thm_spad_in_ad}, restated below  as Theorem \ref{thm_spad_in_ad_in_sec}.

While $\de\0 \colon P_0 \to \Aut_\L(P)$ is merely the restriction of $\de \colon P \to \Aut_\L(P)$ to $P_0$  %While the latter is the restriction of the former to maximal tori, 
these morphisms play different roles in their respective contexts. 
To emphasise this, for $x\in P$ and $t\in P_0$, we will denote $\delta(x)$ by $\widehat{x}$  and  $\delta\0(t)$ by $\lrb{t}$. 
 This is consistent with the notation we have established in previous sections.
An equality of the form $\lrb{x} = \wh{x}$ for $x\in P_0$ will simply mean that the corresponding elements in $\Aut_\L(P)$ coincide. 
Notice that $\lrb{x}$ only makes sense when $x \in P_0$ whereas $\widehat{x}$ is defined for any $x\in P$, or indeed, $x \in N_S(P,Q)$.
Also, we will use the symbol $[\vp]$ for the image in $\L\0$ of a morphism $\vp$ in $\L$, and $\pi(\vp)$ for the image of that morphism in $\F$. 

Let $\F$ be a saturated fusion system over $S$.
A set $\R$ of subgroups of $S$ is called an \emph{$\F$-collection} or simply a \emph{collection} if it is closed under conjugacy in $\F$, namely it is the union of isomorphism classes of objects in $\F$.
We will write $\F^\R$ for the full subcategory of $\F$ with object set $\R$.
If $\R \subseteq \F^c$ we let $\L^\R$ be the full subcategory of $\L$ on the object set $\R$.

\begin{defn}\label{def_special_uao}
Let $\G=\SFL$ be a $p$-local compact group and let $\R$ be a collection of $\F$-centric subgroups.
We say that $(\Psi,\psi) \in \Ad(\G)$ is a \emph{special} unstable Adams operation \emph{relative to} $\R$ if there exists a choice of 
\begin{itemize}
\item[(a)] $\tau_P \in T$ for every $P \in \R$ and
\item[(b)] $\tau_\vp \in Q_0$ for every $\vp \in \L^\R(P,Q)$
\end{itemize}
such that the following hold
\begin{enumerate}
\item
$\psi(P)=\tau_P  P \tau_P^{-1}$ for every $P \in \R$ and

\item
$\Psi(\vp)=\widehat{\tau_Q} \circ \widehat{\tau_{\vp}} \circ \vp \circ \widehat{\tau_P}^{-1}$ for every $\vp \in \L^\R(P,Q)$.
\end{enumerate}
The subset of $\Ad(\G)$ of all the special unstable Adams operation relative to $\R$ is denoted $\SpAd(\G;\R)$.
\end{defn}

In \cite{JLL}, given a $p$-local compact group $\G=\SFL$, we find an integer $m \geq 0$ such that every $\zeta \in \Ga_m(p)$ is the degree of some unstable Adams operation which we construct.
The construction, however, involves many choices and the number $m$ is quite mysterious.
It turns out that all the unstable Adams operation we constructed in \cite{JLL} are special relative to the collection $\H^\bullet(\F^c)$ (see Remark \ref{R:bullet collection}).
We will show this in Appendix \ref{app_jll_special}. 

Recall that $\G$ gives rise to an extension of categories $\E=(\L^\R,\Lred^\R,\Phi|_\R,\pi_{/0},\de)$, as in Definitions \ref{def_Lred}, \ref{def_Phi_L} and Proposition \ref{prop_L_extension_of_Lred}.
As in Section  \ref{sec_extensions} we will write $[ \ ]$ instead of $\pi_{/0}$ and $\lrb{~}$ instead of $\de$.
Our aim is to relate special unstable Adams operations relative to $\R$ to automorphisms of the extension $\E$.

Similar to the case of group extensions, $\L^\R$ gives rise to an element of $H^2(\Lred^\R,\Phi)$ as follows.
First, one chooses a section $\si \colon \Mor(\L^\R_{/0})\to \Mor(\L^\R)$ of $\pi_{/0}$ which lifts identity morphisms to identity morphisms.
A $2$-cochain $z_{\si}\in C^*(\Lred^R,\Phi)$ is then defined by the relation \[\si([\vp]) \circ \si([\vp']) =\lrb{z_\si([\vp],[\vp'])} \circ \si([\vp] \circ[\vp']).\]
(See Definition \ref{def_z-sigma}). 
This turns out to be a $2$-cocycle and it defines an element $[\L^\R] \in H^2(\Lred^\R,\Phi)$ which is independent of the choices. (See Definition \ref{def_cohomology_class_of_extn}).

Let $S$ be a ring and let $\xi\in S$ be a central element.
Let $\C$ a small category and $F \colon \C \to \modl{S}$ be a functor.
Then $\xi$ induces a natural transformation $\xi\colon F \to F$ given by taking an object $c\in \C$ to the morphism $F(c) \xto{} F(c)$ given by multiplication by $\xi$. 
We refer to this natural transformation as \emph{multiplication by $\xi$}.

%%%
%%%
%%%
\begin{lem}\label{lem_specials_as_extns_automs}
Let $\G=\SFL$ be a $p$-local compact group and $\R$ a collection.
Let $(\Psi,\psi)$ be a special unstable Adams operation relative to $\R$ of degree $\zeta$.
Let $\E$ denote the extension $(\L^\R,\Lred^\R,\Phi,[-],\lrb{-})$.
Then $\Psi$ restricts to a functor $\Psi|_\R \colon \L^\R \to \L^\R$.
This gives rise to a homomorphism
\[
\res \ \colon \ \SpAd(\G;\R) \xto{ \ (\Psi,\psi) \mapsto \Psi|_\R \ } \Aut(\E).
\]
As an automorphism of the extension $\E$ (see Definition \ref{def_morphism_of_extensions}), $\Psi|_\R \colon \L^\R \to \L^\R$ has the following properties, where in all three statements, $\widebar{\Psi}$ denotes the automorphism of $\L\0^\R$ induced by $\Psi$.
\begin{enumerate}[(1)]
\item
$\Phi \circ \overline{\Psi} = \Phi$,
\label{lem_specials_as_extns_automs_1}

\item
$\eta(\Psi) \colon \Phi \to \Phi \circ \overline{\Psi}=\Phi$ is multiplication by $\zeta$ (see Lemma \ref{lem:def-Psi-star}) , and
\label{lem_specials_as_extns_automs_2}

\item
$\overline{\Psi}^*([\L^\R])=[\L^\R]$ in $H^2(\Lred^\R,\Phi)$,
\label{lem_specials_as_extns_automs_3}
\end{enumerate}
\end{lem}

\begin{proof}
Since $\R$ is closed under $\F$-conjugation, it is invariant under conjugation by $T$.
Since $\Psi$ is special, $\Psi(P)=\psi(P)$ is a $T$-conjugate of $P$ for any $P \in \R$ and therefore $\Psi$ restricts to a functor on the object set $\R \subseteq \Obj(\L)$ and hence gives a functor $\Psi|_\R \colon \L^\R \to \L^\R$ whose inverse is $\Psi^{-1}|_\R$.
Since the extension $\E$ is $\Lambda$-rigid by Proposition \ref{prop_sfl_rigid_extension}, it follows that $\Psi|_R$ is an automorphism of the extension $\E$.
It is clear that the assignment $\res \colon \Psi \mapsto \Psi|_\R$ is a homomorphism (since composition of functors is the group operation).

Now consider some $(\Psi,\psi) \in \SpAd(\G;\R)$.
By definition of the functor $\overline{\Psi} \colon \Lred^\R \to \Lred^\R$ and since $(\Psi,\psi)$ is an Adams operation, for any $P \in \R$ 
\[
\overline{\Psi}(P)=[\Psi(P)]=\psi(P).
\]
Suppose that $P,Q \in \R$ and that $\vp \in \L(P,Q)$, and consider its image $[\vp] \in \Lred^\R(P,Q)$.
Then by definition of $\overline{\Psi}$
\[
\overline{\Psi}([\vp])=[\Psi(\vp)].
\]
\noindent
{\it Proof of (\ref{lem_specials_as_extns_automs_1})}.
Let $\pi\colon\L\to\F$ denote the projection.
Since $\psi$ is an Adams automorphism of $S$ then $\psi|_T$ is multiplication by $\zeta$, hence it leaves every subgroup of $T$ invariant.
Also note that if $P \in \R$ then $P_0$ is a characteristic subgroup of $P$.
This shows that
\[
\Phi(\overline{\Psi}(P)) =\overline{\Psi}(P)_0 = \psi(P)_0=\psi(P_0)=P_0=\Phi(P).
\]
So $\Phi \circ \overline{\Psi}$ and $\Phi$ attain the same values on objects.
Now suppose that $P \xto{[\vp]} Q$ is a  morphism in $\Lred$ where $P,Q \in \R$ and $\vp \in \L(P,Q)$.
Notice that $\pi(\vp)(P_0)$ is a discrete $p$-torus and it is therefore a subgroup of $T$.
Since $(\Psi,\psi)$ is a special unstable Adams operation relative to $\R$, and since $\psi(P)_0=P_0$
\begin{multline*}
\Phi(\overline{\Psi}([\vp])) \overset{\eqref{def_morphism_of_extensions}}{=} 
\Phi([\Psi(\vp)]) \overset{\eqref{def_special_uao}}{=}
\Phi([\widehat{\tau_Q} \circ \widehat{\tau_\vp} \circ \vp \circ \widehat{\tau_P}^{-1}]) 
\\
\overset{\eqref{def_Phi_L}}{=}
c_{\tau_Q} \circ c_{\tau_\vp} \circ \pi(\vp) \circ c_{\tau_P}{}^{-1}|_{P_0} \overset{}{=} 
\pi(\vp)|_{P_0} \overset{\eqref{def_Phi_L}}{=} 
\Phi([\vp]),
\end{multline*}
where each equality follows from the  definition indicated above it, and the fourth equality holds since $T$ is abelian. 
This shows that $\Phi \circ \overline{\Psi}=\Phi$.

\noindent
{\it Proof of (\ref{lem_specials_as_extns_automs_2})}.
Fix some $P \in \R$ and $x \in \Phi(P_0)=P_0$. 
Then, since $\psi|_T$ is multiplication by $\zeta$, one has
\[
\lrb{\eta(\Psi)(x)} =
\Psi(\lrb{x})  =
\Psi(\widehat{x}) =
\widehat{\psi(x)} =
\widehat{\zeta \cdot x} =
\lrb{\zeta \cdot x},
\]
where the first quality follows from Lemma \ref{lem:def-Psi-star}, and the third from Definition \ref{def_algebraic_unstable_adams_operation}. 
Thus $\eta(\Psi)$ is multiplication by $\zeta \in \ZZ_p$.

\noindent
{\it Proof of (\ref{lem_specials_as_extns_automs_3})}.
Let $\{\tau_P\}_{P \in \R}$ and $\{ \tau_\vp \}_{\vp \in \Mor(\L^\R)}$ be as in Definition \ref{def_special_uao}.
Choose a regular section 
\[
\si \colon \Mor(\Lred^\R) \to \Mor(\L^\R).
\]
Define a $1$-cochain $t \in C^1(\Lred^\R,\Phi)$ by setting 
\[
t(c) \defeq \tau_{\si(c)}, \qquad c \in \Mor(\Lred^\R).
\]
Observe that for every $c \in \Lred^\R(P,Q)$,
\[
[\Psi(\si(c))] =
\overline{\Psi}([\si(c)]) =
\overline{\Psi}(c) 
\]
and therefore there exists a unique element $v(c) \in Q_0$ such that
\begin{equation}\label{specials_as_extns_automs_eq1}
\si(\overline{\Psi}(c))=\lrb{v(c)} \circ \Psi(\si(c)).
\end{equation}
We obtain a $1$-cochain $v \in C^1(\Lred^\R,\Phi)$.
Recall that we use multiplicative notation for the group operation in $\Phi$ and hence in the cochain complex  $C^*(\Lred^\R,\Phi)$.
Set 
\[
u = v \cdot t \in C^1(\Lred^\R,\Phi).
\]
Since $\Psi$ is special relative to $\R$, for any $c \in \Lred^\R(P,Q)$,
\[
\Psi(\si(c)) = \widehat{\tau_Q} \circ \lrb{t(c)} \circ \si(c) \circ \widehat{\tau_P}^{-1}.
\]
Together with \eqref{specials_as_extns_automs_eq1} and the commutativity of $T$ we obtain
\begin{equation}\label{specials_as_extns_automs_eq2}
\si(\overline{\Psi}(c))=
\lrb{v(c)} \circ \Psi(\si(c)) =
\widehat{\tau_Q} \circ \lrb{u(c)} \circ \si(c) \circ \widehat{\tau_P}^{-1}.
\end{equation}
Let $P \xto{c_0} Q \xto{c_1} R$ be a $2$-chain in $\Lred^\R$.
By definition of $z_\si$ (Definition \ref{def_z-sigma}) and the functoriality of $\overline{\Psi}$,
\begin{equation}\label{specials_as_extns_automs_eq3}
\lrb{z_\si(\overline{\Psi}(c_1),\overline{\Psi}(c_0))} \circ \si(\overline{\Psi}(c_1 \circ c_0)) = \si(\overline{\Psi}(c_1)) \circ \si(\overline{\Psi}(c_0)).
\end{equation}
Thus we obtain the following sequence of equalities, where the first equality follows from \eqref{specials_as_extns_automs_eq2} applied to both sides of  \eqref{specials_as_extns_automs_eq3}, the third by Definition \eqref{def_Phi_L} and Axiom (C)  of linking systems, and the fourth by Definition \eqref{def_z-sigma}.
\begin{eqnarray*}
&&\lrb{z_\si\Big(\overline{\Psi}(c_1),\overline{\Psi}(c_0)\Big)} \circ \widehat{\tau_R} \circ \lrb{u(c_1 \circ c_0)} \circ \si(c_1 \circ c_0) \circ \widehat{\tau_P}^{-1} =
\\
&& \qquad 
\widehat{\tau_R} \circ \lrb{u(c_1)} \circ \si(c_1) \circ \widehat{\tau_Q}^{-1} \circ \widehat{\tau_Q} \circ \lrb{u(c_0)} \circ \si(c_0) \circ \widehat{\tau_P}^{-1}  = \\
&& \qquad
\widehat{\tau_R} \circ \lrb{u(c_1)} \circ \si(c_1) \circ \lrb{u(c_0)} \circ \si(c_0) \circ \widehat{\tau_P}^{-1}   =  \\  
&& \qquad
\widehat{\tau_R} \circ \lrb{u(c_1)} \circ \lrb{\Phi(c_1)(u(c_0))} \circ    \si(c_1) \circ \si(c_0) \circ \widehat{\tau_P}^{-1} = \\
&& \qquad
\widehat{\tau_R} \circ \lrb{u(c_1) \cdot \Phi(c_1)(u(c_0))} \circ  \lrb{z_\si(c_1,c_0)}  \circ \si(c_1 \circ c_0) \circ \widehat{\tau_P}^{-1}.
\end{eqnarray*}

Now, $\widehat{\tau_P} \colon P \to \psi(P)$ and $\widehat{\tau_R}\colon R\to \psi(R)$ are both  isomorphisms, and $\si(c_1\circ c_0)$ is an epimorphism in $\L$ by \cite[Corollary 1.8]{JLL}. Hence, 
\[\lrb{z_\si\Big(\overline{\Psi}(c_1),\overline{\Psi}(c_0)\Big)  \cdot u(c_1 \circ c_0)} = 
\lrb{u(c_1) \cdot \Phi(c_1)(u(c_0))\cdot  z_\si(c_1,c_0)},\]
 and so we deduce that
\[
z_\si(\overline{\Psi}(c_1),\overline{\Psi}(c_0)) = z_\si(c_1,c_0) \cdot u(c_1) \cdot u(c_1 \circ c_0)^{-1} \cdot \Phi(c_1)(u(c_0)) = 
z_\si(c_1,c_0) \cdot \de(u)(c_1,c_0),
\]
where $\de$ is the differential in $C^*(\Lred^\R,\Phi)$.
This shows that $z_\si$ and $\overline{\Psi}^*(z_\si)$ are cohomologous.
\end{proof}

%%%%%%%%%
\begin{prop}\label{prop_special_and_class_of_L}
Let $\G=\SFL$ be a $p$-local compact group and $\R \subseteq \F^c$ a collection.
If $(\Psi,\psi) \in \SpAd(\G;\R)$ is of degree $\zeta$ then $\zeta \cdot [\L^\R]=[\L^\R]$ in $H^2(\Lred^\R,\Phi)$.
\end{prop}
\begin{proof}
Proposition \ref{prop_morphisms_of_extensions} and Lemma \ref{lem_specials_as_extns_automs} show that $[\L^\R]=\overline{\Psi}^*([\L^\R])=\eta(\Psi)_*([\L^\R])=\zeta \cdot[\L^\R]$.
\end{proof}

Next we turn to a deeper analysis of the group $\SpAd(\G;\R)$. This requires some preparation.

\begin{lem}[{\cite[Proposition 1.14]{JLL}}]\label{lem _specials_for_large_R}
Let $\G=\SFL$ be a $p$-local compact group and let $\R$ be a collection which contains $\H^\bullet(\F^c)$.
Let $\psi \colon S \to S$ be a fusion preserving Adams automorphism.
Then any functor $\Psi' \colon \L^\R \to \L^\R$ which covers $\psi$ in the sense of Definition \ref{def_algebraic_unstable_adams_operation}, extends uniquely to an unstable Adams operation $(\Psi,\psi)$.
\end{lem}

%%%
%%%
%%%
\begin{defn}\label{def_autE1Z}
Let $\G=\SFL$ be a $p$-local compact group and $\R \subseteq \F^c$ a collection.
Let $\E$ denote the extension $(\L^\R,\Lred^\R,\Phi,[-],\lrb{-})$.
Let $Z \subseteq \ZZ_p^\times$ be a subgroup.
Let
\[
\Aut(\E;\Id_{\Lred^\R},Z)\le \Aut(\E)
\] 
denote the subgroup of the automorphisms $\Theta$, such that $\overline{\Theta}=\Id_{\Lred^\R}$ and $\eta(\Theta)$ is multiplication by some $\zeta \in Z$.
\end{defn}

Observe that $\Aut(\E;\Id_{\Lred^\R},Z)$  is indeed a subgroup, since for any $P \in \R$ and any $x \in P_0$, the definition of $\eta(-)$ implies
\[
\lrb{\eta(\Theta_1 \circ \Theta_2)(x)} =
(\Theta_1 \circ \Theta_2)(\lrb{x}) =
\Theta_1(\lrb{\eta(\Theta_2)(x)}) = \lrb{\eta(\Theta_1)(\eta(\Theta_2)(x))},
\]
and so $\eta(\Theta_1 \circ \Theta_2)=\eta(\Theta_1) \circ \eta(\Theta_2)$. Notice also that for any $\Theta \in \Aut(\E;\Id_{\Lred^\R},\ZZ_p^\times)$ and any $\vp \in \L^\R(P,Q)$ one  has $[\Theta(\vp)]=[\vp]$.

%%%
%%%
%%%
\begin{prop}\label{prop_special_adams_operations_and_extens_automs}
Let $\G=\SFL$ be a $p$-local compact group and $\R$ be a collection which contains $\H^\bullet(\F^c)$.
Let $\E$ denote the extension $(\L^\R,\Lred^\R,\Phi,[-],\lrb{-})$.
Then
\begin{enumerate}[(i)]
\item
There exists a homomorphism
\[
\Aut(\E;\Id_{\Lred^\R},\ZZ_p^\times) \xto{ \ \ \rho \ \ } \SpAd(\G;\R)
\]
such that $\deg(\rho(\Theta))=\eta(\Theta) \in \ZZ_p^\times$ (compare Definition \ref{def_autE1Z}).
\label{item_i_existence_of_special_uAos}

\item
Moreover, composition of $\rho$ with the quotient by $\Inn_T(\G)$ gives a surjective homomorphism
\[
\Aut(\E;\Id_{\Lred^\R},\ZZ_p^\times) \xto{ \ \ \bar{\rho} \ \ } \SpAd(\G;\R)/\Inn_T(\G)
\]
whose kernel contains $\Inn(\E)$ (compare Definition \ref{def_InnE} and Lemma \ref{lem_inner_are_idid}).
\label{item_ii_existence_of_special_uAos}
\end{enumerate}
\end{prop}

\begin{proof}
For each $\Theta \in \Aut(\E;\Id_{\Lred^\R},\ZZ_p^\times)$, fix the following elements:
\begin{enumerate}[(i)]
\item For any $\vp \in \L^\R(P,Q)$ let $\tau_\vp(\Theta)$ be the unique element of $Q_0$ such that
\[
\Theta(\vp)= \lrb{\tau_\vp(\Theta)} \circ \vp.
\] \label{def_tau_theta_i}
\item For every $P \in \R$ set $\tau_P(\Theta) \defeq \tau_{\iota_P^S}(\Theta)$.
\label{def_tau_theta_ii}
\end{enumerate}
Notice that $\tau_P(\Theta) \in T$ for all $P\in\R$.
Consider $\Theta \in \Aut(\E;\Id_{\Lred^\R},\ZZ_p^\times)$ such that $\eta(\Theta)=\zeta$.
Set $\tau_P=\tau_P(\Theta)$ and $\tau_\vp=\tau_\vp(\Theta)$ for short, where $\tau_P(\Theta)$ and  $\tau_\vp(\Theta)$ are as in (\ref{def_tau_theta_i}) and (\ref{def_tau_theta_ii}) above.
Notice that $\tau_S=1$ since $\iota_S^S=\id_S$ in $\L$.

For any $x \in P$, where $P \in \R$, consider $\widehat{x} \in \Aut_\L(P)$.
Then $[\Theta(\widehat{x})]=\overline{\Theta}([\widehat{x}])=[\widehat{x}]$ , so $\Theta(\widehat{x}) \in \widehat{P} \leq \Aut_\L(P)$.
By identifying $P$ with $\widehat{P}$ via $\de_P$, this shows that for every $P \in \R$, the functor $\Theta$ induces an automorphism $\theta_P \in \Aut(P)$ by the equation
\[
\widehat{\theta_P(x)}=\Theta(\widehat{x}), \qquad (\forall x \in P).
\]
Notice that $S \in \R$, and we set
\[
\psi \defeq \theta_S.
\]
Consider some $P \in \R$ and $x \in P$.
By applying $\Theta$ to the equality $\iota_P^S \circ \de_P(x) = \de_S(x) \circ \iota_P^S$ we obtain $\widehat{\tau_P} \circ \widehat{\theta_P(x)} = \widehat{\psi(x)} \circ\widehat{\tau_P}$ in $\L$.
Therefore %we get the following relation between $\psi, \theta_P$ and $\tau_P$:
\begin{equation}\label{eq_theta_existence_of_special_uAos}
\psi(x) = (c_{\tau_P} \circ \theta_P)(x), \qquad (x \in P).
\end{equation}
In particular it follows that
\[
\tau_P \cdot P \cdot \tau_P^{-1} =  \psi(P),
\]
i.e, $\tau_P\in N_T(P,\psi(P))$.

\noindent
{\bf  Step 1. }(\emph{$\psi$ is a normal Adams automorphism of degree $\zeta$}).
Notice first that for any $x \in T$ we have the following equalities in $\Aut_\L(S)$
\[
\widehat{\psi(x)} = \Theta(\widehat{x})=\Theta(\lrb{x}) \overset{\eqref{compat}}{=} \lrb{\eta(\Theta)_S(x)} = \lrb{\zeta \cdot x}= \widehat{\zeta \cdot x}.
\]
Therefore $\psi|_T$ is multiplication by $\zeta$.
Also, the image of $S$ in $\Aut_{\Lred}(S)$ is exactly $S/T$ and since $\overline{\Theta}=\Id$, it follows that $\psi$ induces the identity on $S/T$.
Hence $\psi$ is a normal Adams automorphism, as claimed.

\noindent
{\bf Step 2. } ($\psi$ is fusion preserving.)
Since $\R \supseteq \H^\bullet(\F^c)$, it controls fusion.
Therefore, it is enough to show that for any $P \in \R$ and any $f \in \Hom_\F(P,S)$ there exists $g \in \Hom_\F(\psi(P),S)$ such that $\psi \circ f = g \circ \psi|_P^{\psi(P)}$.
Let $\vp \in \L(P,S)$ be a lift for $f$.
By axiom (C) of linking systems, for every $x \in P$ we have $\vp \circ \widehat{x} = \widehat{f(x)} \circ \vp$.
By applying $\Theta$,
\begin{equation}\label{step2.1}
\Theta(\vp) \circ \widehat{\theta_P(x)} = \widehat{\psi(f(x))} \circ \Theta(\vp).
\end{equation}
Set $f'=\pi(\Theta(\vp))$.
By Axiom (C) and  \eqref{eq_theta_existence_of_special_uAos},
\begin{equation}\label{step2.2}
\Theta(\vp) \circ \widehat{\theta_P(x)} = f'(\theta_P(x)) \circ \Theta(\vp) = \widehat{f'(c_{\tau_P}{}^{-1}(\psi(x)))} \circ \Theta(\vp).
\end{equation}
Comparing the right hand sides of (\ref{step2.1}) and (\ref{step2.2}), and using the fact that $\Theta(\vp)$ is an epimorphism in $\L$ by \cite[Corollary 1.8]{JLL},
\[
\psi(f(x))=f'(c_{\tau_P}{}^{-1}(\psi(x)))
\]
for every $x \in P$.
Set $g=f' \circ c_{\tau_P}{}^{-1}|_{\psi(P)}^P$.
Then $g \in \Hom_\F(\psi(P),S)$ and $\psi \circ f = g \circ \psi|_P^{\psi(P)}$.
This shows that $\psi$ is fusion preserving and completes Step 2.

Define a functor $\Psi \colon \L^\R \to \L^\R$, on objects  $P, Q \in \R$ and morphisms $\vp \in \L^\R(P,Q)$, by
\begin{eqnarray}
&& \Psi(P) \defeq \psi(P) \label{def_rho_1}
\label{def psi from theta 1}\\
&& \Psi(\vp) \defeq \widehat{\tau_Q} \circ \Theta(\vp) \circ \widehat{\tau_P}^{-1} \label{def_rho_2},
\label{def psi from theta 2}
\end{eqnarray}
where $\tau_P$ and $\tau_Q$ are considered here as elements of $N_T(P,\psi(P))$ and $N_T(Q,\psi(Q))$ respectively.
The functoriality of $\Psi$ is clear from that of $\Theta$.
In fact, the morphisms $\widehat{\tau_P} \in \L(P,\psi(P))$ give a natural isomorphism
\[
\widehat{\tau} \colon \Theta \to \Psi.
\]

\noindent
{\bf Step 3. } (\emph{$\Psi$ covers $\psi$}).
First, we  show that $\pi \circ \Psi = \psi_* \circ \pi$.
The functors on both sides agree on objects, by definition of $\Psi$ and $\psi_*$, and since the projection $\pi$ is the identity on objects.
Let  $\vp \in \L^\R(P,Q)$ be a morphism. Consider the squares:

\[
\xymatrix{
P \ar[r]^{\vp} \ar[d]_{\widehat{x}} & Q \ar[d]^{\widehat{\pi(\vp)(x)}} 
\\
P \ar[r]_\vp & Q 
}
\qquad\qquad
\xymatrix{
P \ar[r]^{\Theta(\vp)} \ar[d]_{\widehat{\theta_P(x)}} & Q \ar[d]^{\widehat{\theta_Q(\pi(\vp)(x))}} 
\\
P \ar[r]_{\Theta(\vp)} & Q 
}
\]
The left square  commutes by Axiom (C), and the right one is obtained from it by applying $\Theta$.
Since $\Theta(\vp)$ is an epimorphism in $\L$, Axiom (C) applied to the right square gives
\begin{equation}\label{eq_thetapq_existence_of_special_uAos}
\theta_Q(\pi(\vp)(x))=\pi\big(\Theta(\vp)\big)\,(\theta_P(x)).
\end{equation}
Therefore, for any $x \in P$
\begin{multline*}
\pi\big(\Psi(\vp)\big)(\psi(x)) = 
\pi\big(\widehat{\tau_Q} \circ \Theta(\vp) \circ \widehat{\tau_P}^{-1}\big) (\psi(x)) =
(c_{\tau_Q} \circ \pi(\Theta(\vp)) \circ c_{\tau_P}^{-1})(\psi(x)) \overset{\eqref{eq_theta_existence_of_special_uAos}}{=}
\\
(c_{\tau_Q} \circ \pi(\Theta(\vp)))(\theta_P(x))  \overset{\eqref{eq_thetapq_existence_of_special_uAos}}{=}
c_{\tau_Q}(\theta_Q(\pi(\vp)(x)))  \overset{\eqref{eq_theta_existence_of_special_uAos}}{=}
\psi(\pi(\vp)(x))
\end{multline*}
Thus, $\pi(\Psi(\vp)) \circ \psi|_P = \psi \circ \pi(\vp)$ and consequently $(\pi \circ \Psi)(\vp)=(\psi_* \circ \pi)(\vp)$ as claimed.

Next we show that for any $P,Q \in \R$ and $g \in N_S(P,Q)$ we have $\Psi(\widehat{g})=\widehat{\psi(g)}$.
First, notice that $\psi(g) \in N_S(\psi(P),\psi(Q))$.
Next,
\begin{align*}
\iota_{\psi(Q)}^S \circ \Psi(\widehat{g}) & = & \text{by definition of $\Psi$ \eqref{def psi from theta 2} } \\
\iota_{\psi(Q)}^S \circ \widehat{\tau_Q} \circ \Theta(\widehat{g}) \circ \widehat{\tau_P}^{-1} &=& \text{by definition of $\tau_Q$}\\
\Theta(\iota_Q^S) \circ \Theta(\widehat{g}) \circ \widehat{\tau_P}^{-1} & =\\
\Theta(\iota_Q^S \circ \widehat{g}) \circ \widehat{\tau_P}^{-1} & = \\
\Theta(\widehat{g} \circ \iota_P^S) \circ \widehat{\tau_P}^{-1} & = \\
\Theta(\widehat{g}) \circ \Theta(\iota_P^S) \circ \widehat{\tau_P}^{-1}  &=&  \text{by definition of $\theta_S$}\\
\widehat{\theta_S(g)} \circ \Theta(\iota_P^S) \circ \widehat{\tau_P}^{-1} &=& \text{by definition of $\tau_P$}\\
\widehat{\theta_S(g)} \circ \iota_{\psi(P)}^S & = \\
\iota_{\psi(Q)}^S \circ \widehat{\psi(g)}.
\end{align*}
Since $\iota_{\psi(Q)}^S$ is a monomorphism in $\L$ it follows that $\Psi(\widehat{g})=\widehat{\psi(g)}$.
This shows that $\Psi$ covers $\psi$ (see Definition \ref{def_algebraic_unstable_adams_operation}) and completes the proof of Step 3.

\noindent
{\bf Step 4.} (\emph{Proof of (\ref{item_i_existence_of_special_uAos})}). By Lemma \ref{lem _specials_for_large_R}, $\Psi$ extends to an unstable Adams operation $(\rho(\Theta),\psi)$. 
We retain the notation $\Psi = \rho(\Theta)$ for convenience. Then $\Psi$ 
is special relative to $\R$ because for any $P \in \R$,
\[\psi(P) = c_{\tau_P}(\theta_P(P))=c_{\tau_P}(P)=\tau_P P \tau_P^{-1},
\]
where the first and second equality follow from \eqref{eq_theta_existence_of_special_uAos},
and for any   $\vp \in \L^\R(P,Q)$,
\[ \Psi(\vp) = 
\widehat{\tau_Q} \circ \Theta(\vp) \circ \widehat{\tau_P}^{-1} =
\widehat{\tau_Q} \circ \lrb{\tau_\vp} \circ \vp \circ \widehat{\tau_P}^{-1},
\]
where the second equality follows from the definition of $\tau_\vp$ and \eqref{def psi from theta 2}.
It has degree $\zeta$ because for any $x \in T$,
\[
\lrb{\psi(x)} = \lrb{\theta_S(x)} = \Theta(\lrb{x}) = \eta(\Theta)(\lrb{x}) = \lrb{\zeta \cdot x}.
\]

It remains to prove that $\rho$ is a homomorphism.
Choose $\Theta, \Theta' \in \Aut(\E;\Id_{\Lred},\ZZ_p^\times)$ and set $\Theta''=\Theta' \circ \Theta$.
Let $\Psi, \Psi', \Psi''$ denote their images in $\Ad(\G;\R)$ by $\rho$.
We need to show that $\Psi''=\Psi' \circ \Psi$.
Keeping the notation above, if $P \in \R$ then it is clear that $\theta_S'' = \theta_S' \circ \theta_S$, namely $\psi''=\psi'\circ \psi$ and therefore
\[
\Psi''(P)=\psi''(P)=\psi'(\psi(P))=(\Psi' \circ \Psi)(P).
\]
Hence $\Psi''$ and $\Psi' \circ \Psi$ agree on objects.
Moreover, observe that by the definition of the elements $\tau_P, \tau'_P, \tau''_P$ of $T$
\[
\iota_{\psi''(P)}^S \circ \widehat{\tau_{\psi(P)}'} \circ \Theta'(\widehat{\tau_P}) =
\Theta'(\iota_{\psi(P)}^S) \circ \Theta'(\widehat{\tau_P}) =
\Theta'(\iota_{\psi(P)}^S \circ \widehat{\tau_P}) =
\Theta'(\Theta(\iota_P^S)) =
\Theta''(\iota_P^S)  =
\iota_{\psi''(P)}^S \circ \widehat{\tau_P''}.
\]
Since $\iota_{\psi''(P)}^S$ is a monomorphism in $\L$ it follows that
\[
\widehat{\tau_{\psi(P)}} \circ \Theta'(\widehat{\tau_P})=\widehat{\tau_P''}.
\]
Now suppose that $P,Q \in \R$ and that $\vp \in \L(P,Q)$.
Then
\begin{multline*}
\Psi'(\Psi(\vp))=
\Psi'(\widehat{\tau_Q} \circ \Theta(\vp) \circ \widehat{\tau_P}^{-1}) =
\widehat{\tau_{\psi(Q)}'} \circ \Theta'(\widehat{\tau_Q} \circ \Theta(\vp) \circ \widehat{\tau_P}^{-1}) \circ \widehat{\tau_{\psi(P)}'}^{-1} =
\\
\widehat{\tau_{\psi(Q)}'} \circ \Theta'(\widehat{\tau_Q}) \circ \Theta'(\Theta(\vp)) \circ \Theta'(\widehat{\tau_P}^{-1}) \circ \widehat{\tau_{\psi(P)}'}^{-1} =
\widehat{\tau_{Q}''} \circ \Theta'(\Theta(\vp)) \circ \widehat{\tau_{P}''}^{-1} = 
\Psi''(\vp).
\end{multline*}
This shows that $\Psi''$ and $\Psi'\circ \Psi$ agree on morphisms and completes the proof of (\ref{item_i_existence_of_special_uAos}).

\noindent
{\bf Step 5.} (\emph{Proof of (\ref{item_ii_existence_of_special_uAos})}). Suppose that $(\Psi,\psi) \in \SpAd(\G;\R)$ has degree $\zeta$.
For each $P\in\R$ and each $\vp\in\Mor(\L^\R)$, fix elements $\tau_P$ and $\tau_\vp$, as in Definition \ref{def_special_uao}.
Define a functor $\Theta \colon \L^\R \to \L^\R$ by 
\begin{eqnarray*}
&& \Theta(P) = P, \qquad (P \in \R) \\
&& \Theta(\vp)=\widehat{\tau_Q}^{-1} \circ \Psi(\vp) \circ \widehat{\tau_P}, \qquad (\vp \in \L^\R(P,Q)).
\end{eqnarray*}
The functoriality of $\Theta$ is clear, and it is a morphism of extensions by Propositions \ref{prop_functors_of_rigid_extensions} and \ref{prop_sfl_rigid_extension}.
Since $(\Psi,\psi)$ is special relative to $\R$, 
%and since $T$ is abelian 
\[
[\Theta(\vp)] = [\widehat{\tau_Q}^{-1}] \circ [\widehat{\tau_Q} \circ \widehat{\tau_\vp} \circ \vp \circ \widehat{\tau_P}^{-1}] \circ [\widehat{\tau_P}] =[\vp],
\]
and so $\overline{\Theta}=\Id_{\Lred^\R}$.
Also, given $P \in \R$ and $g \in P_0$, since $T$ is abelian we get
\[
\Theta(\lrb{g})=\Theta(\widehat{g})=
\widehat{\tau_P}^{-1} \circ \Psi(\widehat{g}) \circ \widehat{\tau_P} =
\widehat{\tau_P}^{-1} \circ \widehat{\psi(g)} \circ \widehat{\tau_P} =
\lrb{\psi(g)} =\lrb{\zeta \cdot g}.
\]
Hence $\eta(\Theta)=\zeta$ by Lemma \ref{lem:def-Psi-star}.
Thus, $\Theta \in \Aut(\E;\Id_{\Lred^\R};\ZZ_p^\times)$.

It remains to show that $\rho(\Theta)$ and $(\Psi,\psi)$ differ by $c_{\widehat{\tau_S}} \in \Inn_T(\G)$.
Observe first that $\tau_P(\Theta)$ is the unique element satisfying $\Theta(\iota_P^S) = \lrb{\tau_P(\Theta)}\circ\iota_P^S$. On the other hand, 
\[
\Theta(\iota_P^S) = \widehat{\tau_S}^{-1} \circ \Psi(\iota_P^S) \circ \widehat{\tau_P} =
\widehat{\tau_S}^{-1} \circ \iota^S_{\psi(P)} \circ \widehat{\tau_P} = \widehat{\tau_S^{-1} \tau_P} \circ \iota_P^S.
\]
By comparing the two expressions for $\Theta(\iota_P^S)$ and since $\iota_P^S$ is an epimorphism in $\L$ we deduce that 
\[
\tau_P(\Theta)=\tau_S^{-1}\tau_P.
\]

%\[\iota_P^S = \wh{\tau_P(\Theta)}^{-1} \circ\wh{\tau_S}^{-1} \circ\iota_{\psi(P)}^S \circ\wh{\tau_P} = 
%\wh{\tau_P(\Theta)}^{-1}\wh{\tau_S}^{-1} \circ\wh{\tau_P} \circ\wh{\tau_P}^{-1} \circ\iota_{\psi(P)}^S \circ\wh{\tau_P} = \wh{\tau_P(\Theta)}^{-1} \circ\wh{\tau_S}^{-1} \circ\wh{\tau_P} \circ\iota_P^S,\]
% and since all morphisms in $\L$ are surjective, $\tau_P(\Theta)=\tau_S^{-1} \tau_P$.}
By \eqref{eq_theta_existence_of_special_uAos}, \eqref{def_rho_1} and \eqref{def_rho_2} it follows that 
\begin{eqnarray*}
&& \rho(\Theta)(P) = \tau_P(\Theta) \cdot P \cdot \tau_P(\Theta)^{-1} = \tau_S^{-1} \tau_P \cdot P \cdot \tau_P^{-1} \tau_S  = \tau_S^{-1} \cdot \psi(P) \cdot \tau_S \\
&& \rho(\Theta)(\vp) = \widehat{\tau_Q(\Theta) }\circ \Theta(\vp) \circ \widehat{\tau_P(\Theta)}^{-1} =
\widehat{\tau_S^{-1}} \circ \widehat{\tau_Q} \circ \Theta(\vp) \circ \widehat{\tau_P}^{-1} \circ \widehat{\tau_S} =
\widehat{\tau_S}^{-1} \circ \Psi(\vp) \circ \widehat{\tau_S}.
\end{eqnarray*}
Therefore $\Psi = c_{\widehat{\tau_S}^{-1}} \circ \rho(\Theta)$, see Definition \ref{def_inner},  and this shows that $\bar{\rho}$ is surjective.

Finally, consider some $\Theta \in \Inn(\E)$.
Then for every $P, Q \in \R$ and any $\vp \in \L^\R(P,Q)$, there is an element  $t_P \in P_0$, such that $\Theta(\vp)=\widehat{t_Q}^{-1} \circ \vp \circ \widehat{t_P}$.
In particular $\tau_P(\Theta)=t_S^{-1}t_P$.
Then for any $\vp \in \L^\R(P,Q)$
\[
\rho(\Theta)(\vp) = 
\widehat{\tau_Q(\Theta)} \circ \Theta(\vp) \circ \widehat{\tau_P(\Theta)}^{-1} = 
\wh{t_S^{-1}t_Q} \circ (\wh{t_Q}^{-1} \circ \vp \circ \wh{t_Q}) \circ \wh{t_P^{-1}t_S} =
\wh{t_S}^{-1}\circ\vp\circ\wh{t_S}.
\]
This shows that $\rho(\Theta)|_\R=c_{\widehat{t_S}^{-1}}$ and by Lemma \ref{lem _specials_for_large_R}, $\rho(\Theta)=c_{\widehat{t_S}^{-1}}$.
This completes the proof of Step 5 and hence of the proposition.
\end{proof}

We will write $\Aut_T(\G)$ for the inner automorphisms of $\G$ induced by the elements of $\widehat{T} \leq \Aut_\L(S)$, see Definition \ref{def_inner}.
Observe that conjugation by $t \in T$ induces an Adams automorphism of $S$ of degree $1$ and $c_t$ is a special unstable Adams operation with $\tau_P=t$ for every $P \in \F^c$ and $\tau_\vp=1$ for any $\vp \in \Mor(\L)$.

Suppose that $M$ is a $\ZZ_p$-module and that $x \in M$ has order $p^m$ for some $m$, i.e $p^mx=0$.
Then for any $\zeta \in \ZZ_p^\times$ we claim that $\zeta \cdot x =x$ if and only if $\zeta \in \Ga_m(p)$.
To see this write $\zeta = u + p^m v$ for some $u \in \ZZ \subseteq \ZZ_p$ and $v \in \ZZ_p$.
Then $\zeta \cdot x = u \cdot x$ so $\zeta \cdot x= x$ if and only if $u \cdot x = x$ which happens if and only if $u \in \Ker(\ZZ_p^\times \to \Aut(\ZZ/p^m)) = \Ga_m(p)$.

Proposition \ref{prop_existence_of_special_uAos} below gives a conceptual meaning to the integer $m$ such that for any $\zeta\in\Gamma_m(p)$ there exists an unstable Adams operation of degree $\zeta$ on $\G$, as in \cite{JLL}.
It also implies that if $H^1(\Lred^\R;\Phi)=0$ then special unstable Adams operations relative to $\R$ are determined, modulo $\Aut_T(\G)$, by their degree.

%%%
%%%
%%%
\begin{prop}\label{prop_existence_of_special_uAos}
Let $\G=\SFL$ be a $p$-local compact group and $\R \subseteq \F^c$ be a collection which contains $\H^\bullet(\F^c)$.
Let $\E$ denote the extension $(\L^\R,\Lred^\R,\Phi,[-],\lrb{-})$ and suppose that the order of $[\L^\R]$ in $H^2(\Lred^\R;\Phi)$ is $p^m$ for some $m \geq 0$.
Then there is an exact sequence
\[
H^1(\Lred^\R;\Phi) \to \SpAd(\G;\R)/\Aut_T(\G) \xto{\deg} \Ga_m(p) \to 1.
\]
\end{prop}
\begin{proof}
Suppose that $(\Psi,\psi) \in \SpAd(\G;\R)$ and let $\zeta = \deg(\psi)$.
By Propositions \ref{prop_functors_of_rigid_extensions} and \ref{prop_sfl_rigid_extension}, $\Psi|_\R \in \Aut(\E)$.
By Proposition \ref{prop_special_and_class_of_L}, $\zeta \cdot [\L^\R] = [\L^\R]$.
By the remark  above, $\zeta \in \Ga_m(p)$.
Conversely, if $\zeta \in \Ga_m(p)$ the same remark shows that $\zeta \cdot [\L^\R]=[\L^\R]$ and Proposition \ref{prop_morphisms_of_extensions} implies that there exists $\Theta \in \Aut(\E)$ such that $\overline{\Theta}=\Id_{\Lred^\R}$, and $\eta(\Theta)=\zeta$.
By Proposition \ref{prop_special_adams_operations_and_extens_automs}, $\rho(\Theta)$ is a special unstable Adams operation relative to $\R$ of degree $\zeta$.
This shows that the degree homomorphism is onto $\Gamma_m(p)$.

The kernel of $\deg$ is  the group of special unstable Adams operations relative to $\R$ of degree $1$.
Its preimage under $\rho$ is $\Aut(\E;\Id_{\Lred^\R},\Id_{\Phi})$.
The exactness of the sequence at $\SpAd(\G;\R)/\Inn_T(\G)$ now follows from Proposition \ref{prop_special_adams_operations_and_extens_automs}(\ref{item_ii_existence_of_special_uAos}) and Proposition \ref{prop_h1_and_automorphisms_of_E}.
\end{proof}

We end this section with an analysis of the group $\SpAd(\G;\R)$ as a subgroup of  $\Ad(\G)$. The next two lemmas will be needed.

%%%%%%%%%
\begin{lem}\label{lem_special_are_normal_subgroup}
Let $\G=\SFL$ be a $p$-local finite group and $\R \subseteq \F^c$ a collection.
Then  $\SpAd(\G;\R)$ is a normal subgroup of $\Ad(\G)$.
\end{lem}

% \anote{The proof is very straightforward. Should we leave it to the reader?}
\begin{proof}
Suppose $(\Psi,\psi)$ and $(\Psi',\psi')$ are special unstable Adams operations relative to $\R$. Fix structure elements  $\{\tau_P\}_{P \in \R}$ and $\{\tau_\vp\}_{\vp \in \Mor(\L^\R)}$ for $(\Psi,\psi)$, and $\{\tau'_P\}_{P \in \R}$ and $\{\tau'_\vp\}_{\vp \in \Mor(\L^\R)}$ for $(\Psi',\psi')$.
Set $(\Psi'',\psi'')=(\Psi \circ \Psi', \psi \circ \psi')$,
$\tau''_P=\tau_P \psi(\tau'_P)$ for every $P \in \R$, and $\tau_\vp \psi(\tau'_\vp)$ for any $\vp \in \Mor(\L^\R)$.
Notice that $\psi''(P)=\tau''_p P \tau''_P{}^{-1}$ and that $\tau''_\vp \in \psi''(Q)_0$ for any $\vp \in \L^\R(P,Q)$.
Using the fact that $\Psi(\widehat{g})=\widehat{\psi(g)}$ it is straightforward to check that $(\Psi'',\psi'')$ is a special unstable Adams operation with structure  elements $\{\tau''_P\}_{P\in \R}$ and $\{\tau''_\vp\}_{\vp \in \Mor(\L^\R)}$.
The details are straightforward and left to the reader.
In addition $\Psi^{-1}$ is a special unstable Adams operation with structure elements $\{\psi^{-1}(\tau_P^{-1})\}_{P \in \R}$ and $\{\psi^{-1}(\tau_\vp^{-1})\}_{\vp \in \Mor(\L^\R)}$.
The verification uses the commutativity of $T$ and the fact that $\psi^{-1}(\tau_\vp^{-1}) \in Q_0$ for $\vp \in \L^\R(P,Q)$.
It follows that $\SpAd(\G;\R)$ is a subgroup of $\Ad(\G)$.

Let $(\Psi,\psi)$ be a special unstable Adams operation relative to $\R$. For any $(\Theta,\theta) \in \Ad(\G)$ consider the Adams operation $\Psi' \defeq \Theta^{-1} \circ \Psi \circ \Theta$.
Set $\tau'_P \defeq \theta^{-1}(\tau_{\theta(P)})$ and $\tau_{\vp}' \defeq \theta^{-1}(\tau_{\Theta(\vp)})$.
For any $P \in \L^\R$ and any $\vp \in \L^\R(P,Q)$,
\begin{eqnarray*}
\Psi'(P) &=& 
     \theta^{-1}(\psi(\theta(P)) = 
     \theta^{-1}(\tau_{\theta(P)}  \cdot \theta(P) \cdot  \tau_{\theta(P)}^{-1}) = 
%     c_{\theta^{-1}(\tau_{\theta(P)})}(P) =
     \tau'_P \cdot  P \cdot \tau'_P{}^{-1}, 
\\
\Psi'(\vp) &=& 
     \Theta^{-1}( \widehat{\tau_{\theta(Q)}} \circ \widehat{\tau_{\Theta(\vp)}} \circ \Theta(\vp) \circ \widehat{\tau_{\theta(P)}}^{-1}) =
         \widehat{\tau_Q'} \circ \widehat{\tau'_{\vp}} \circ \vp \circ \widehat{\tau'_P}.
\end{eqnarray*}
Therefore the sets $\{\tau'_P\}_{P \in \R}$ and $\{\tau_{\vp}'\}_{\vp \in \Mor(\L^\R)}$ give $\Psi'$ the structure of a special Adams operation, and so $\SpAd(\G;\R) \nsg \Ad(\G)$.
\end{proof}

\begin{lem}\label{lem_extn_class_L finite_order}
Let $\G=\SFL$ be a $p$-local compact group and let $\R \subseteq \F^c$ be any collection with finitely many $\F$-conjugacy classes.
Then the the order of the class $[\L^\R]\in H^2(\Lred^\R,\Phi)$ is a power of $p$.
\end{lem}

\begin{proof}
The category $\Lred^\R$ can be replaced with a finite subcategory $\M$ such that $H^*(\Lred^\R,\Phi) \cong H^*(\M,\Phi)$.
Since $\M$ is finite the cobar construction has the property that for any $n \geq 0$ the group $C^n(\M,\Phi)$ is a product of finitely many discrete $p$-tori and it is therefore a $p$-torsion group.
\end{proof}

\begin{prop}\label{prop_spad_in_ad_finite_index_solv}
Let $\G=\SFL$ be a $p$-local compact group and $\R \subseteq \F^c$ be a collection which contains finitely many $\F$-conjugacy classes.
Then $\SpAd(\G;\R)$ has finite index in $\Ad(\G)$ and the factor group is solvable of class at most $3$.
\end{prop}

\begin{proof}
We have seen in Lemma \ref{lem_special_are_normal_subgroup} that $\SpAd(\G;\R) \nsg \Ad(\G)$.
By Lemma \ref{lem_extn_class_L finite_order} and Proposition \ref{prop_existence_of_special_uAos} there is a morphism of exact sequences
\[
\xymatrix{
1 \ar[r] &
\SpAd^{\deg=1}(\G;\R) \ar@{^(->}[d] \ar[r] &
\SpAd(\G;\R) \ar@{^(->}[d] \ar[r]^{\deg} &
\Ga_m(p) \ar@{^(->}[d] \ar[r] & 1
\\
1 \ar[r] &
\Ad^{\deg=1}(\G) \ar[r] &
\Ad(\G) \ar[r]_{\deg} &
\ZZ_p^\times.
}
\]
The cokernel of the last column is a finite abelian group.
Thus, by the snake lemma it remains to show that the quotient group in the first column is a solvable finite group of class at most $2$.

Consider the following commutative diagram with exact rows
\[
\xymatrix{
1 \ar[r] &
\SpAd^{\id_S}(\G;\R) \ar[r] \ar@{^(->}[d] &
\SpAd^{\deg=1}(\G;\R) \ar[rr]^{(\Psi,\psi) \mapsto \psi} \ar@{^(->}[d] &&
\Ad^{\deg=1}(S) \ar@{=}[d] \\
1 \ar[r] &
\Ad^{\id_S}(\G) \ar[r] &
\Ad^{\deg=1}(\G) \ar[rr]^{(\Psi,\psi) \mapsto \psi} &&
\Ad^{\deg=1}(S)
}
\]
where the superscript $\id_S$ means operations whose underlying Adams automorphism is the identity on $S$. %that restrict to the identity on $S$. 
Let $U \le V$ be the images of the right horizontal maps.
Notice that $\Aut_T(S)\le U$ because for every $t \in T$, $c_{\widehat{t}} \in \SpAd^{\deg=1}(\G;\R)$.
By \cite[Proposition 2.8]{JLL} $\Ad^{\deg=1}(S)/\Aut_T(S) \cong H^1(S/T;T)$.
Since this is a finite abelian group by Lemma \ref{L:coholology G with T}, it follows that $V/U$ is a finite abelian group. 
It remains to show that the quotient group in the left column in this diagram is a finite abelian group.
This will be done in two step as follows.

\noindent
{\bf Step 1.} ($\Ad^{\id_S}(\G)$ \emph{is an abelian discrete $p$-toral group}).
Recall from \cite[Lemma 3.2]{BLO3} that $\H^\bullet(\F)$ has finitely many $S$-conjugacy classes and hence finitely many $T$-conjugacy classes. 
Let $\Q=\{Q_1,\dots,Q_r\}$ be a set of representatives and for any $P \in \H^\bullet(\F)$ we choose once and for all some $t_P \in T$ such that $t_P P t_P^{-1} \in \Q$.
If $(\Psi,\id_S) \in \Ad(\G)$, then for $\vp \in \L(P,Q)$ we have $\pi(\Psi(\vp))=\pi(\vp)$, so $\Psi(\vp)=\vp \circ \widehat{z(\vp)}$ for a unique $z(\vp) \in Z(P)$.
This gives a function
\[
z \colon \Ad^{\id_S}(\G) \xto{\ \Psi \mapsto (z(\vp)) \ } \prod_{Q,Q' \in \Q} \prod_{\L(Q,Q')} Z(Q).
\]
This is a homomorphism because if $\Psi, \Psi' \in \Ad^{\id_S}(\G)$, then for any $\vp \in \L(P,Q)$
\[
(\Psi' \circ \Psi)(\vp) = 
\Psi'(\vp \circ \widehat{z(\Psi)(\vp)}) =
\Psi'(\vp) \circ \Psi'(\widehat{z(\Psi)(\vp)}) =
\vp \circ \widehat{z(\Psi')(\vp)} \circ \widehat{z(\Psi)(\vp)},
\]
so $z(\Psi' \circ \Psi) = z(\Psi') \cdot z(\Psi)$.

We claim that $z$ is injective.
Choose some $\Psi \in \ker(z)$. Thus $z(\Psi)(\vp)=1$ for all $\vp \in \L^\Q$.
This means that $\Psi$ is the identity on $\L^\Q$.
If $P,P' \in \H^\bullet(\F^c)$, then with the notation above, there are unique $Q,Q' \in \Q$ such that $\L(P,P')=\widehat{t_{P'}}^{-1} \circ \L(Q,Q') \circ \widehat{t_P}$.
Since $\Psi(\widehat{t_P})=\widehat{\id_S(t_P)} = \widehat{t_P}$ and $\Psi(\widehat{t_{P'}})=\widehat{t_{P'}}$, and since $\Psi$ is the identity on $\L(Q,Q')$, it follows that $\Psi$ is the identity on $\L^{\H^\bullet(\F^c)}$.
By the uniqueness part in Lemma \ref{lem _specials_for_large_R} it follows that $\Psi=\Id$.

Since $z$ is injective and the codomain of $z$ is a finite product of abelian discrete $p$-toral groups,  it follows that its image is an abelian discrete $p$-toral group.

\noindent
{\bf Step 2.} (\emph{$\SpAd^{\id_S}(\G;\R)$ has finite index in $\Ad^{\id_S}(\G)$}).
By Step 1, $\Ad^{\id_S}(\G)$ is an abelian discrete $p$-toral. Thus by the structure theorem of abelian discrete $p$-toral groups, it is isomorphic to $D_0 \times A$, where $D_0$ is a discrete $p$-torus and $A$ is a finite abelian $p$-group.
Therefore, $n\cdot \Ad^{\id_S}(\G)$ has finite index in $\Ad^{\id_S}(\G)$ for any $n \geq 1$.
We will show that there exists some $n$ such that $\Psi^n \in \SpAd^{\id_S}(\G;\R)$ for any $\Psi \in \Ad^{\id_S}(\G)$, and this will complete the proof of the statement.

Any $(\Psi,\id_S) \in \Ad(\G)$ induces, by Propositions \ref{prop_functors_of_rigid_extensions} and \ref{prop_sfl_rigid_extension}, a functor $\overline{\Psi} \in \Aut(\Lred)$ which is the identity on objects and therefore induces a permutation on $\Lred(P,Q)$ for all $P,Q \in \L$.
We therefore obtain a homomorphism
\[
\si \colon \Ad^{ \id_S}(\G) \to \prod_{P,Q \in \L^\R} \Sym(\Lred(P,Q))
\]
where $\Sym(\Omega)$ is the symmetric group of a set $\Omega$.
Since the sets $\Lred(P,Q)$ are finite and since $\R$ has finitely many $T$-conjugacy classes, we may define
\[
r= \max \{ | \Lred(P,Q) |  : P,Q \in \H^\bullet(\F^c)\}.
\]
Then $n=r!$ annihilates any element in the codomain of $\si$, and in particular $\Psi^n \in \Ker(\si)$ for any $\Psi \in \Ad^{\id_S}(\G)$.
It remains to show that $\Ker(\si) \le \SpAd^{\id_S}(\G;\R)$.
Suppose that $(\Psi,\id_S) \in \Ker(\si)$.
Then $\Psi(P)=P$ for any $P \in \R$ and also $[\Psi(\vp)]=[\vp]$ for any $\vp \in \L^\R(P,Q)$, namely $\Psi(\vp)=\widehat{\tau_\vp} \circ \vp$ for a unique $\tau_\vp \in Q_0$.
This shows that $(\Psi,\psi)$ has the structure of a special unstable Adams operation with structure $\tau_P=1$ and the elements $\tau_\vp$ above.
This completes the proof of the claim and the proposition follows.
\end{proof}

\begin{remark}\label{rem:finite index of spad}
The proof gives an explicit, albeit crude, bound for the index of $\SpAd(\G;\R)$ in $\Ad(\G)$.
It is 
\[
|\ZZ_p^\times/\Ga_m(p)| \cdot |H^1(S/T,T)| \cdot r! = |H^1(S/T,T)| \cdot p^m r!
\]
where $r$ as in the proof of Step 2.
\end{remark}

We are now ready to prove Theorem \ref{thm_spad_in_ad}, which we restate  as follows.
 
\begin{thm}\label{thm_spad_in_ad_in_sec}
Let $\G=\SFL$ be a weakly connected $p$-local compact group.
Suppose that a collection $\R \subseteq \F^c$ has finitely many $\F$-conjugacy classes.
Then the following statements hold.
\begin{enumerate}[\rm (i)]
\item If $(\Psi,\psi) \in \SpAd(\G;\R)$ is of degree $\zeta$ then $\zeta \cdot [\L^\R]=[\L^\R]$ in $H^2(\Lred^\R,\Phi)$. If in addition $\R \supseteq \H^\bullet(\F^c)$, see \cite[Sec. 3]{BLO3}, then there is an exact sequence 
\[
H^1(\Lred^\R,\Phi) \to \SpAd(\G;\R)/\Inn_T(\G) \xto{\deg} \Gamma_m(p) \to 1
\]
where $p^m$ is the order of $[\L^\R]$ in $H^2(\Lred^\R,\Phi)$. 
In particular, $\SpAd(\G;\R)$ has a normal Sylow $p$-subgroup with $\Inn_T(\G)$ as its maximal discrete $p$-torus.
\label{thm_spad_in_ad_in_sec_ii}

\item $\SpAd(\G;\R)$ is a normal subgroup of $\Ad(\G)$ of finite index.
The quotient group is solvable of class at most $3$. \label{thm_spad_in_ad_in_sec_i}
\end{enumerate}
\end{thm}
\begin{proof}
Part (\ref{thm_spad_in_ad_in_sec_ii}) is the contents of Propositions \ref{prop_special_and_class_of_L} and \ref{prop_existence_of_special_uAos}.
Part (\ref{thm_spad_in_ad_in_sec_i}) is  Lemma 
\ref{lem_special_are_normal_subgroup}
 and  Proposition \ref{prop_spad_in_ad_finite_index_solv}. The last statement follows at once from Part (\ref{thm_spad_in_ad_in_sec_ii}).
 \end{proof}

%%%%%%%%%%%%%%%%%%%%%%%%%%%%%%
%%%%%%%%%%%%%%%%%%%%%%%%%%%%%%
%%%%%%%%%%%%%%%%%%%%%%%%%%%%%%
%%%%%%%%%%%%%%%%%%%%%%%%%%%%%%
%%%%%%%%%%%%%%%%%%%%%%%%%%%%%%
%%%%%%%%%%%%%%%%%%%%%%%%%%%%%%
%%%%%%%%%%%%%%%%%%%%%%%%%%%%%%
%%%%%%%%%%%%%%%%%%%%%%%%%%%%%%
%%%%%%%%%%%%%%%%%%%%%%%%%%%%%%
%%%%%%%%%%%%%%%%%%%%%%%%%%%%%%
%%%%%%%%%%%%%%%%%%%%%%%%%%%%%%
%%%%%%%%%%%%%%%%%%%%%%%%%%%%%%
\section{Not all Adams operations are special}
\label{sec_not_all_special}

In this section we find examples of weakly connected $p$-local compact groups $\G$ that afford unstable Adams operations which are not special relative to the collection $\R$ of the $\F$-centric $\F$-radical subgroups.
The idea is to find such $\G$ which fulfils the conditions of Lemma \ref{lem_psi_not_special} below.

The collection $\R$ is a very natural one to look at since $|\L^\R| \to |\L|$ is a mod-$p$ equivalence.
To see this observe that $\R \subseteq \H^\bullet(\F^c)$ by \cite[Corollary 3.5]{BLO3} and that $|\L^{\H^\bullet(\F^c)}| \to |\L|$ is a homotopy equivalence by \cite[Proposition 1.12]{JLL} which provides a natural transformation from the identity on $\L$ to the functor $\L \xto{P \mapsto P^\bullet} \L$.
Since $\H^\bullet(\F^c)$ has finitely many $\F$-conjugacy classes by \cite[Lemma 3.2]{BLO3}, there results a finite filtration of $\L^{\H^\bullet(\F^c)}$ which, together with the $\La$-functors machinery in \cite[Section 5]{BLO3}, can be used to prove that $|\L^\R| \to |\L^{\H^\bullet(\F^c)}|$ induces an isomorphism in $H^*(-,\ZZ_{(p)})$.

%%%%%
\begin{lem}\label{lem_psi_not_special}
Let $\G=\SFL$ be a $p$-local compact group and $\R$ be a collection of $\F$-centric subgroups.
Let $(\Psi,\psi) \in \Ad(\G)$ and assume that
\begin{enumerate}[(i)]
\item
$\deg(\psi) \in \ZZ_p^\times  \setminus \Gamma_1(p)$ and that
\label{cond_1_psi_not_special}

\item
there exists $P \in \R$ such that $P$ is not a semi-direct product of $P_0$ with $P/P_0$.
\label{cond_2_psi_not_special}
\end{enumerate}
Then $(\Psi,\psi) \notin \SpAd(\G;\R)$.
\end{lem}

\begin{proof}
Assume by contradiction that $(\Psi,\psi) \in \SpAd(\G;\R)$ and set $\zeta=\deg(\psi)$.
We claim that $[\L^\R]$ is the trivial element in $H^2(\Lred^\R,\Phi)$.
To see this, let $L$ be the $\ZZ_p$-submodule of $H^2(\Lred^\R,\Phi)$ generated by $[\L^\R]$.
By Proposition \ref{prop_special_and_class_of_L}, $\zeta$ acts as the identity on $L$.
If $L$ is infinite then $L \cong \ZZ_p$ and therefore $\zeta=1\in \Ga_1(p)$ which is a contradiction.
%Since $1 \in \Ga_m(p)$ for all $m$, this contradicts (\ref{cond_1_psi_not_special}).
Therefore $L \cong \ZZ/p^m$ and $\zeta \in \Ga_m(p)$ which by hypothesis (\ref{cond_1_psi_not_special}) implies that $m=0$, namely $L=0$.

Lemma \ref{lem_zero_class_split_extension} now implies that there exists a functor $s \colon \Lred^\R \to \L^\R$ which is a right inverse to the projection $\L^\R \to \Lred^\R$.
Consider any $P \in \R$. % guaranteed by hypothesis (\ref{cond_2_psi_not_special}).
Notice that $\Aut_\L(P)$ contains $\widehat{P}$ as a copy of $P$ and similarly $\Aut_{\Lred}(P)$ contains a copy of $P/P_0$.
The functor $s$ gives a section $s \colon P/P_0 \to P$ for the projection $P \to P/P_0$.
This is a  contradiction to hypothesis (\ref{cond_2_psi_not_special}).
\end{proof}

It turns out that compact Lie groups provide examples of $p$-local compact group $\G$ which satisfy the conditions of the lemma.
First, let us recall from \cite[Section 9]{BLO3} how compact Lie groups give rise to $p$-local compact groups.

The poset of all discrete $p$-toral subgroups of a compact Lie group $G$ contains a maximal element $S$.
Every discrete $p$-toral $P \leq G$ is conjugate to a subgroup of $S$ and in particular all maximal discrete $p$-toral subgroups of $G$ are conjugate.
The fusion system $\F=\F_S(G)$ over $S$ has by definition $\Hom_\F(P,Q)=\Hom_G(P,Q)$ namely the homomorphisms $P \to Q$ induced by conjugation by elements of $g$.
This fusion system is saturated by \cite[Lemma 9.5]{BLO3}, it admits an associated
centric linking system $\L=\L_S^c(G)$ that is unique up to isomorphism, and 
$\pcomp{|\L_S^c(G)|} \simeq \pcomp{BG}$ by \cite[Theorem 9.10]{BLO3}.

Recall that a closed subgroup $Q \leq G$ is called $p$-toral if it is an extension of a torus by a finite $p$-group.
Let $P$ be a discrete $p$-toral subgroup of $G$.
Then $\overline{P}$ (the closure of $P$) is a a $p$-toral subgroup of $G$.
In fact, $\overline{P_0}$ is the maximal torus of $\overline{P}$.

A discrete $p$-toral $P \leq G$ is called \emph{snugly embedded} if $\overline{P} = \overline{P_0} \cdot P$ and every $p$-power torsion element in $(\overline{P})_0$ belongs to $P_0$.

\begin{lem}\label{lem_closure_semidirect}
Let $G$ be a compact Lie group and $P \leq G$ a snugly embedded discrete $p$-toral subgroup.
If $P$ is the semidirect product of $P_0$ with $P/P_0$ then $\overline{P}$ is the semidirect product of $\overline{P}_0$ with $P/P_0$.
\end{lem}

\begin{proof}
Suppose $\pi \leq P$ is a complement of $P_0$.
Since $P$ is snugly embedded, $\overline{P}=\overline{P_0} \cdot P = (\overline{P})_0 \cdot P = (\overline{P})_0 \cdot \pi$.
Also, $P_0$ contains all $p$-torsion in $(\overline{P})_0$ so $(\overline{P})_0 \cap \pi \leq P_0 \cap \pi = 1$.
\end{proof}

Fix a compact Lie group $G$.
A subgroup $P$ is called \emph{$p$-stubborn}, if $P$ is $p$-toral and if $N_G(P)/P$ is finite and $O_p(N_G(P)/P)=1$, where $O_p(K)$ denotes the largest normal $p$-subgroup of a finite group $K$ \cite{JMO1}.
Let us now recall from \cite{Ol74} the structure of the $p$-stubborn subgroups of the classical groups $U(n)$ and $SU(n)$ when $p>2$.

First, consider the regular representation of $\ZZ/p$ on $\CC^p$ with the standard basis $e_0,\dots,e_{p-1}$.
This representation sends a generator of $\ZZ/p$ to the permutation matrix in $U(p)$:
\[
B=
\begin{pmatrix}
0 & 1 & 0 & \dots & 0 \\
0 & 0 & 1 & \dots & 0 \\
\dots & & \dots & & \dots \\
0 & 0 & 0 & \dots & 1  \\
1 & 0 & 0 & \dots & 0 
\end{pmatrix}
\qquad B \colon e_i \mapsto e_{i-1}
\]
Since the regular representation contains one copy of every irreducible representation of $\ZZ/p$, it is clear that $B$ is conjugate in $U(p)$ to the matrix
\[
A=\diag(1,\zeta,\zeta^2, \dots,\zeta^{p-1})=
\begin{pmatrix}
1 & 0 & 0 & \dots & 0 \\
0 & \zeta & 0 \dots & 0 \\
0 & 0 & \zeta^2 & \dots & 0 \\
\dots & & & & \dots \\
0 & 0 & 0 & \dots & \zeta^{p-1}
\end{pmatrix}
\qquad
A \colon e_i \mapsto \zeta^i e_i
\]
where $\zeta$ is a $p$-th root of unity.
It is easy to check that
\[
[A,B]=ABA^{-1}B^{-1} = \zeta I_{p}.
\]
Thus, $A$ and $[A,B]$ belong to the standard maximal torus of $U(p)$ (namely the unitary diagonal matrices).

Consider the natural action of $U(p^k)$ on $\CC^{p^k} \cong \CC^p \otimes \dots \otimes \CC^p$.
For every $i=0,\dots,k-1$ we let $A_i$ and $B_i$ denote the matrices that correspond to the action of $A$ and $B$ on the $i$-th factor of the tensor product and the identity on the other factor.
From this description it is clear that
\begin{equation}\label{sec_not_all_special_eq1}
[A_i,A_j]=[B_i,B_j]=I, \qquad [A_i,B_j]=I \ (i \neq j), \qquad [A_i,B_i]=\zeta I.
\end{equation}
In matrix notation, $A_i=I^{\otimes i-1} \otimes A \otimes I^{\otimes k-i}$  and $B_i=I^{\otimes i-1} \otimes B \otimes I^{\otimes k-i}$ where $I$ denotes the $p\times p$ identity matrix and we use the Kronecker tensor product of matrices.

For any $k \geq 0$  define $\Gamma^U_{p^k} \leq U(p^k)$ by
\[
\Gamma^U_{p^k} \defeq \Big\langle A_0,\dots,A_{k-1},\  B_0,\dots,B_{k-1}, \ u\cdot I \ : \ u \in U(1) \Big\rangle,
\]
where $I$ denotes the identity matrix in $U(p^k)$.
It is clear from \eqref{sec_not_all_special_eq1} that the identity component of $\Gamma^U_{p^k}$ is isomorphic to $U(1)$ and that the factor group is isomorphic to 
\begin{equation}\label{sec_not_all_special eq2}
\ZZ/p^{2k} =\langle \overline{A_0},\dots,\overline{A_{k-1}}, \overline{B_0},\dots,\overline{B_{k-1}} \rangle
\end{equation}
where $\overline{A_i}$ and $\overline{B_i}$ are the images of $A_i$ and $B_i$ in the quotient.
Thus, the $\Gamma_{p^k}^U$ are $p$-toral groups.
Notice that 
\[
\Gamma_1^U=U(1) \qquad \text{ and } \qquad 
\Gamma_p^U = \big\langle A,\ B, \ u\cdot I \ : \ u \in U(1) \big\rangle.
\]
Next, fix some $k \geq 1$ and recall that $\Sigma_{p^k} \leq U(p^k)$ via permutation matrices.
Let $E_{p^k}=(\ZZ/p)^k$ act on itself by left translation.
This gives a monomorphism $E_{p^k} \to \Sigma_{p^k}$ and identifies $E_{p^k}$ as a subgroup of $U(p^k)$.
Given any $H \leq U(m)$ the wreath product $H \wr E_{p^k}$ is naturally a subgroup of $U(mp^k)$.

Now fix some $n \geq 1$ and let $p$ be an odd prime.
Write $n=p^{m_1} + \dots + p^{m_r}$.
Identifying the product $U(p^{m_1}) \times \dots \times U(p^{m_r})$ as a subgroup of $U(n)$ in the standard way, consider the subgroups $Q_1 \times  \dots \times Q_r$, where  for each $1\le i\le r$,  $Q_i \leq U(p^{m_i})$ has the form
\[
\Gamma^U_{p^k} \wr E_{q_1} \wr \dots \wr E_{q_t},
\]
and where $t \geq 0$ and $q_1,\dots,q_t$ are all $p$-powers such that $p^{m_i}=p^kq_1\dots q_t$.
By \cite[Theorems 6 and 8]{Ol74}, these groups give a complete set of representatives for the conjugacy classes of $p$-stubborn subgroups of $U(n)$ when $p>2$.
By \cite[Theorem 10]{Ol74} the assignment $P \mapsto P \cap SU(n)$ gives a bijection between the $p$-stubborn subgroups of $U(n)$ and $SU(n)$.

%%%%%%%%%%%%%%
\begin{lem}\label{lem_su2p_cst}
Let $G$ be a connected compact Lie group and $S \leq G$ a maximal discrete $p$-toral group.
Then $\F_S(G)$ is weakly connected, namely $T=S_0$ is self centralising in $S$.
\end{lem}

\begin{proof}
Let $T \leq S$ be the maximal discrete $p$-torus in $S$.
Then $\overline{T}$ is the maximal torus in $\overline{S}$, and hence it is the maximal torus of $G$.
Since $G$ is connected, its maximal torus is self-centralising  and in particular $C_S(T)=C_S(\overline{T})=S \cap C_G(\overline{T})= S \cap \overline{T}=T$, where the last equality holds since $S$ is snugly embedded in $G$.
\end{proof}

Recall that a space $X$ is called $p$-good in the sense of Bousfield and Kan \cite{BK72} if the natural map $X \to \pcomp{X}$ induces isomorphism in $H_*(-;\ZZ/p)$.
Also recall that $H^*_{\QQ_p}(X) \defeq H^*(X;\ZZ_p) \otimes \QQ$.

\begin{lem}\label{L:Qp_cohomology_easy}
Let $X$ be a CW-complex.
If $X$ is $p$-good then $H^*_{\QQ_p}(X) \cong H^*_{\QQ_p}(\pcomp{X})$ via the natural map.
If $H_*(X;\ZZ)$ is finitely generated in every degree then $H^*_{\QQ_p}(X) \cong H^*(X) \otimes \QQ_{p} \cong H^*(X;\QQ) \otimes \ZZ_p$.
\end{lem}

\begin{proof}
If $X$ is $p$-good then $X \to \pcomp{X}$ induces an isomorphism in $H_*(-;\ZZ/p)$ and hence in $H^*(-;\ZZ/p^n)$ for all $n \geq 1$.
Since $\ZZ_p = \varprojlim_n \ZZ/p^n$ a standard spectral sequence argument gives $H^*(X;\ZZ_p) \cong H^*(\pcomp{X};\ZZ_p)$ and hence $H^*_{\QQ_p}(X) \cong H^*_{\QQ_p}(\pcomp{X})$.

Suppose $H_*(X)$ is finitely generated in every degree.
If $A$ is a torsion-free abelian group, \cite[Theorem 10 in Chap. 5.5]{Spanier} implies that $H^*(X;A) \cong H^*(X) \otimes A$.
Therefore \[H^*_{\QQ_p}(X) = H^*(X;\ZZ_p) \otimes \QQ \cong H^*(X) \otimes \ZZ_p \otimes \QQ = H^*(X) \otimes \QQ_p = H^*(X) \otimes \QQ \otimes \ZZ_p = H^*(X;\QQ) \otimes \ZZ_p.\]
\end{proof}

The next proposition is the main result of this section.

%%%%%%%%%%%%%%
\begin{prop}\label{prop_in_psu2p_not_all_special}
Let $p \geq 3$ be a prime and set $G=\PSU(2p)$.
Let $S \leq G$ be a maximal discrete $p$-toral group and let $\G=(S,\F_S(G),\L_S^c(G))$ be the associated $p$-local compact group.
Let $\R$ denote the collection of all centric radical subgroups of $S$.
Then $\SpAd(\G;\R) \lneq \Ad(\G)$.
In fact, the following composition is not surjective
\[
\SpAd(\G;\R) \xto{ \incl } \Ad(\G) \xto{(\Psi,\psi) \mapsto \pcomp{|\Psi|}} \Ad^\geom(\G).
\]
\end{prop}

\begin{proof}
Write $\G=\SFL$ for short. Let $W$ be the Weyl group of $G$.
Since $p > 2$, Dirichlet's theorem on the existence of infinitely many primes in arithmetic progressions implies that there are infinitely many primes  $k$ such that $k \not\equiv 1 \mod p$. In particular there exists such $k$ that, in addition, satisfies $(k,|W|)=1$.
By \cite[Theorem 2]{JMO1} there exist unstable Adams operations $f \colon BG \to BG$ of degree $k$.
That means that $f^* \colon H^{2m}(BG;\QQ) \to H^{2m}(BG;\QQ)$ is multiplication by $k^m$ for every $m>0$.
Lemma \ref{L:Qp_cohomology_easy} and the functoriality of $H^*(-)$ imply that $\pcomp{f}$ induces multiplication by $k^m$ on $H^{2m}(BG;\QQ) \otimes \ZZ_p \cong H^{2m}_{\QQ_p}(BG) \cong H^{2m}_{\QQ_p}(\pcomp{BG})$.
Since $\pcomp{BG} \simeq B\G$, we obtain a self equivalence $h$ of $B\G$ which induces multiplication by $k^m$ on $H^*_{\QQ_p}(B\G)$.
Since $\G$ is weakly connected by Lemma \ref{lem_su2p_cst} it follows from \cite[Proposition 3.2]{BLO7} that $h \in \Ad^\geom(\G)$.
Theorem \ref{thm_geom_vs_alg_uao} applies to show that $h$ is homotopic to $\pcomp{|\Psi|}$ for some $(\Psi,\psi) \in \Ad(\G)$.

Set $\zeta=\deg(\psi)$.
Thus, $\psi|_T$ is multiplication by $\zeta$.
By the commutativity of 
\[
\xymatrix{
BT \ar[r] \ar[d]_{B\psi|_T} & B\G \ar[d]^{\pcomp{|\Psi|}} \\
BT \ar[r] & B\G
}
\]
and since $H^*_{\QQ_p}(B\G) \cong H^*_{\QQ_p}(BT)^{W(\G)}$ by \cite[Theorem A]{BLO7}, we see that  that $\pcomp{|\Psi|}$ induces multiplication by $\zeta^m$ on $H^{2m}_{\QQ_p}(B\G)$ for all $m>0$.
On the other hand, $\pcomp{|\Psi|} \simeq h$, so $\pcomp{|\Psi|}$ induces multiplication by $k^m$ on $H^{2m}_{\QQ_p}(B\G)$ for all $m>0$.

Now, $G=\PSU(2p)$ and since $Z(\SU(2p))$ is a finite group, $BG$ and $BSU(2p)$ are rationally equivalent.
It is well known that 
\[
H^*(BSU(n)) \cong \ZZ[c_2,c_3,\dots,c_n], \qquad c_i \in H^{2i}(BSU(n)).
\]
It follows from Lemma \ref{L:Qp_cohomology_easy} that $H_{\QQ_p}^*(B\G)$ is a polynomial algebra over $\QQ_p$ generated by $c_2,\dots, c_p, \dots,c_{2p}$.
In particular $H^{2p}(B\G)$ is a non-trivial vector space over $\QQ_p$ on which $\pcomp{|\Psi|}$ induces multiplication by $\zeta^p$, and this map is the same as multiplication by $k^p$.
This implies that $\zeta$ and $k$ differ by a $p$-th root of unity in $\QQ_p$ and \cite[Sec. 6.7, Prop. 1,2]{Ro} implies that $\zeta=k$ since $\QQ_p$ contains no $p$-th roots of unity if $p>2$.

In order to complete the proof we  apply Lemma \ref{lem_psi_not_special} to show that $(\Psi,\psi)$ is not special relative to $\R$.
Since $k \not\equiv 1 \mod p$, it follows that $k \notin \Ga_1(p)$, and so condition (\ref{cond_1_psi_not_special}) of Lemma \ref{lem_psi_not_special} holds.
The remainder of the proof is dedicated to showing that condition (\ref{cond_2_psi_not_special}) of the lemma also holds.

Let $\tilde{Q} \leq U(2p)$ be the subgroup
\[
\tilde{Q}=\Gamma^U_p \times \Gamma^U_p
\]
that, by the discussion above, is $p$-stubborn in $U(2p)$.
As we explained above, $\tilde{Q}$ is generated by the matrices 
\begin{eqnarray*}
&& 
A^{(1)} = \begin{pmatrix}
A & 0 \\ 
0 & I
\end{pmatrix},
\qquad
B^{(1)} = \begin{pmatrix}
B & 0 \\ 
0 & I
\end{pmatrix},
\\
&& 
A^{(2)} = \begin{pmatrix}
I & 0 \\ 
0 & A
\end{pmatrix},
\qquad
B^{(2)} = \begin{pmatrix}
I & 0 \\ 
0 & B
\end{pmatrix},
\\
&&
\begin{pmatrix}
uI & 0 \\
0 & vI
\end{pmatrix},
\qquad (u,v \in U(1)).
\end{eqnarray*}
Throughout we will write $I$ for the identity matrix in $U(p)$ and we will also write $\diag(uI,vI)$ for the matrices of the last type, and $\diag(A,I)$ for $A^{(1)}$ etc.
It is clear from \eqref{sec_not_all_special_eq1} that
\begin{eqnarray*}
&&
[A^{(1)},A^{(2)}] = 
[A^{(1)},B^{(2)}] = 
[B^{(1)},B^{(2)}] = 1, 
\\
&& 
[A^{(1)},B^{(1)}] = \diag(\zeta I,I), \qquad [A^{(2)},B^{(2)}] = \diag(I, \zeta I).
\end{eqnarray*}
Also, $A^{(1)},A^{(2)},B^{(1)},B^{(2)}$ commute with all the matrices $\diag(uI,vI)$.
Set
\[
Q=\tilde{Q} \cap SU(2p).
\]
By \cite[Theorem 10]{Ol74}, $Q$ is $p$-stubborn in $\SU(2p)$ and by \cite[Lemma 7]{Ol74} it is also centric in $\SU(2p)$, namely $C_{\SU(2p)}(Q)=Z(Q)$.
Notice that $Q$ contains $A^{(1)}, A^{(2)}, B^{(1)},B^{(2)}$ because $\det(B)=1$ since it is the signature of the odd cycle $(1,2,\dots,p)$.
Therefore $Q$ is generated by $A^{(1)}, A^{(2)}, B^{(1)},B^{(2)}$ and by the subgroup consisting of the matrices
\begin{equation}\label{eq_1_in_psu2p_not_all_special}
\begin{pmatrix}
uI & 0 \\
0 & vI
\end{pmatrix},
\qquad
u,v \in U(1), \ u^p v^p=1.
\end{equation}
This subgroup is easily seen to be isomorphic to $U(1) \times \ZZ/p$.
The maximal torus of $Q$ is therefore the set of matrices
\[
Q_0 = \{ \diag(uI,vI) : v=u^{-1} \} \cong U(1).
\]
Clearly $Q$ contains $Z(SU(2p))$ which is the set of matrices $u \cdot I$ where $u^{2p}=1$.

Let $R$ be the image of $Q$ in $\PSU(2p)$.
It is generated by the images of $A^{(1)}, A^{(2)}, B^{(1)},B^{(2)}$ which we denote by adding ``bars'' - $\bar{A}^{(i)}$ and $\bar{B}^{(i)}$ - and by the images of the matrices in \eqref{eq_1_in_psu2p_not_all_special} which we denote by $\bar{U}$.
The maximal torus of $R$ is the image of $Q_0$, namely these matrices $U$ with $v=u^{-1}$.

By \cite[Proposition 1.6(i)]{JMO1}, $R$ is $p$-stubborn in $G$ and hence, by \cite[Lemma 7(ii)]{Ol74} it is also centric in $G=\PSU(2p)$, namely $C_G(R) = Z(R)$.
We claim that the projection $R \to R/R_0$ does not have a section.
Assume to the contrary that such a section $s \colon R/R_0 \to R$ exists.
Consider the following elements in $R$
\[
X=\bar{A}^{(1)} \bar{B}^{(2)}, \qquad 
Y=\bar{B}^{(1)} \bar{A}^{(2)}.
\]
A straightforward calculation using \eqref{sec_not_all_special_eq1} gives
\[
%\begin{multline*}
[X,Y] =
[\bar{A}^{(1)} \bar{B}^{(2)},\bar{B}^{(1)} \bar{A}^{(2)}] =
[\bar{A}^{(1)}, \bar{B}^{(1)}] \cdot [\bar{B}^{(2)}, \bar{A}^{(2)}] =
\overline{\diag(\zeta I,I)} \cdot \overline{\diag(I,\zeta^{-1}I)} = \overline{\diag(\zeta I, \zeta^{-1} I)}.
%\end{multline*}
\]

Let $x,y$ denote the images of $X,Y$ in $R/R_0$.
These are clearly non-trivial elements, and 
since $s \colon R/R_0 \to R$ is a section, there are $\bar{U},\bar{V} \in R_0$ such that
\[
s(x)=X\cdot \bar{U} \qquad \text{ and } \qquad s(y)=Y \cdot \bar{V}.
\]
Say, $U=\diag(uI,u^{-1}I)$ and $V=\diag(vI,v^{-1}I)$ where $u,v \in U(1)$.
Then $U,V$ commute with the matrices $A^{(i)}$ and $B^{(i)}$, and since $s$ is a homomorphism and  $R/R_0$ is abelian, \eqref{sec_not_all_special eq2}, it follows that
\[
1 = [s(x),s(y)] = [X \bar{U}, Y \bar{V}] = [X,Y] = \overline{\diag(\zeta I,\zeta^{-1}I)}.
\]
This means that $\diag(\zeta I,\zeta^{-1}I) \in Z(\SU(2p))$ so it is diagonal, namely $\zeta=\zeta^{-1}$.
Therefore $\zeta^2=1$ in $\ZZ_p$ which implies that $\zeta=\pm 1$.
However, we have seen that $\zeta = k$, and $k$ was chosen such that $k>0$ and $k \neq 1 \mod p$, which is absurd.
We conclude that $R \to R/R_0$ does not have a section.

Let $P$ be a maximal discrete $p$-toral subgroup of $R$.
Up to conjugation in $G$ we may assume that $P \leq S$.
Since $R$ is centric, it is $p$-centric (namely $Z(R)$ is the maximal $p$-toral subgroup of $C_G(R)$) and  \cite[Lemma 9.6(c)]{BLO3} shows that $P$ is $\F$-centric.
Also, $P$ is snugly embedded in $R$, hence in $G$, and  by \cite[Lemma 9.4]{BLO3} there is an isomorphism of groups
\[
\Out_G(P) = \Rep_G(P,P) \xto{\cong} \Rep_G(R,R) = \Out_G(R).
\]
Since $R$ is centric in $G$, it follows that $\Out_G(R)=N_G(R)/R$ and since $R$ is $p$-stubborn, $\Out_G(R)$ is finite and 
\[
O_p(\Out_\F(P)) =
O_p(\Out_G(P)) \cong
O_p(\Out_G(R)) =
O_p(N_G(R)/R) = 1
\]
Therefore $P$ is $\F$-radical.
We deduce that $P \in \R$.
Since $\overline{P}=R$, Lemma \ref{lem_closure_semidirect} shows that $P$ cannot have a complement for $P_0$ in $P$.
This shows that $P$ satisfies the condition (\ref{cond_2_psi_not_special}) in Lemma \ref{lem_psi_not_special} and finishes the proof of this proposition.
\end{proof}

%%%%%%%%%%%%%%%%%%%%%%%%%%%%%%
%%%%%%%%%%%%%%%%%%%%%%%%%%%%%%
%%%%%%%%%%%%%%%%%%%%%%%%%%%%%%
%%%%%%%%%%%%%%%%%%%%%%%%%%%%%%
%%%%%%%%%%%%%%%%%%%%%%%%%%%%%%
%%%%%%%%%%%%%%%%%%%%%%%%%%%%%%
%%%%%%%%%%%%%%%%%%%%%%%%%%%%%%
%%%%%%%%%%%%%%%%%%%%%%%%%%%%%%
%%%%%%%%%%%%%%%%%%%%%%%%%%%%%%
%%%%%%%%%%%%%%%%%%%%%%%%%%%%%%
%%%%%%%%%%%%%%%%%%%%%%%%%%%%%%
%%%%%%%%%%%%%%%%%%%%%%%%%%%%%%
\appendix

\section{Extensions of categories - supplementary material}
\label{appendixA}

We collect here some results on extensions of categories that are needed in Section \ref{specials}. Most of this material is well known in one form or other, but not quite in the form we need it in this paper, which is why it is included here.

Fix a small category $\C$.
An \emph{$n$-chain} in $\C$ is a sequence $X_0 \xto{c_0} X_1 \xto{c_1} \dots \xto{c_{n-1}} X_n$ of composable morphisms.
We write $\C_n$ for the set of $n$-chains.
Now consider a functor $\Phi \colon \C \to \Ab$.
Recall that the cobar construction is the cochain complex $C^*(\C,\Phi)$ where 
\[
C^n(\C,\Phi) = \prod_{X_0 \xto{c_1} \dots \xto{c_{n-1}} X_n} \Phi(X_n).
\]
We view it as a set of functions $u \colon \C_n \to \coprod_X \Phi(X)$ such that $u(X_0 \xto{c_0} \dots \xto{c_{n-1}} X_n) \in \Phi(X_n)$.
The differential $\de \colon C^n(\C,\Phi) \to C^{n+1}(\C,\Phi)$ is defined on the factor $X_0 \xto{c_0} \dots \xto{c_n} X_{n+1}$ of the target by
\[
\de(u)(X_\bullet) = \sum_{j=0}^n (-1)^j u(\de_i(X_\bullet))+(-1)^{n+1}\Phi(c_n)(\de_{n+1}(X_\bullet))
\]
where $\de_i(X_\bullet)$ is the $n$-chain obtained by deleting $X_i$ from $X_\bullet$.

The cohomology groups of $C^*(\C,\Phi)$ are isomorphic to $\varprojlim^* \Phi$ \cite[Appendix II, Section 3]{GZ}.
The following facts are elementary and are left to the reader.

\begin{lem}\label{lem_simple_facts_on_cobar}
Let $\C^*(C,\Phi)$ be the cochain complex defined above. Then the following hold.
\begin{itemize}
\item[(a)]
Any $1$-cocycle  $z \in C^1(\C,\Phi)$ satisfies $z(1_X)=1$ for any $X \in \C$.

\item[(b)]
A $2$-cocycle $z$ is called \emph{regular} if $z(1_{X_1},c)=1=z(c,1_{X_0})$ for any $c \in \C(X_0,X_1)$.
Every $2$-cocycle $z' \in C^2(\C,\Phi)$ is cohomologous to a regular $2$-cocycle $z$.
\end{itemize}
\end{lem}

\begin{defn}\label{def_section_D_to_C}
Let $\E=(\D,\C,\Phi,\pi,\de)$ be an extension (Definition \ref{def:extension-C-by-Phi}).
A \emph{section} is a function $\si \colon \Mor(\C) \to \Mor(\D)$ such that $[\si(c)]=c$ for every $c \in \Mor(\C)$.
We say that a section $\si$ is \emph{regular} if $\si(1_X)=1_X$ for every $X \in \C$.
\end{defn}

The following definition is an analogue of the well known construction of the $2$-cocycles associated with extensions of groups.
Compare with \cite{Hoff}.

%%%
%%%
%%%
\begin{defn}\label{def_z-sigma}
Let $\E=(\D,\C,\Phi,\pi,\de)$ be an extension and $\si \colon \C_1 \to \D_1$ a regular section.
Define a $2$-cochain as follows.
Given a $2$-chain $X_0 \xto{c_0} X_1 \xto{c_1} X_2$ notice that $[\si(c_1) \circ \si(c_0)]=c_1 \circ c_0 =[\si(c_1 \circ c_0)]$ and therefore there exists a unique element $z_\si(c_1,c_0)$ in $\Phi(X_2)$ such that
\[
\si(c_1) \circ \si(c_0) = \lrb{z_\si(c_1,c_0)} \circ \si(c_1 \circ c_0).
\]
\end{defn}

The next lemma is analogous the the well known result about the $2$-cocycles associated to a given extension of groups.

%%%
%%%
%%%
\begin{lem}\label{lem_2_cocycle_of_extension}
The $2$-cochain $z_\si$ defined above is a regular $2$-cocycle.
Moreover, if $\si'$ is another regular section then $z_{\si'}$ and $z_\si$ are cohomologous.
\end{lem}
\begin{proof}
Let $c \in \C(X_0,X_1)$ be a morphism.
By the definition of $z_\si$ we have $\lrb{z_\si(1_{X_1},c)} \circ \si(1_{x_1} \circ c) = \si(1_{X_1}) \circ \si(c) = \si(c)$ and since $\Phi(X_1)$ acts freely on $\D(X_0,X_1)$ it follows that $z_\si(1_{X_1},c)=1$.
Similarly $z_\si(c,1_{X_0})=1$.

Now we show that $z_\si$ is a $2$-cocycle.
We need to show that $\de(z_\si)(c_2,c_1,c_0)=1$ for any $3$-chain $X_0 \xto{c_0} X_1 \xto{c_1} X_2 \xto{c_2} X_3$.
Consider the following diagram in $\D$
\[
\xymatrix{
 & &
X_0 \ar@{=}[rrrr] \ar[ddll]_{\si(c_2 \circ c_1 \circ c_0)} \ar[d]^{\si(c_1 \circ c_0)}
& & & &
X_0 \ar[d]_{\si(c_0)} \ar[ddrr]^{\si(c_2 \circ c_1 \circ c_0)}
\\
& &
X_2 \ar[rr]_{\lrb{z_\si(c_1,c_0)}} \ar[d]_{\si(c_2)} & &
X_2 \ar[d]^{\si(c_2)} & &
X_1 \ar[ll]^{\si(c_1)} \ar[d]^{\si(c_2 \circ c_1)}
\\
X_3 \ar[rr]_{\lrb{z_\si(c_2,c_1\circ c_0)}} & &
X_3 \ar[rr]_{\lrb{{c_2}_*(z_\si(c_1,c_0))}} & &
X_3 & &
X_3 \ar[ll]^{\lrb{z_\si(c_2,c_1)}} & &
X_3 \ar[ll]^{\lrb{z_\si(c_2 \circ c_1,c_0)}}
}
\]
The rectangle, the bottom-right square and the two triangles commute by the definition of 
$z_\si$.
The bottom-left square commutes by \eqref{compat}.
The diagram is therefore commutative.
Since $\Phi(X_3)$ acts freely on $\D(X_0,X_3)$ it follows that 
\[\lrb{{c_2}_*(z_\si(c_1,c_0))}\circ\lrb{z_\si(c_2,c_1\circ c_0)} = 
\lrb{z_\si(c_2,c_1)}\circ\lrb{z_\si(c_2 \circ c_1,c_0)}\]
because the compositions of both sides with $\si(c_2 \circ c_1 \circ c_0)$ give the same morphism in $\D(X_0,X_3)$.
This, in turn, is equivalent to the $2$-cocycle condition for $z_\si$.

Now suppose that $\si'$ is another regular section.
Since $[\si(c)]=c=[\si'(c)]$ for any $c \in \C(X_0,X_1)$, there exists a unique $u(c) \in \Phi(X_1)$ such that $\si'(c) = \lrb{u(c)} \circ \si(c)$.
This gives a $1$-cochain $u \in C^1(\C,\Phi)$.
By the defining relation of $z_\si$ and $z_{\si'}$ we get for any $2$-cochain $X_0 \xto{c_0} X_1 \xto{c_1} X_2$ in $\C$
\begin{multline*}
\lrb{z_{\si'}(c_1,c_0)} \circ \lrb{u(c_1 \circ c_0)} \circ \si(c_1 \circ c_0) \overset{(u)}{=}
\lrb{z_{\si'}(c_1,c_0)} \circ \si'(c_1 \circ c_0) \overset{(z_{\si'})}{=}
\si'(c_1) \circ \si'(c_0) \overset{(u)}{=} \\
\lrb{u(c_1)} \circ \si(c_1) \circ \lrb{u(c_0)} \circ \si(c_0) \overset{\eqref{compat}}{=}
\lrb{u(c_1)} \circ \lrb{{c_1}_*(u(c_0))} \circ \si(c_1) \circ \si(c_0) \overset{(z_\si)}{=}
\\
\lrb{u(c_1)} \circ \lrb{{c_1}_*(u(c_0))} \circ \lrb{z_\si(c_1,c_0)} \circ \si(c_1 \circ c_0).
\end{multline*}
Since $\Phi(X_2)$ acts freely on $\D(X_0,X_2)$ and since $\lrb{-} \colon \Phi(X_2) \to \Aut_\D(X_2)$ is injective, we deduce that $z_{\si'}(c_1,c_0)=z_\si(c_1,c_0) \cdot u(c_1) \cdot u(c_1 \circ c_0)^{-1} \cdot {c_1}_*(u(c_0))$.
Since this holds for all $2$-cochains, $z_{\si'}=z_\si \cdot \de(u)$, namely $z_\si$ and $z_{\si'}$ are cohomologous.
\end{proof}

Lemma \ref{lem_2_cocycle_of_extension} justifies the following definition.

\begin{defn}\label{def_cohomology_class_of_extn}
Let $\E=(\D,\C,\Phi,\pi,\de)$ be an extension.
Let $[\D]$ denote the element of $H^2(\C,\Phi)$ defined by the $2$-cocycle $z_\si$ associated with a section $\si \colon \C_1 \to \D_1$.
\end{defn}

Here is a simple consequence of the definitions analogous to the statement that an extension of groups is split if and only if the associated $2$-cohomology class is trivial.

%%%
%%%
%%%
\begin{lem}\label{lem_zero_class_split_extension}
Let $\E=(\D,\C,\Phi,[-],\lrb{-})$ be an extension.
Then $[\D]=0$ if and only if there exists a functor $s \colon \C \to \D$ which is a right inverse to $\D \xto{[-]} \C$.
\end{lem}

\begin{proof}
Suppose first that $s$ exists.
Then it provides a regular section $s \colon \C_1 \to \D_1$ and the functoriality of $s$ readily implies that $z_s=1$, hence $[\D]=0$.

Conversely, suppose that $[\D]=0$.
Choose a regular section $\si \colon \C_1 \to \D_1$.
Then $z_\si = \de(u)$ where $\de$ is the differential in the cobar construction and $u \in C^1(\C,\Phi)$.
Thus, given a $2$-chain $C_0 \xto{c_1} C_1 \xto{c_2} C_2$ in $\C$,
\begin{equation}\label{lem_zero_class_split_extension eq 1}
z_\si(c_2,c_1) = u(c_2) \cdot u(c_2 \circ c_1)^{-1} \cdot \Phi(c_2)(u(c_1)).
\end{equation}
Define a functor $s \colon \C \to \D$ as follows.
On objects $s(C)=C$ for all $C \in \C$.
On morphisms $c \in \C(C_1,C_2)$,
\[
s(c) = \lrb{u(c)} \circ \si(c).
\]
Then $s$ is a functor because it respects units and compositions.
First, since $z_\si$ is regular $1=z_\si(1_X,1_X) = u(1_X)\cdot u(1_X \circ 1_X)^{-1} \cdot (1_X)_*(u(1_X)) = u(1_X)$, and therefore $u(1_X)=1$.
Hence $s(1_X)=1_X$ because $\si$ is regular.
Next, $s$ respects composition because
\begin{multline*}
s(c_2 \circ c_1) =
\lrb{u(c_2\circ c_1)} \circ \si(c_2 \circ c_1) = 
\lrb{u(c_2\circ c_1)} \circ \lrb{z_\si(c_2,c_1)} \circ \si(c_2) \circ \si(c_1) \overset{\eqref{lem_zero_class_split_extension eq 1}}{=}
\\
\lrb{u(c_2)} \circ \lrb{\Phi(c_2)(u(c_1))} \circ \si(c_2) \circ \si(c_1) \overset{\eqref{compat}}{=}
\lrb{u(c_2)} \circ \si(c_2) \circ \lrb{u(c_1)} \circ \si(c_1) = 
s(c_2) \circ s(c_1).
\end{multline*}
Clearly $s$ is a right inverse to $\D \xto{[-]} \C$.
\end{proof}

\begin{remark}
Given a functor $\Phi \colon \C \to \Ab$, Thomason's construction \cite{Thomason} $\Tr_\C(\Phi)$ gives rise to an extension $\D$.
Inspection of this construction shows that is comes equipped with a section, namely a functor $s \colon \C \to \D$ which is a right inverse to the projection $\D \xto{[-]} \C$.
It is not hard to see that the extension class $[\Tr_\C(\Phi)]$ is equal to 0, and conversely, if $\D$ is an extension with $[\D]=0$ then $\D$ is isomorphic as an extension to $\Tr_\C(\Phi)$.
\end{remark}

The next definition is analogous to that of an equivalence of extensions of groups.

\begin{defn}\label{def_equivalence_of_extensions}
Let $\C$ be a small category and $\Phi \colon \C \to \Ab$ a functor.
Two extensions $\E=(\D,\C,\Phi,\pi,\de)$ and $\E'=(\D',\C,\Phi,\pi',\de)$ are called \emph{equivalent} if there exists an isomorphism $\Psi \colon \E \to \E'$ such that $\overline{\Psi}=\Id_\C$, and the natural transformation $\eta(\Psi)\colon \Phi \to \Phi \circ \overline{\Psi}=\Phi$ of Lemma \ref{lem:def-Psi-star} is the identity transformation.

This defines an equivalence relation on the class of all extensions of $\C$ by $\Phi$.
The equivalence class of an extension $\E$ is denoted $\{ \E\}$.
Let $\Ext(\C,\Phi)$ denote the collection of equivalence classes of these extensions.
\end{defn}

The next lemma is a special case of the results in \cite{Hoff}.

\begin{lem}\label{lem_h2_classifies_extensions}
Fix a small category $\C$ and a functor $\Phi \colon \C \to \Ab$.
Then there exists a one-to-one correspondence
\[
\Gamma \ \colon \ \Ext(\C,\Phi) \xto{ \ \{ \E \} \mapsto [\E] \ } H^2(\C,\Phi).
\]
\end{lem}
\begin{proof}
To see that $\Gamma$ is well defined we need to show that if $\E$ and $\E'$ are equivalent extensions of $\C$ by $\Phi$ then they give rise to the same element of $H^2(\C,\Phi)$.
Fix regular sections $\si \colon \C_1 \to \D_1$ and $\si' \colon \C_1 \to \D'_1$ and let $z_\si$ and $z_{\si'}$ be their associated $2$-cocycles.
Also fix an equivalence $\Psi \colon \E \to \E'$, i.e. $\overline{\Psi}=\Id_\C$ and $\eta(\Psi)=\Id_{\Phi}$.
Notice that $\Psi(\si) \colon \C_1 \to \D'_1$ is a regular section for $\E'$ because $\overline{\Psi}=\Id_\C$.
Since $\eta(\Psi)=\Id_{\Phi}$, by applying $\Psi$ to the defining relation of $z_\si$ and using \eqref{eta-def}  we get
\[
\lrb{z_\si(c_1,c_0)} \circ \Psi(\si(c_1 \circ c_0)) = \Psi(\si(c_1)) \circ \Psi(\si(c_0)).
\]
Therefore $z_\si=z_{\Psi(\si)}$ which is cohomologous to $z_{\si'}$ by Lemma \ref{lem_2_cocycle_of_extension}.

Now we construct $\Sigma \colon H^2(\C,\Phi) \to \Ext(\C,\Phi)$.
Given $\zeta \in H^2(\C,\Phi)$ choose a regular $2$-cocycle $z \in \zeta$ (this is possible by Lemma \ref{lem_simple_facts_on_cobar}).
Define a category $\D_z$ as follows.
First, $\Obj(\D_z)=\Obj(C)$ and $\D_z(X,Y) = \Phi(Y) \times \C(X,Y)$.
Composition of morphisms $(g_0,c_0) \in \D_z(X_0,X_1)$ and $(g_1,c_1) \in \D_z(X_1,X_2)$ is given by
\[
(g_1,c_1) \circ (g_0,c_0) = (g_1 \cdot \Phi(c_1)(g_0)\cdot z(c_1,c_0), c_1 \circ c_0).
\]
It is a standard calculation to show that composition defined in this way is unital and associative, hence making $D_z$ a small category.
The functor $\pi_z \colon \D_z \to \C$ is the identity on objects and the obvious projection on morphisms.
The assignment $g \mapsto (g,1_X)$ gives the maps $\de_X \colon \Phi(X) \to \Aut_{\D_z}(X)$.
It is easy to check that $\D_z$ is an extension of $\C$ by $\Phi$.

Suppose that $z' \in \zeta$ is another regular $2$-cocycle and let $D_{z'}$ be the associated extension.
We claim that $\D_z$ and $\D_{z'}$ are equivalent extensions.
Since $z$ and $z'$ are cohomologous there is $u \in C^1(\C,\Phi)$ such that $z=z' \cdot \de(u)$.
Regularity of $z$ and $z'$ implies that $u(1_X)=1$ for all $X \in \C$ (by looking at the $2$-chain $X \xto{\id} X \xto{\id} X$).
Define $\Psi \colon \D_z \to \D_{z'}$ as the identity on objects, and on morphisms
\[
\Psi \colon (g,c) \mapsto (u(c) \cdot g,c).
\]
Then $\Psi$ respects identity morphisms since $u(1_X)=1$.
One easily checks it respects compositions because $z(c_1,c_0) =z'(c_1,c_0) \cdot u(c_1) \cdot u(c_1 \circ c_0)^{-1} \cdot {c_1}_*(u(c_0))$.
Finally, $\overline{\Psi}=\Id_\C$ as is evident from the definitions and $\eta(\Psi)=\Id$.
Therefore $\D_z$ and $\D_{z'}$ define equivalent extensions.
This shows that a map $\Sigma \colon H^2(\C,\Phi) \to \Ext(\C,\Phi)$ is well defined.
Notice that $\si \colon \C_1 \to (\D_z)_1$ defined by $c \mapsto (1,c)$ where $c \in \C(X_0,X_1)$ and $1$ denotes the identity element of $\Phi(X_1)$ gives a regular section such that $z_\si=z$.
This shows that $\Gamma \circ \Sigma = \Id$.
If $\E=(\D,\C,\Phi,\pi,\de)$ is an extension and $\si \colon \C_1 \to \D_1$ is a section then there is an equivalence $\Psi \colon \D_{z_\si} \to \D$ defined as the identity on objects and $\Psi \colon (g,c) \mapsto \lrb{g} \circ \si(c)$ on morphisms.
This shows that $\Sigma \circ \Gamma=\Id$.
\end{proof}

Now we deal with constructing morphisms of extensions. The following proposition is a restatement of Proposition \ref{prop_morphisms_of_extensions}. 
It should be compared with \cite[Lemma 2.2(i)]{JLL}.

\begin{prop}\label{prop_morphisms_of_extensions_apndx}
Let $\E=(\D,\C,\Phi,\pi,\de)$ and $\E'=(\D',\C',\Phi',\pi',\de')$ be extensions. 
Let $\psi \colon \C \to \C'$ be a functor, and let $\eta\colon\Phi\to\Phi'\circ\psi$ be a natural transformation. 
Then the following are equivalent.
\begin{enumerate}[(i)]
\item
There exists a morphism of extensions  $\Psi\colon\E \to \E'$ such that $\psi=\overline{\Psi}$ and $\eta = \eta(\Psi)$.
\label{c_i_morphisms_of_extensions_apndx}

\item 
The homomorphisms in cohomology induced by $\psi$ and $\eta$
\[
H^2(\C;\Phi) \xto{\eta_*} H^2(\C;\Phi' \circ \psi) \xleftarrow{\psi^*} H^2(\C',\Phi')
\]
satisfy $\eta_*([\D])=\psi^*([\D'])$.
\label{c_ii_morphisms_of_extensions_apndx}
\end{enumerate}
\end{prop}
\begin{proof}[Proof of Proposition \ref{prop_morphisms_of_extensions}]
Fix regular sections $\si \colon \C_1 \to \D_1$ and $\si' \colon \C'_1 \to \D'_1$ and let $z_\si$ and $z_{z_\si'}$ be the associated $2$-cocycles (\ref{def_z-sigma}).

\noindent
(\ref{c_i_morphisms_of_extensions}) $\implies$ (\ref{c_ii_morphisms_of_extensions}).
Let $\Psi \colon \E \to \E'$ be a morphism of extensions such that $\overline{\Psi}=\psi$ and $\eta(\Psi)=\eta$.
Then for any $c \in \C(X_0,X_1)$ we have $[\Psi(\si(c))]=\psi(c)=[\si'(\psi(c))]$ so $\Psi(\si(c))=\lrb{g} \circ \si'(\psi(c))$ for a unique $g \in \Phi'(\psi(c))$.
Therefore  there exists $u \in C^1(\C,\Phi' \circ \psi)$ such that
\begin{equation}\label{eq_morphisms_of_extensions_def_u}
\Psi(\si(c))=\lrb{u(c)} \circ \si'(\psi(c)), \qquad (c \in \C_1).
\end{equation}
Now, given  a $2$-chain $X_0 \xto{c_0} X_1 \xto{c_1} X_2$ in $\C$ we have
\begin{multline*}
\lrb{\eta_{X_2}(z_\si(c_1,c_0))} \circ \lrb{u(c_1 \circ c_0)} \circ \si'(\psi(c_1\circ c_0)) \overset{\eqref{eq_morphisms_of_extensions_def_u}}{=}
\lrb{\eta_{X_2}(z_\si(c_1,c_0))} \circ \Psi(\si(c_1 \circ c_0)) \overset{\eqref{eta-def}}{=}
\\
\Psi\Big(\lrb{z_\si(c_1,c_0)} \cdot \si(c_1\circ c_0) \Big) \overset{\eqref{def_z-sigma}}{=}
\Psi(\si(c_1)) \circ \Psi(\si(c_0)) \overset{\eqref{eq_morphisms_of_extensions_def_u}}{=}
\\
\lrb{u(c_1)} \circ \si'(\psi(c_1)) \circ \lrb{u(c_0)} \circ \si'(\psi(c_0)) \overset{\eqref{compat}}{=}
\lrb{u(c_1)} \circ \lrb{\Phi'(\psi(c_1))(u(c_0))} \circ \si'(\psi(c_1)) \circ \si'(\psi(c_0)) \overset{(z_{\si'})}{=}
\\
\lrb{u(c_1)} \circ \lrb{\Phi'(\psi(c_1))(u(c_0))} \circ \lrb{z_{\si'}(\psi(c_1),\psi(c_0))} \circ \si'(\psi(c_1 \circ c_0)).
\end{multline*}
Since $\Phi'(\psi(X_2))$ acts freely on $\D'(X_0,X_2)$ we deduce that 
\[
\eta_{X_2}(z_\si(c_1,c_0)) = z_{\si'}(\psi(c_1),\psi(c_0)) \cdot u(c_1) \cdot u(c_1 \circ c_0)^{-1} \cdot \Phi'(\psi(c_1))(u(c_0)).
\] 
This shows that $\eta_*(z_\si)=\psi^*(z_{\si'}) \cdot \de(u)$, hence $\eta_*([\D])=\psi^*([\D'])$.

\noindent
(\ref{c_ii_morphisms_of_extensions}) $\implies$ (\ref{c_i_morphisms_of_extensions}).
Assume that $\eta_*([\D])=\psi^*([\D'])$.
We will construct a morphism $\Psi \colon \E \to \E'$.
By the hypothesis, there exists $u \in C^1(\C,\Psi' \circ \psi)$ such that $\eta_*(z_\si)=\psi^*(z_{\si'}) \cdot \de(u)$.
Notice that since $z_\si$ and $z_{\si'}$ are regular $2$-cocycles, $u(1_X)=1$ for every $X \in \C$ (to see this, evaluate $\eta_*(z_\si)$ and $\psi^*(z_{\si'})$ on the $2$-chain $X \xto{\id} X \xto{\id} X$).

Define $\Psi \colon \D \to \D'$ a follows.
On objects $\Psi \colon X \mapsto \psi(X)$.
Every morphism in $\D(X,Y)$ has the form $\lrb{g} \circ \si(c)$ for unique $c \in \C(X,Y)$ and $g \in \Phi(Y)$.
Define
\begin{equation}\label{morphisms_of_extension_ eq1}
\Psi \colon \lrb{g} \circ \si(c) \mapsto \lrb{\eta_Y(g)} \circ \lrb{u(c)} \circ \si'(\psi(c)).
\end{equation}
Then $\Psi$ respects identities since $u(1_X)=1$ and $\si'(1_X)=1_X$.
It respects composition as one verifies directly for $X_0 \xto{\lrb{g_0} \circ \si(c_0)} X_1 \xto{\lrb{g_1} \circ \si(c_1)} X_2$.
On one hand,
\begin{eqnarray*}
&&\Psi(\lrb{g_1} \circ c_1) \circ \Psi(\lrb{g_0} \circ c_0) =
\lrb{\eta_{X_2}(g_1) \cdot u(c_1)} \circ \si'(\psi(c_1)) \circ \lrb{\eta_{X_1}(g_0) \cdot u(c_0)} \circ \si'(\psi(c_0)) \overset{\eqref{compat}}{=} \\
&& \qquad \lrb{\eta_{X_2}(g_1) \cdot u(c_1) \cdot \Phi'(\psi(c_1))(\eta_{X_1}(g_0)) \cdot \Phi'(\psi(c_1))(u(c_0))} \circ \si'(\psi(c_1)) \circ \si'(\psi(c_0)) \overset{\text{nat. of $\eta$}}{=} \\
&& \qquad \lrb{\eta_{X_2}(g_1) \cdot u(c_1) \cdot \eta_{X_2}(\Phi(c_1)(g_0)) \cdot \Phi'(\psi(c_1))(u(c_0))} \circ \si'(\psi(c_1)) \circ \si'(\psi(c_0)) \overset{(z_{\si'})}{=} \\
&& \qquad \lrb{\eta_{X_2}(g_1) \cdot u(c_1) \cdot \eta_{X_2}(\Phi(c_1)(g_0)) \cdot \Phi'(\psi(c_1))(u(c_0)) \cdot z_{\si'}(\psi(c_1),\psi(c_0))} \circ \si'(\psi(c_1 \circ c_0)) 
\end{eqnarray*}
On the other hand 
\begin{eqnarray*}
&& \Psi(\lrb{g_1} \circ c_1 \circ \lrb{g_0} \circ c_0) \overset{\eqref{compat}, (z_\si)}{=} 
\Psi(\lrb{g_1 \cdot \Phi(c_1)(g_0) \cdot z_\si(c_1,c_0)} \circ \si(c_1 \circ c_0)) \overset{\eqref{eta-def}, \eqref{morphisms_of_extension_ eq1}}{=} \\
&& \qquad
\lrb{\eta_{X_2}(g_1) \cdot \eta_{X_2}(\Phi(c_1)(g_0)) \cdot \eta_{X_2}(z_\si(c_1,c_0))} \circ \lrb{u(c_1 \circ c_0)} \circ  \si'(\psi(c_1 \circ c_0))
\end{eqnarray*}
The right hand sides of the two equations are equal since $\eta_*(z_\si)=\psi^*(z_{\si'}) \cdot \de(u)$.
We have just shown that $\Psi \colon \D \to \D'$ is a functor.
By its construction $\pi' \circ \Psi = \psi \circ \pi$ and $\eta(\Psi)=\eta$ because for any $g \in \Phi(X)$ we have $\Psi(\lrb{g} \circ \si(1_X)) = \lrb{\eta_X(g) \cdot u(1_X)} \circ \si'(\psi(1_X)) =\lrb{\eta_X(g)}$ and use \eqref{eta-def}.
\end{proof}

Let $\Aut(\E;1_\C,1_\Phi)$ denote the group of all automorphisms $\Psi$ of $\E$ such that $\overline{\Psi}=\Id_\C$ and $\eta(\Psi)=\Id_\Phi$.

\begin{prop}\label{prop_h1_and_automorphisms_of_E}
Let $\E=(\D,\C,\Phi,\pi,\de)$ be an extension.
Then there is an isomorphism of groups
\[
\Gamma \colon H^1(\C,\Phi) \to \Aut(\E;1_\C,1_\Phi)/\Inn(\E).
\]
\end{prop}
\begin{proof}
For any $1$-cocycle $z \in C^1(\C;\Phi)$ define $\al_z \colon \D \to \D$ as the identity on objects and for any $d \in \D(X,Y)$
\[
\al_z \colon d \mapsto \lrb{z([d])} \circ d.
\]
Since $z$ is a $1$-cocycle, $z(1_X)=1$ for any $X \in \C$ and therefore $\al_z$ respects identity morphisms.
It respects composition since $z$ is a $1$-cocycle:
\begin{multline*}
\al_z(d_1) \circ \al_z(d_0)=
\lrb{z([d_1])} \circ d_1 \circ \lrb{z([d_0])} \circ d_0 \overset{\eqref{compat}}{=}
\lrb{z([d_1]) \cdot [d_1]_*(z([d_0]))} \circ d_1 \circ d_0 \overset{\text{$1$-cocycle}}{=}
\\
\lrb{z([d_1 \circ d_0])} \circ d_1 \circ d_0 = \al_z(d_1 \circ d_0).
\end{multline*}
Given $u \in C^0(\C,\Phi)$,
\[\al_{\de(u)}(d) = \lrb{\delta(u){[d]}}= 
\lrb{u(Y)}\lrb{d_*(u(X))^{-1}}\circ d \overset{\eqref{compat}}{=} \lrb{u(Y)}\circ d \circ \lrb{u(X))}^{-1}\overset{\eqref{def_InnE}}{=}\tau_u(d)\in \Inn(\E).\]
If $x\in H^1(\C;\Phi)$ is represented by a cocycle $z$ as above, define $\Gamma(x) = [\alpha_z]$, the class of $\alpha_z$ modulo $\Inn(\E)$. The discussion above shows that $\Gamma$ is well defined. 

To see that $\Gamma$ is a homomorphism, let $z, z'\in C^1(\C;\Phi)$ be 1-cocycles, and let $d \in \D(X,Y)$ be a morphism. By definition of $\alpha$ (using multiplicative notation), 
\begin{multline*}\alpha_z(\alpha_{z'}(d)) = \lrb{z([\alpha_{z'}(d)])}\circ \alpha_{z'}(d) = \lrb{z([\lrb{z'([d])}\circ d])}\circ \alpha_{z'}(d)=\lrb{z([d])}\circ \alpha_{z'}(d)=\\
\lrb{z([d])}\circ \lrb{z'([d])}\circ d =\lrb{z([d]) \cdot z'([d])}\circ d= \lrb{(z\cdot z')([d])}\circ d=\alpha_{z\cdot z'}(d).
\end{multline*}
The inverse of $\Gamma$ sends every self equivalence $\Psi$ to the $1$-cochain $z \in C^1(\C,\Phi)$ which is defined by the relation $\Psi(d)=\lrb{z([d]} \circ d$ which follows from the fact that $[\Psi(d)]=\overline{\Psi}([d])=[d]$.
It is left as an exercise to check that $z$ is a $1$-cocycle and that $\al_z=\Psi$.
\end{proof}

%%%%%%%%%%%%%%%%%%%%%%%%%%%%%%
%%%%%%%%%%%%%%%%%%%%%%%%%%%%%%
%%%%%%%%%%%%%%%%%%%%%%%%%%%%%%
%%%%%%%%%%%%%%%%%%%%%%%%%%%%%%
%%%%%%%%%%%%%%%%%%%%%%%%%%%%%%
%%%%%%%%%%%%%%%%%%%%%%%%%%%%%%
%%%%%%%%%%%%%%%%%%%%%%%%%%%%%%
%%%%%%%%%%%%%%%%%%%%%%%%%%%%%%
%%%%%%%%%%%%%%%%%%%%%%%%%%%%%%
%%%%%%%%%%%%%%%%%%%%%%%%%%%%%%
%%%%%%%%%%%%%%%%%%%%%%%%%%%%%%
%%%%%%%%%%%%%%%%%%%%%%%%%%%%%%
%%%%% JLL ARE SPECIAL
%%%%%
\section{The unstable Adams operations in \cite{JLL} are special}
\label{app_jll_special}

Fix a $p$-local compact group $\G=\SFL$ and set $\R=\H^\bullet(\F^c)$.
Let $\P$ be a set of representatives for the $S$-conjugacy classes in $\R$.
By \cite[Lemma 3.2]{BLO3} the set $\P$ is finite.
For any $P,Q \in \P$ let $\M_{P,Q}$ be a set of representatives for the orbits of $N_S(Q)$ on $\L(P,Q)$.
These sets are finite by \cite[Lemma 2.5]{BLO3}.
Also we remark that since every $R \in \R$ is $\F$-centric, $N_S(R)/R_0$ is finite.

In \cite{JLL} we showed that there exists some $m$ such that for any $\ze \in \Ga_m(p)$ we can construct an unstable Adams operation $(\Psi,\psi)$ of degree $\ze$.
This operation has the following properties
\begin{enumerate}
\item $\psi(P)=P$ for any $P \in \P$.
\item If $P \in \P$ then $\psi|_{N_S(P)}$ is an automorphism of $N_S(P)$ which induces the identity on $N_S(P)/P_0$.
\item $\Psi(\vp)=\vp$ for every $\vp \in \M_{P,Q}$ where $P,Q \in \P$.
\end{enumerate}
In the rest of this section we show that any $(\Psi,\psi)$ which satisfies these conditions must be special relative to $\H^\bullet(\F^c)$.
As usual we will write $T$ for the identity component of $S$.

For any $R \in \H^\bullet(\F^c)$ we fix once and for all $g_R \in S$ such that $R=g_R P g_R^{-1}$ where $P \in \P$ (clearly $P$ is unique).
For any $R,R' \in \H^\bullet(\F^c)$ and any $\vp \in \L(R,R')$ we consider $P=g_R^{-1}Rg_R$ and $P'=g_{R'}^{-1}R'g_{R'}$; these are the representatives in $\P$ for the $S$-conjugacy classes of $R$ and $R'$ respectively.
There is a unique $\mu \in \M_{P,P'}$ and a unique $n_\vp \in N_S(P')$ such that
\begin{equation}\label{jll_special_eq1}
\widehat{g_{R'}^{-1}} \circ \vp \circ \widehat{g_R} = \widehat{n_\vp} \circ \mu.
\end{equation}
Since $\psi$ induces the identity on $S/T$, there is a unique $\tau_R \in T$ such that
\[
\psi(g_R) = \tau_R \cdot g_R.
\]
Similarly, with the notation above $\psi$ induces the identity on $N_S(P')/P'_0$ and therefore there exists a unique $t_\vp \in P'_0$ such that
\[
\psi(n_\vp) = t_\vp \cdot n_\vp.
\]
Notice that $g_{R'} \in N_S(P',R')$ and therefore conjugation by $g_R$ carries $P'_0$ onto $R'_0$.
Set
\[
\tau_\vp = g_{R'} \cdot t_\vp \cdot g_{R'}^{-1} \in R'_0.
\]
We now claim that $\tau_R$ and $\tau_\vp$ chosen above render $(\Psi,\psi)$ a special unstable Adams operation relative to $H^\bullet(\F^c)$, see Definition \ref{def_special_uao}.

First, for any $R \in H^\bullet(\F^c)$,
\[
\psi(R)=\psi(g_R P g_R^{-1}) = \psi(g_R) P \psi(g_R)^{-1} = \tau_R g_R P g_R^{-1} \tau_R^{-1} = \tau_R R \tau_R^{-1}.
\]
Next, if $R,R' \in \H^\bullet(\F^c)$ and $\vp \in \L(R,R')$ then \eqref{jll_special_eq1} implies
\begin{multline*}
\Psi(\vp) = \Psi(\widehat{g_{R'}} \circ \widehat{n_\vp} \circ \mu \circ \widehat{g_R^{-1}}) =
\widehat{\psi(g_{R'})} \circ \widehat{\psi(n_\vp)} \circ \Psi(\mu) \circ \widehat{\psi(g_R^{-1})} =
\\
\widehat{\tau_{R'}} \circ \widehat{g_{R'}} \circ \widehat{t_\vp} \circ \widehat{n_\vp} \circ \mu \circ \widehat{g_R^{-1}} \circ \widehat{\tau_R^{-1}} =
\widehat{\tau_{R'}} \circ  \widehat{\tau_\vp} \circ \widehat{g_{R'}} \circ \widehat{n_\vp} \circ \mu \circ \widehat{g_R^{-1}} \circ \widehat{\tau_R^{-1}} =
\widehat{\tau_{R'}} \circ  \widehat{\tau_\vp} \circ \vp \circ \widehat{\tau_R^{-1}} 
\end{multline*}

%===================================================
%    BIBLIOGRAPHY
%===================================================

\end{document}